\documentclass{article}
\usepackage{graphicx} 
\usepackage{subcaption}
\usepackage{amsmath}
\numberwithin{equation}{section}
\allowdisplaybreaks
\usepackage{amsthm}
\usepackage{amssymb}
\usepackage{verbatim}
\usepackage{float}
\usepackage{needspace}
\usepackage[
    colorlinks=true,
    linkcolor=blue,
    citecolor=blue,
    urlcolor=blue
]{hyperref}
\usepackage{tikz}
\usetikzlibrary{positioning,calc}
\usetikzlibrary{decorations.pathreplacing,arrows.meta}
\usepackage{array}
\definecolor{holecolor}{RGB}{250,128,114}
\colorlet{modelgray}{lightgray}
\tikzset{
  paper path/.style={
    line cap=round,line join=round
  },
  paper fused/.style={
    paper path,line width=1.5pt
  },
  paper normal/.style={
    paper fused,draw=modelgray
  },
  paper hole/.style={
    paper fused,draw=holecolor
  },
  paper wire/.style={
    paper path,draw=modelgray,line width=.65pt
  },
  paper label/.style={
    font=\scriptsize,fill=white,inner sep=1pt,text=black
  },
  paper boundary/.style={
    circle,draw=none,fill=black,
    inner sep=0pt,outer sep=0pt,
    minimum size=2.0pt,transform shape=false
  },
}
\tikzset{
  paper crossing/.style={
    circle,fill=black,inner sep=0pt,minimum size=3.8pt
  }
}
\usepackage[a4paper, total={6in, 8in}]{geometry}
\newtheorem{theorem}{Theorem}[section]
\newtheorem{corollary}[theorem]{Corollary}
\newtheorem{lemma}[theorem]{Lemma}
\newtheorem{proposition}[theorem]{Proposition}
\newtheorem{definition}[theorem]{Definition}
\newtheorem{example}[theorem]{Example}
\newtheorem{remark}[theorem]{Remark}
\newtheorem{conjecture}[theorem]{Conjecture}
\title{Stationary Measures and Boundary Correlations for Multispecies Half-Line ASEP}
\author{Alexey Bufetov, Yuan Tian}
\date{\today}

\begin{document}
\maketitle
\begin{abstract}
		We construct stationary measures of the multispecies ASEP on the positive
		half-line, with injection of highest-priority particles at the left
		boundary, using vertex models. Our results extend existing constructions
		of Martin and Aggarwal-Nicoletti-Petrov for the full-line ASEP to the half-line. For the two-species process, we compute the joint distribution
		of occupation colors at the first two sites. We also partially address
		the connection with the half-space ASEP speed process.
	\end{abstract}
\tableofcontents

\section{Introduction}
\label{sec:introduction}

\subsection{Overview}

The asymmetric simple exclusion process (ASEP) is a basic model of
nonequilibrium transport, and at the same time one of the principal
exactly solvable interacting particle systems. Particles attempt to
exchange positions with neighboring holes at unequal rates in the two
directions. The asymmetry creates a nonzero stationary current and,
on large scales, connects exclusion dynamics to the theory of
hyperbolic conservation laws and to the Kardar-Parisi-Zhang (KPZ)
universality class. On the infinite integer line, the single-species ASEP has
simple translation-invariant Bernoulli stationary measures; introducing
an open boundary, however, leads to substantially richer stationary
laws. In particular, exact phase diagrams, density profiles, and
correlation functions of the open ASEP can be accessed by the
matrix-product ansatz \cite{DerridaEHP1993} and its
subsequent refinements. We refer to
\cite{Spitzer1970,Liggett1975,Liggett1985,
DerridaLebowitzSpeer2003,BlytheEvans2007,
UchiyamaSasamotoWadati2004,BrycWesolowski2017,
BrycWangWesolowski2023,BarraquandLeDoussal2023,
WangWesolowskiYang2024} for a non-exhaustive list of key works on
single-species ASEP stationary measures.

In a multispecies exclusion process, every site carries one of several
particle colors or a hole, and neighboring colors exchange according to
a fixed priority order. Merging the higher-priority colors produces an
ordinary one-species ASEP, but the collection of these projections does
not determine the correlations of the full colored stationary measure.
Already with two species, the stationary distributions exhibit
nontrivial clustering and microscopic shock structure
\cite{DerridaJLS1993,FerrariFontesKohayakawa1994,Angel2006}.
On the infinite integer line, stationary measures for arbitrary numbers
of species in the totally asymmetric case were constructed using
multiline diagrams and queues \cite{MR2319708}. The partially
asymmetric case was treated using a matrix-product representation
\cite{ProlhacEvansMallick2009} and a queueing formulation
\cite{MR4089793}; see also the semi-infinite single-species
matrix-product approach in \cite{MR3324896}.

On the integer line, queueing constructions are especially useful for
explicit computations with stationary distributions. They have led,
for instance, to beautiful formulas for the TASEP speed process and multispecies
correlations \cite{AmirAngelValko2011,AyyerLinusson2014,AyyerLinusson2017}.
Almost-sure existence of the full-line ASEP speed process was proved in
\cite{AggarwalCorwinGhosal2023}; related half-line second-class particle
asymptotics have been studied in \cite{Chen2026HalfLine}.
More recently, such constructions were interpreted in the language
of colored stochastic vertex models \cite{MR4968006}, with the
Yang-Baxter equation playing a central role.
Relevant developments in stochastic higher-spin and colored vertex
models include
\cite{MR3477349,MR3782413,MR4518478,
BorodinCorwinGorin2016,BufetovMucciconiPetrov2021}.

A boundary changes the stationary measures of multispecies ASEP in an
essential way. Nevertheless, the vertex-model connection suggests a possibility of
a half-line vertex model construction: at an integrable boundary the
Yang-Baxter equation is supplemented by the reflection equation
\cite{Sklyanin1988}. Related half-space integrable models and half-line ASEP have been
studied in \cite{BarraquandBorodinCorwinWheeler2018,
HeBoundary2024,MR5008991,BarraquandCorwin2024}. Integrable boundary
weights and half-space six-vertex constructions appear in
\cite{MR3959822,MR4924418}; see also the recent half-space TASEP
transition formulas \cite{deGierMeadRemenikWheeler2025} and work
on stationary measures for six-vertex models on a strip
\cite{Yang2024Strip}.

The present paper provides such a vertex construction.
We study the positive half-line ASEP, with particles of the strongest
color injected at an arbitrary rate at the left endpoint. Our first
main result, Theorem~\ref{thm:main-stationary-construction}, is a
vertex-model construction of stationary measures for the multispecies
process. Our second main result, Theorem~\ref{thm:two-site-marginal},
gives an explicit calculation of the joint occupation distribution at
the first two sites of the half-line; obtaining this formula was our
original motivation. As explained below, this distribution is also
related to the joint distribution of speeds in the \textit{half-space ASEP speed process},
conditional on plausible conjectures that we do not prove here.

It should be noted that we study only a special case of multispecies ASEP on the half-line. Our boundary update is totally asymmetric, and only first-class particles can be injected. There are two reasons for the restriction. First, it is this version of ASEP, with the injection of first-class particles only, that appears in the description of speeds in the half-space ASEP speed process. This makes this part of the restriction more natural. Second, 
while it is plausible that the general strategy for constructing stationary measures works for a wider class of boundary interactions, the technical details become more complicated, and we prefer to avoid them.

During our work on this project, we learned of independent related work in progress \cite{AggarwalCorwinHeHegdeWheeler}. 

Let us present the results in more detail.

\subsection{The stationary measure construction}
\label{subsec:intro-stationary-construction}

First, let us define in more detail the version of the multispecies ASEP studied in this paper. Fix the number of species
$r\in\mathbb Z_{\ge1}$. The state space is $\{0,1,\ldots,r\}^{\mathbb Z_{\ge 1}}$,
and the priority order is
$1\succ2\succ\cdots\succ r\succ0$; color $0$ is a hole.
A stronger color (we use the words species, color, type and class interchangeably) exchanges with a weaker color to its right at rate $1$
and to its left at rate $q\in[0,1)$. At site $1$, every color other than
$1$ is replaced by $1$ at rate $(1-q)\alpha$, where $\alpha>0$.

Fix density parameters $\rho_1,\ldots,\rho_r$ such that
\begin{equation}
\label{eq:intro-vertex-parameters}
 0<\rho_1<\cdots<\rho_r<1,\qquad \rho_1>\max(1/2,1-\alpha).
\end{equation}

Our goal is to construct a stationary measure $\mu^{\alpha,\rho_1,\ldots,\rho_r}$ of ASEP, which has asymptotic density $\rho_1$ of first-class (or color 1) particles, asymptotic density $\rho_2$ of color 1 and color 2 particles, ..., asymptotic density $\rho_r$ of colors from $1$ to $r$ particles. Consequently, the density of holes (color 0) is $1-\rho_r$. 

In short, this stationary measure can be read off from the right-hand side of Figure~\ref{fig:intro-vertex-array} below, in which we apply to trivial input from the left and from below a collection of stochastic vertices (also, the limit ''height to infinity'' of this picture has to be taken in order to get half-space), which produces a probability measure on the right. Let us define all necessary ingredients. 

We construct a stationary measure from three types of stochastic
vertices. To distinguish various components of these vertices more easily, we introduce the following terminology. 
A \textit{physical} edge carries a single color (it corresponds eventually to the space of ASEP). A \textit{carrier} edge contains infinitely many colors and
records the numbers of arrows of the relevant colors. It corresponds to queues. 
Write
$\mathcal C_m=\{0,1,\ldots,m\}$ for $0\le m\le r$, and set $x=\max\{0,1-\alpha^{-1}\}$.

\paragraph{Usual vertices.}
Fix $1\le m\le r$. The usual vertex
$\mathbb W^{(m)}(G,i;G',j)$ has incoming carrier $G$ and physical
color $i\in\mathcal C_{m-1}$, and outgoing carrier $G'$ and color
$j\in\mathcal C_m$. Its carrier states have the form
\[
 G=(G_1,\ldots,G_{m-1},\infty),\qquad
 G_1,\ldots,G_{m-1}\in\mathbb Z_{\ge0}.
\]
The last coordinate represents an unlimited amount of color $m$. Adding or subtracting $E_m$ does not change
this coordinate. Here $E_k$ is the standard basis vector in coordinate
$k$, and
$G_{[1,k-1]}=\sum_{a=1}^{k-1}G_a$.
The nonzero weights are displayed in
Figure~\ref{fig:intro-normal-weights}. The carrier is read from bottom
to top and the physical color from left to right.

\begin{figure}[htbp]
\centering
\begingroup
\newcommand{\NormalConfig}[4]{%
\begin{tikzpicture}[baseline=0pt,
                    every node/.style={font=\small,inner sep=1pt}]
  \draw[paper normal,line width=3.8pt,
        -{Stealth[length=5.5pt,width=7.5pt]}]
        (0,-.60)--(0,.70);
  \draw[paper normal,line width=1.5pt,
        -{Stealth[length=4pt,width=4.5pt]}]
        (-.72,0)--(.78,0);
  \node[above=3pt] at (0,.70) {$#1$};
  \node[left=4pt] at (-.72,0) {$#2$};
  \node[right=4pt] at (.78,0) {$#3$};
  \node[below=3pt] at (0,-.60) {$#4$};
\end{tikzpicture}}
\setlength{\tabcolsep}{0pt}
\setlength{\arrayrulewidth}{.35pt}
\renewcommand{\arraystretch}{1.08}
\begin{tabular}{|*{3}{w{c}{\dimexpr(.85\textwidth-4\arrayrulewidth)/3\relax}|}}
\hline
\rule[-1.04cm]{0pt}{2.30cm}
\NormalConfig{G}{0}{0}{G}
&\NormalConfig{G}{k}{k}{G}
&\NormalConfig{G-E_k}{0}{k}{G}
\\
\rule[-.50cm]{0pt}{1.10cm}
{\small $\displaystyle 1-\rho_m$}
&{\small $\displaystyle
  \rho_m q^{G_{[1,k-1]}}(1-x q^{G_k})$}
&{\small $\displaystyle
  \rho_m q^{G_{[1,k-1]}}(1-q^{G_k})$}
\\ \hline
\rule[-1.04cm]{0pt}{2.30cm}
\NormalConfig{G+E_k}{k}{0}{G}
&\NormalConfig{G+E_k-E_\ell}{k}{\ell}{G}
&\NormalConfig{G+E_\ell-E_k}{\ell}{k}{G}
\\
\rule[-.50cm]{0pt}{1.10cm}
{\small $\displaystyle 1-\rho_m$}
&{\small $\displaystyle
  \rho_m x q^{G_{[1,\ell-1]}}(1-q^{G_\ell})$}
&{\small $\displaystyle
  \rho_m q^{G_{[1,k-1]}}(1-q^{G_k})$}
\\ \hline
\end{tabular}
\endgroup
\caption{The usual vertex $\mathbb W^{(m)}$.
The carrier edge is read from bottom to top and the physical edge from left to right. Thick gray edges carry carrier states and thin gray edges carry physical colors. We have $1\le k<\ell\le m$.}
\label{fig:intro-normal-weights}
\end{figure}

\paragraph{Return vertices.}
The return vertex $\widehat L^{\mathrm{ver}}(i,H;j,H')$ has incoming
physical color $i$ and carrier $H$, and outgoing color $j$ and carrier
$H'$, with $i,j\in\mathcal C_m$. The carrier edge should be viewed as almost entirely filled with color $1$, with finitely many ''defects'' of colors $0,2,\ldots,m$. It is indexed via
\[
 \widehat{\mathcal C}_m=\mathcal C_m\setminus\{1\},\qquad
 H=(H_0,H_2,\ldots,H_m)
 \in\mathbb Z_{\ge0}^{\widehat{\mathcal C}_m}.
\]
Color $1$ is the implicit background and is not recorded in $H$;
a physical input $1$ therefore adds no carrier occupation. Write
$|H|=H_0+\sum_{a=2}^m H_a$, and let $\widehat E_k$ denote the
standard basis vector in coordinate $k\in\widehat{\mathcal C}_m$.
Figure~\ref{fig:intro-return-weights} gives all nonzero weights. The vertex weights satisfy $H'=H+\widehat E_i-\widehat E_j$, where $\widehat E_1=0$. 
The physical color is read from left to right, as before, but the
return carrier is read from top to bottom.
For example, in the bottom-right configuration, the carrier coming from the top receives color $\ell$ from the left and releases color $k$ to the right. 

Both usual and return vertices have a useful queue-type description, see Proposition \ref{def:engine-graphical-realization} below.  

\begin{figure}[htbp]
\centering
\begingroup
\newcommand{\ReturnConfig}[4]{%
\begin{tikzpicture}[baseline=0pt,
                    every node/.style={font=\small,inner sep=1pt}]
  \draw[paper hole,line width=3.8pt,
        -{Stealth[length=5.5pt,width=7.5pt]}]
        (0,.60)--(0,-.70);
  \draw[paper normal,line width=1.5pt,
        -{Stealth[length=4pt,width=4.5pt]}]
        (-.72,0)--(.78,0);
  \node[above=3pt,text=holecolor] at (0,.60) {$#1$};
  \node[left=4pt] at (-.72,0) {$#2$};
  \node[right=4pt] at (.78,0) {$#3$};
  \node[below=3pt,text=holecolor] at (0,-.70) {$#4$};
\end{tikzpicture}}
\setlength{\tabcolsep}{0pt}
\setlength{\arrayrulewidth}{.35pt}
\renewcommand{\arraystretch}{1.08}
\begin{tabular}{|*{3}{w{c}{\dimexpr(.85\textwidth-4\arrayrulewidth)/3\relax}|}}
\hline
\rule[-1.04cm]{0pt}{2.20cm}
\ReturnConfig{H}{1}{1}{H}
&\ReturnConfig{H}{k}{k}{H}
&\ReturnConfig{H}{1}{k}{H-\widehat E_k}
\\
\rule[-.50cm]{0pt}{1.10cm}
{\small $\displaystyle 1-\rho_m+\rho_m q^{|H|}$}
&{\small $\displaystyle
  \rho_m q^{\sum_{a\prec k}H_a}(1-x q^{H_k})$}
&{\small $\displaystyle
  \rho_m q^{\sum_{a\prec k}H_a}(1-q^{H_k})$}
\\ \hline
\rule[-1.04cm]{0pt}{2.20cm}
\ReturnConfig{H}{k}{1}{H+\widehat E_k}
&\ReturnConfig{H}{k}{\ell}
  {H+\widehat E_k-\widehat E_\ell}
&\ReturnConfig{H}{\ell}{k}
  {H+\widehat E_\ell-\widehat E_k}
\\
\rule[-.50cm]{0pt}{1.10cm}
{\small $\displaystyle 1-\rho_m+\rho_m x q^{|H|}$}
&{\small $\displaystyle
  \rho_m x q^{\sum_{a\prec\ell}H_a}(1-q^{H_\ell})$}
&{\small $\displaystyle
  \rho_m q^{\sum_{a\prec k}H_a}(1-q^{H_k})$}
\\ \hline
\end{tabular}
\endgroup
\caption{The return vertex $\widehat L^{\mathrm{ver}}$.
The carrier edge is read from top to bottom and the physical edge from left to right. Thick salmon edges carry carrier states and thin gray edges carry physical colors. Here $k,\ell\in\{0,2,\ldots,m\}$ and $k \prec \ell$.}
\label{fig:intro-return-weights}
\end{figure}

\paragraph{Boundary vertices.}
The boundary vertex $\mathbb K^{(m)}(G;H)$ joins the carrier of a
usual vertex to the carrier of a return vertex. It discards $G$ and
samples $H$ independently of the incoming state. For $0<\alpha\le1$,
put
\[
 b_{\alpha,\rho_m}
 =\frac{(1-\alpha)(1-\rho_m)}{\alpha\rho_m}.
\]
Condition~\eqref{eq:intro-vertex-parameters} ensures
$0\le b_{\alpha,\rho_m}<1$. The output is
\begin{equation}
\label{eq:intro-cap-distribution}
 H=(H_0,0,\ldots,0),\qquad
 \mathbb P(H_0=h)=p_{b_{\alpha,\rho_m}}^{(q)}(h)
 =(b_{\alpha,\rho_m};q)_\infty
 \frac{b_{\alpha,\rho_m}^{\,h}}{(q;q)_h},
 \quad h\in\mathbb Z_{\ge0},
\end{equation}
where we use standard q-Pochhammer notation:
\[
 (u;q)_h=\prod_{j=0}^{h-1}(1-uq^j),\qquad
 (u;q)_\infty=\prod_{j=0}^\infty(1-uq^j).
\]
For $\alpha \ge 1$, the output is deterministically the zero state.
The two cases are summarized in
Figure~\ref{fig:intro-boundary-weights}.

\begin{figure}[htbp]
\centering
\begin{equation}
\label{eq:intro-cap-vertex-weights}
 \vcenter{\hbox{%
 \begin{tikzpicture}[x=.8cm,y=.8cm,
   every node/.style={font=\small,inner sep=1pt}]
  \draw[paper normal,line width=2.4pt]
    (-.8,0)--(-.8,.55)--(0,1.15);
  \draw[paper hole,line width=2.4pt]
    (0,1.15)--(.8,.55)--(.8,0);
  \node[paper boundary] at (0,1.15) {};
  \node[above=4pt] at (0,1.15) {$\mathbb K^{(m)}$};
  \node[below=3pt] at (-.8,0) {$G$};
  \node[below=3pt,text=holecolor] at (.8,0) {$H$};
 \end{tikzpicture}}}
 \qquad
 \mathbb K^{(m)}(G;H)
 =\mathbf1_{\{H_2=\cdots=H_m=0\}}
 \begin{cases}
    p_{b_{\alpha,\rho_m}}^{(q)}(H_0),&0<\alpha\le1,\\[1mm]
    \mathbf1_{\{H_0=0\}},&\alpha>1.
   \end{cases}
\end{equation}
\caption{The boundary vertex $\mathbb K^{(m)}$. The carrier $G$ enters
along the gray edge and the return carrier $H$ leaves along the salmon
edge. The output law does not depend on $G$; for $\alpha\ge1$ it is
concentrated on the zero state.}
\label{fig:intro-boundary-weights}
\end{figure} 

\paragraph{Engines.}
An \emph{engine} combines the three types of vertices into a reflected double
row (this type of objects was used previously in \cite{Sklyanin1988} and, especially, \cite{MR4924418}, \cite{deGierMeadRemenikWheeler2025}): a passage through usual vertices towards the boundary, one
boundary vertex, and a passage through return vertices away from the
boundary. We first define this construction on finitely many sites.

\Needspace{10\baselineskip}
\begin{definition}
\label{def:intro-queue-array}
\label{def:intro-vertex-array}
Fix $N\ge1$ and assume \eqref{eq:intro-vertex-parameters}. For
$1\le m\le r$, set
\[
 \mathcal G_m=\mathbb Z_{\ge0}^{m-1}\times\{\infty\},\qquad
 \mathcal H_m=\mathbb Z_{\ge0}^{\widehat{\mathcal C}_m},\qquad
 \mathbf0^{(m)}=(0,\ldots,0,\infty).
\]
In the following sum, $\boldsymbol v\in\mathcal C_m^N$,
$\boldsymbol G=(G^{(0)},\ldots,G^{(N-1)})\in\mathcal G_m^N$,
$\boldsymbol H=(H^{(0)},\ldots,H^{(N)})\in\mathcal H_m^{N+1}$,
and $G^{(N)}=\mathbf0^{(m)}$.
The \emph{$m$-th engine} is the Markov kernel
$\mathfrak E^{N,m}:\mathcal C_{m-1}^N\times\mathcal C_m^N\to[0,1]$
defined by
\[
\mathfrak E^{N,m}(\boldsymbol\eta,\boldsymbol\xi)=\sum_{\boldsymbol v,\boldsymbol G,\boldsymbol H}
 \mathbb K^{(m)}(G^{(0)};H^{(0)})\notag\prod_{i=1}^N
 \mathbb W^{(m)}(G^{(i)},\eta_i;G^{(i-1)},v_i)
 \widehat L^{\mathrm{ver}}(v_i,H^{(i-1)};\xi_i,H^{(i)}).
\]
Thus the usual carrier runs through sites $N,\ldots,1$, and the
return carrier through sites $1,\ldots,N$. The kernel acts on measures
by
$(\nu\mathfrak E^{N,m})(\boldsymbol\xi)
 =\sum_{\boldsymbol\eta}\nu(\boldsymbol\eta)
 \mathfrak E^{N,m}(\boldsymbol\eta,\boldsymbol\xi)$.
The engine output law is
\begin{equation}
\label{eq:intro-vertex-array-law}
 \mu^N_{\mathrm{eng}}
 =\delta_{\boldsymbol0_N}
   \mathfrak E^{N,1}\mathfrak E^{N,2}\cdots\mathfrak E^{N,r},
 \qquad \boldsymbol0_N=(0,\ldots,0).
\end{equation}
A graphical illustration of this procedure is shown in Figure~\ref{fig:intro-vertex-array}.
\end{definition}

\begin{figure}[!htbp]
\centering
\scalebox{1.0}{%
\begin{tikzpicture}[
  x=1.08cm,y=.63cm,
  engstate/.style={font=\scriptsize,inner sep=.8pt},
  site arrow/.style={-{Stealth[length=3.6pt,width=3.2pt]}},
  carrier arrow/.style={-{Stealth[length=5pt,width=5pt]}}
]

\foreach \yy/\site in {1.0/N,2.5/{N-1},4.6/2,6.0/1}
{
  \draw[paper wire,site arrow] (.3,\yy)--(10.7,\yy);
  \foreach \xx in {3.725,5.50,7.275}
    \draw[paper wire,site arrow]
      ({\xx-.20},\yy)--({\xx+.06},\yy);
  \node[engstate,left=2pt] at (.3,\yy) {$0$};
}
\node[engstate,left=2pt] at (.3,3.4) {$\vdots$};
\node[engstate] at (10.7,3.4) {$\vdots$};

\draw[decorate,decoration={brace,amplitude=4pt},line width=.5pt]
  (10.95,6.2)--(10.95,.8)
  node[midway,right=7pt,font=\small] {$\mu_{\mathrm{eng}}$};

\foreach \cx/\eng in {1.95/1,5.50/m,9.05/r}
{
  \draw[paper normal]
    ({\cx-.75},.25)--({\cx-.75},6.3)--(\cx,8.3);
  \draw[paper hole]
    (\cx,8.3)--({\cx+.75},6.3)--({\cx+.75},.25);

  \foreach \ay in {1.45,5.55}
  {
    \draw[paper normal,carrier arrow]
      ({\cx-.75},{\ay-.30})--({\cx-.75},{\ay+.14});
    \draw[paper hole,carrier arrow]
      ({\cx+.75},{\ay+.30})--({\cx+.75},{\ay-.14});
  }

  \node[paper boundary] at (\cx,8.3) {};
  \node[engstate,above=3pt] at (\cx,8.3)
    {$\mathbb K^{(\eng)}$};
}

\foreach \cx/\eng in {1.95/1,9.05/r}
{
  \node[engstate,above left=3pt] at ({\cx-.75},6.0)
    {$\mathbb W^{(\eng)}$};
  \node[engstate,above right=3pt,text=holecolor]
    at ({\cx+.75},6.0) {$\widehat L^{\mathrm{ver}}$};
}

\path (4.75,6.3)--(5.50,8.3)
  coordinate[pos=.56] (cap-in-start)
  coordinate[pos=.74] (cap-in-end);
\path (5.50,8.3)--(6.25,6.3)
  coordinate[pos=.22] (cap-out-start)
  coordinate[pos=.38] (cap-out-end);
\draw[paper normal,carrier arrow]
  (cap-in-start)--(cap-in-end);
\draw[paper hole,carrier arrow]
  (cap-out-start)--(cap-out-end);

\path (4.75,6.3)--(5.50,8.3)
  node[pos=.45,engstate,left=3pt] {$G^{(0)}$};
\path (5.50,8.3)--(6.25,6.3)
  node[pos=.55,engstate,right=3pt,text=holecolor]
    {$H^{(0)}$};

\foreach \yy/\idx in {5.3/1,4.0/2,1.95/{N-1}}
{
  \node[engstate,left=3pt] at (4.75,\yy)
    {$G^{(\idx)}$};
  \node[engstate,right=3pt,text=holecolor] at (6.25,\yy)
    {$H^{(\idx)}$};
}
\node[engstate] at (4.4,3.4) {$\vdots$};
\node[engstate,text=holecolor] at (6.6,3.4) {$\vdots$};

\node[engstate,below=4pt] at (4.75,.25)
  {$G^{(N)}$};
\node[engstate,below=4pt,text=holecolor] at (6.25,.25)
  {$H^{(N)}$};

\foreach \yy/\site in {1.0/N,2.5/{N-1},4.6/2,6.0/1}
{
  \node[engstate,above=2pt] at (3.725,\yy)
    {$\eta_{\site}$};
  \node[engstate,above=2pt] at (5.50,\yy)
    {$v_{\site}$};
  \node[engstate,above=2pt] at (7.275,\yy)
    {$\xi_{\site}$};
}
\node[engstate] at (5.50,3.4) {$\vdots$};

\foreach \xx in {3.725,7.275}
  \node[engstate] at (\xx,3.4) {$\cdots$};
\end{tikzpicture}
}
\caption{Engine construction of the stationary measure for the multi-species half-line ASEP. Engines $1,\ldots,r$ act successively from left to right, starting from empty input at the left boundary. For the middle engine $m$, $\eta_i$, $v_i$, and $\xi_i$ denote the local input, intermediate color, and local output. The gray carrier reads sites $N,\ldots,1$, evolving from $G^{(N)}$ to $G^{(0)}$. The boundary vertex $\mathbb K^{(m)}$ samples $H^{(0)}$, and the salmon return carrier reads sites $1,\ldots,N$. The stationary law is obtained by taking the local limit of $\mu_{\mathrm{eng}}$ as $N\to\infty$.}
\label{fig:intro-vertex-array}
\end{figure}
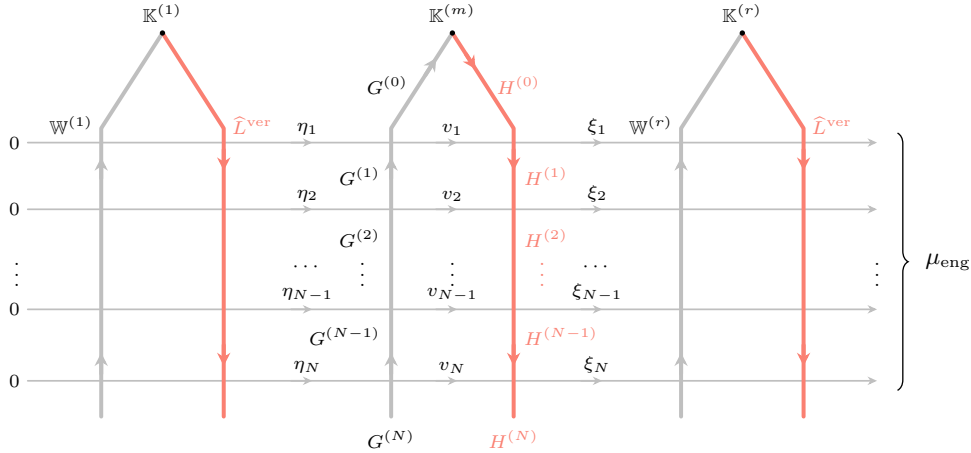

The laws in Definition~\ref{def:intro-vertex-array} depend on $N$.
Their convergence in the next theorem is \emph{local}: the restriction
to each fixed finite set of sites converges in distribution.

\begin{theorem}
\label{thm:main-stationary-construction}
\label{thm:engine-array-stationarity}
\label{thm:intro-vertex-form}
Suppose that
\begin{equation}
\label{eq:intro-strict-density-range}
 0\le q<1,\qquad 0<\alpha<\infty,\qquad
 \max(1/2,1-\alpha) < \rho_1<\cdots<\rho_r<1.
\end{equation}
The laws $\mu^N_{\mathrm{eng}}$ of
Definition~\ref{def:intro-vertex-array} converge locally, as
$N\to\infty$, to a probability measure
$\mu^{\alpha,\rho_1,\ldots,\rho_r}$
on $\mathcal C_r^{\mathbb Z_{>0}}$.
This measure is stationary for the multispecies half-line ASEP defined
above.
\end{theorem}

The theorem also admits the following limit, which gives an additional low-density family.

\begin{remark}
\label{prop:intro-low-density-entrance}

Assume that $0 <\alpha \le 1/2$ and $1-\alpha < \rho_2 < \dots < \rho_r<1$. The following limit exists:
\begin{equation}
\label{eq:intro-low-density-boundary-limit}
 \mu^{\alpha,u,\rho_2,\ldots,\rho_r}
 \ \xrightarrow[\ u\downarrow1-\alpha\ ]{\mathrm{loc}}\ 
 \mu^{\alpha,\alpha,\rho_2,\ldots,\rho_r},
 \qquad 1-\alpha<u<\rho_2,
\end{equation}
where the convergence, as before, is in the local sense. 
The limiting measure is the stationary measure for densities $\alpha, \rho_2, \dots, \rho_r$. Note that the parameter of the measure before the limit converges to $(1-\alpha)$, while the limiting density of first class particles is $\alpha$.
\end{remark}

\subsection{The first two sites of the two-species stationary measure}
\label{subsec:intro-boundary-marginal}

We now specialize to two particle species, with priority order
$1\succ2\succ0$, and compute the joint distribution of the colors at
sites $1$ and $2$. The vertex construction of the previous subsection allows (after considerable additional work) to obtain the following result, which is the second main theorem of the present paper. 

\begin{theorem}
\label{thm:two-site-marginal}
Fix $0\le q<1$ and suppose that one of the two regimes holds:
\begin{equation}
\label{eq:intro-two-site-regimes}
\begin{aligned}
\mathrm{(H)}\quad&0<\alpha<\infty,\qquad
    \max\{1/2,1-\alpha\}<\rho_1<\rho_2<1,\\
\mathrm{(L)}\quad&0<\alpha\le\tfrac12,\qquad
    \rho_1=\alpha,\qquad1-\alpha<\rho_2<1.
\end{aligned}
\end{equation}
Let $\eta\sim\mu^{\alpha,\rho_1,\rho_2}$ be the stationary law
from Theorem~\ref{thm:main-stationary-construction} in regime
$\mathrm{(H)}$, or from Remark~\ref{prop:intro-low-density-entrance}
in regime $\mathrm{(L)}$. Set
\begin{equation}
\label{eq:intro-two-site-abbreviations}
 J_i:=J(\rho_i)=\rho_i(1-\rho_i),\qquad
 j_i:=\frac{J_i}{\alpha},\qquad i=1,2,
 \qquad \theta:=q+(1-q)\rho_2.
\end{equation}
Write $p_{ij}=\mathbb P(\eta_1=i,\eta_2=j)$; rows and columns in
the following table are indexed in the order $0,2,1$.

We have

\Needspace{15\baselineskip}
\begingroup
\setlength{\arraycolsep}{4pt}
\begin{equation}
\label{eq:two-site-marginal}
 (p_{ij})=
 \begin{pmatrix}
  j_2^2
  & (j_1-j_2)\bigl(j_2+j_1(1-\theta)\bigr)
  & \begin{gathered}
      j_1(1-j_1)\\[-2pt]
      {}-(j_1-j_2)(1-j_1\theta)
    \end{gathered}
  \\[3mm]
  (j_1-j_2)\bigl(j_1\theta(1-\theta)+qj_2\bigr)
  & (j_1-j_2)\bigl(j_1\theta^2-qj_2\bigr)
  & (j_1-j_2)(1-j_1\theta)
  \\[3mm]
  \begin{gathered}
    j_2\bigl((1-q)\alpha+q(1-j_1)\bigr)\\[-2pt]
    {}-(j_1-j_2) j_1\theta(1-\theta)
  \end{gathered}
  & \begin{gathered}
     (j_1-j_2)\bigl((1-q)\alpha+q(1-j_1)\\[-2pt]
     {}+j_1\theta(1-\theta)\bigr)
    \end{gathered}
  & \begin{gathered}
      1-j_1-(1-q)\alpha j_1\\[-2pt]
      {}-qj_1(1-j_1)
    \end{gathered}
 \end{pmatrix}.
\end{equation}
\endgroup
\end{theorem}

We believe that explicit computations of correlations are the most interesting part of the story behind stationary measures of (multi-species) ASEP, see \cite[Theorem~1.5(b)]{MR4089793}, \cite[Theorem~1.7]{AmirAngelValko2011}, \cite[Theorem~5.1]{AyyerLinusson2017}, and \cite[Theorems~1.1--1.2]{PahujaDeSarkar2025}. The main motivation of our work was to extend this list of results with a half-space example. 

For $q=0$, it seems plausible that our low-high-density (case (L) above) correlation formulas could also be obtained from the matrix-product solution for the finite open two-species TASEP in \cite{CrampeEvansMallickRagoucyVanicat2016}. 

\subsection{The half-space speed process}
\label{subsec:half-space-ASEP}

A stationary measure of multi-species ASEP is interesting in its own right. However, it also has a quite different and very appealing characterization in terms of the \textit{speed process}. To prove this connection precisely requires several results of a different nature than those used to study the stationary measure. We only partially address the required types of statements in Section 6 and formulate stronger results in the form of conjectures. 

To avoid additional nuances, we set $q=0$ and fix $0<\alpha\le1$.
Consider the half-space TASEP with infinitely many colors
on $\mathbb Z^\ast=\mathbb Z\setminus\{0\}$. Its configuration
$\eta_t$ is a signed permutation: $\eta_t(x)$ is the color at site $x$,
$\eta_t(-x)=-\eta_t(x)$, and initially $\eta_0(i)=i$.
For each $x\ge1$, a rate-one clock exchanges the colors at $x,x+1$
and simultaneously at $-x,-x-1$ whenever
$\eta_t(x)<\eta_t(x+1)$. At the boundary, a rate-$\alpha$ clock
exchanges the colors at $-1,1$ whenever
$\eta_t(-1)<\eta_t(1)$. All these clocks are independent.
Thus the negative half-line contains the same configuration with a different sign.
Write $X_i(t)=\eta_t^{-1}(i)$ for the position of color $i$.

\Needspace{7\baselineskip}
\begin{conjecture}
\label{conj:half-space-asymptotic-speeds}
For any $i\in\mathbb Z^\ast$, the almost sure limit
\[
 U_i:=\lim_{t\to\infty}\frac{X_i(t)}{t},
\]
exists.
\end{conjecture}

Assuming this conjecture, we call the random sequence
$(U_i)_{i\in\mathbb Z^\ast}$ the \emph{TASEP half-space speed process}.
Although almost-sure convergence remains conjectural, the distributional
limit for the first particle can be proved, including its negative part.

\begin{theorem}
\label{thm:intro-first-color-signed-speed}
Set $\hat\alpha=\max\{0,1-2\alpha\}$. As $t\to\infty$,
\[
 \frac{X_1(t)}t\Longrightarrow V_\alpha,
\]
where the absolutely continuous part of $V_\alpha$ has density
\[
 f_\alpha(v)=\frac12\mathbf1_{(-1,-\hat\alpha)}(v)
    +\frac{v}{2\alpha}\mathbf1_{(\hat\alpha,1)}(v).
\]
For $\alpha>1/2$, the limiting law also has an atom at zero of mass
$(2\alpha-1)/(4\alpha)$; for $0<\alpha\le1/2$ there are no atoms.
\end{theorem}

This is Theorem~\ref{thm:first-color-signed-speed} below.

The main result of Section \ref{sec:colored-particle-speeds} is the following joint distribution of speeds of first two particles, conditioned on the event that they stayed on the positive half-space. It is conditional on the existence of certain local equilibrium limit. 

\begin{theorem}
\label{thm:intro-joint-boundary-speeds}
Assume Conjecture~\ref{conj:two-species-boundary-local-equilibrium}
and set $\mathcal S_t=\{X_1(t)>0,\ X_2(t)>0\}$.
Conditional on $\mathcal S_t$, the pair
$\bigl(X_1(t)/t,X_2(t)/t\bigr)$ converges weakly to a probability
measure on $[\hat\alpha,1]^2$ with off-diagonal density
\[
 \frac{1}{(1-\hat\alpha^2)^2}
 \begin{cases}
 2uv(v^2-u^2),&\hat\alpha<u<v<1,\\[1mm]
 v(1+2u+3u^2-2v^2),&\hat\alpha<v<u<1,
 \end{cases}
\]
and diagonal line density at $(u,u)$, with respect to $du$, given by
\[
 \frac{u(1+u)^2(1-u^2)}{2(1-\hat\alpha^2)^2},
 \qquad \hat\alpha<u<1.
\]
This is Theorem~\ref{thm:joint-boundary-speeds} below.
\end{theorem}

One peculiar detail of this distribution is that it does not depend on $\alpha$ as long as $\alpha \in [1/2,1]$ (so that $\hat \alpha =0$). For example, (modulo Conjectures \ref{conj:half-space-asymptotic-speeds} and \ref{conj:two-species-boundary-local-equilibrium} ), one has $P(U_1=U_2 | U_1>0, U_2>0) = 3/10$ for $\alpha \in [1/2;1]$. 

A conjectural formula for the full joint law (without conditioning on $\mathcal S_t$)
is recorded in Appendix~\ref{app:full-two-speed-law}.

\paragraph{Organization of the paper.}
Section~\ref{sec:preliminaries} recalls known preliminaries about the model. In Section~\ref{sec:colored-vertex-weights} we address the analytic continuation of vertex weights that we use in our construction. The material is close to previous works, but we have not found it in exactly this form. In particular, a hole-indexed analytic continuation of vertex weights might be slightly new. In Section~\ref{sec:queue-specialization-stationarity} we prove our first main result, Theorem~\ref{thm:main-stationary-construction}. The key to the proof, as in Aggarwal-Nicoletti-Petrov \cite{MR4968006}, is a commutation relation provided by the Yang-Baxter equation and the reflection equation. Nevertheless, it requires significant effort to properly apply this idea to our case. An important role is played by a correct choice of \textit{queue specialization}. In Section~5 we use this vertex construction for explicit calculations, proving Theorem~\ref{thm:two-site-marginal} above. Finally, we study distributions of speeds of particles in Section~\ref{sec:colored-particle-speeds}, in order to emphasize the significance of the joint distribution from Theorem~\ref{thm:two-site-marginal}. Appendix~\ref{app:full-two-speed-law} contains the conjectural full two-speed law and a simulation result.

\paragraph{AI use.}
Main results and ideas of the paper were developed by the authors. In particular, both main theorems above were fully obtained without the use of AI. The content of Section 6 is based on the interaction with ChatGPT; also, the writing of the paper was done with the use of ChatGPT. 

\paragraph{Acknowledgements.}
We are grateful to P.L. Ferrari for discussions related to material of Section \ref{sec:colored-particle-speeds}. 
A. Bufetov was partially supported by the European Research Council (ERC), Grant Agreement No. 101041499.

\section{Preliminaries}
\label{sec:preliminaries}
This section introduces notation and recalls several known results that will be used in the paper. 

\subsection{Matrix product ansatz for the half-space ASEP}
\label{subsec:half-space-mpa}
For \(\alpha, \rho\in(0,1)\), write
\[
    \bar\rho:=1-\rho,
    \qquad
    J(\rho):=\rho\bar\rho,
    \qquad
    j(\rho):=\frac{J(\rho)}{\alpha}.
\]
For \(u,v\in(0,1]\), set
\begin{equation}
\label{eq:density-odds-product}
    b_{u,v}:=\frac{(1-u)(1-v)}{uv}.
\end{equation}
Consider the half-space ASEP with right and left jump rates
\(1\) and \(q\in[0,1)\), respectively, and injection rate
\((1-q)\alpha\) at site \(1\).  Sasamoto and Williams
\cite[Theorems~4.1 and~5.1]{MR3324896} extended Grosskinsky's half-space
matrix-product ansatz \cite[Theorem~3.2]{Grosskinsky2004} to this setting.

\begin{proposition}
\label{prop:half-space-mpa}
Let \(0\leq q<1\), \(0<\alpha<1\), and \(0<\rho<1\).  
Consider matrices \(D,E\) and vectors \(\langle W|,|V\rangle\) satisfying
\begin{equation}
\label{eq:half-space-mpa-algebra}
    DE-qED=(1-q)J(\rho)(D+E),
    \qquad
    \alpha\langle W|E=J(\rho)\langle W|,
    \qquad
    (D+E)|V\rangle=|V\rangle.
\end{equation}
These relations define a stationary probability law $\mu^{\alpha,\rho}$ of the half-space ASEP
through its cylinder probabilities. More precisely, for $\alpha>\min{\rho, 1-\rho}$  if $\eta\sim\mu^{\alpha,\rho}$, then
for every \(n\geq1\) and \(\varepsilon_1,\ldots,\varepsilon_n\in\{0,1\}\),
\begin{equation}
\label{eq:half-space-mpa-cylinder}
\mathbb P(\eta(1)=\varepsilon_1,\ldots,\eta(n)=\varepsilon_n)
=\left\langle W\left|\prod_{i=1}^n
\bigl(\varepsilon_iD+(1-\varepsilon_i)E\bigr)\right|V\right\rangle
\end{equation}
Its stationary current is
\begin{equation}
\label{eq:half-space-mpa-current}
\mathbb P(\eta(i)=1,\eta(i+1)=0)
-q\mathbb P(\eta(i)=0,\eta(i+1)=1)
=(1-q)J(\rho),\qquad i\geq1.
\end{equation}
\end{proposition}

The algebraic relation determines the distribution at the first two sites even without using an explicit solution.

\begin{proposition}
\label{prop:one-species-left-edge}
Under the assumptions of Proposition~\ref{prop:half-space-mpa}, let
$\eta\sim\mu^{\alpha,\rho}$ and set \(j:=j(\rho)\).  We have
\begingroup
\small
\setlength{\arraycolsep}{3pt}
\begin{equation}
\label{eq:one-species-left-edge}
\left(
\mathbb P(\eta(1)=a,\eta(2)=b)
\right)_{a,b\in\{0,1\}}
=
\begin{pmatrix}
j^2 & j(1-j)\\[1mm]
j\bigl((1-q)\alpha+q(1-j)\bigr)
& (1-j)(1-qj)-(1-q)\alpha j
\end{pmatrix}.
\end{equation}
\endgroup
\end{proposition}

\begin{proof}
The boundary relation and \((D+E)|V\rangle=|V\rangle\) give
\[
\mathbb P(\eta(1)=0,\eta(2)=0) =\langle W|E^2|V\rangle=j^2,
\mathbb P(\eta(1)=0)=\langle W|E(D+E)|V\rangle=j.
\]
Moreover, the bulk relation gives
\[
\mathbb P(\eta(1)=1,\eta(2)=0)
-q\mathbb P(\eta(1)=0,\eta(2)=1)
=(1-q)J(\rho).
\]
The four probabilities now follow by subtraction.
\end{proof}

\subsection{The Ferrari--Martin construction}
\label{subsec:ferrari-martin}

Ferrari and Martin \cite{MR2319708} constructed the stationary measures of the
full-space multispecies TASEP using multi-line queues.  Martin
\cite{MR4089793} extended this construction to the multispecies ASEP.  In this section,
we collect their results in our notation. Throughout this
subsection, $q\in[0,1)$ and the priority order is $1\succ2\succ0$. 

We begin with the two-species case.  Fix
\[
    0<\rho_1<\rho_2<1,
\]
where $\rho_1$ is the density of first-class particles and $\rho_2$ is the
combined density of first- and second-class particles.  Let
$A=(A(i))_{i\in\mathbb Z}$ and $S=(S(i))_{i\in\mathbb Z}$ be independent
Bernoulli processes of parameters $\rho_1$ and $\rho_2$, respectively.  We
interpret $A(i)=1$ as an arrival and $S(i)=1$ as an available service.

Let $Q_{i+1}$ be the queue length entering site $i$, and let $Q_i$ be the queue 
length after that site has been processed.  Suppose that $Q_{i+1}=k$. Denoting the output by \(d(i)\in\{0,1,2\}\),
the Martin's multi-line algorithm is defined as follows:
\begin{itemize}
    \item If $A\left(i\right) = S\left(i\right) = 0$, then $d\left(i\right) = 0$ and $Q_i = k$,
    \item if $A\left(i\right) = S\left(i\right) = 1$, then the arrival immediately takes the service, so $d\left(i\right) = 1$ and $Q_i = k$,
    \item if $A\left(i\right) = 1, S\left(i\right) = 0$, then $d\left(i\right) = 0$, $Q_i = k+1$, and the arrival service is added to the waiting queue,
    \item if $A\left(i\right) = 0, S\left(i\right) = 1$, then each particle in the queue independently rejects the service with probability $q$. Hence, with probability $q^k$, the service is rejected by all particles, producing a second-class particle at site $i$, so $d\left(i\right) = 2$, $Q_i = k$. With probability $1 - q^k$, the service is accepted by some particle in the queue, giving $Q_i = k -1 $ and $d\left(i\right) = 1$. 
\end{itemize}

We therefore have
\begin{equation}
\label{eq:martin-queue-recursion}
    Q_i=Q_{i+1}+A(i)-\mathbf1_{\{d(i)=1\}}.
\end{equation}

The construction extends recursively to any number of species.  Given
$0<\rho_1<\cdots<\rho_n<1$, take independent Bernoulli service processes of
densities $\rho_1,\ldots,\rho_n$.  The first line has a class-$1$ particle at
each service and a hole elsewhere.  At level $k\geq2$, use the output of level
$k-1$ as the arrival process and the $k$th Bernoulli line
as its service process.  The algorithm at each level is analogous to the two-color case, 
with service now offered in class-priority order. At level $k$, if a class-$b$ customer
arrives, where $1\leq b\leq k-1$, waiting customers of classes
$1,\ldots,b-1$ are offered any available service first. If one accepts,
it departs and the arriving customer joins the queue. If they all reject,
a class-$b$ departure always occurs. Thus, a class-$k$ departure is only created when 
no arrival but a service rejected by all the waiting particles.
We denote this law by \(\pi^{\rho_1,\rho_2, \dots \rho_n}\).

To state the queue-length distribution, introduce the $q$-Pochhammer symbols.
For a complex parameter \(u\) and \(m\in\mathbb Z_{\geq0}\), set
\[
    (u;q)_m:=\prod_{j=0}^{m-1}(1-uq^j),
    \qquad
    (u;q)_\infty:=\prod_{j=0}^{\infty}(1-uq^j).
\]
We also write
\[
    \begin{bmatrix}m\\r\end{bmatrix}_{\!q}
    :=\frac{(q;q)_m}{(q;q)_r(q;q)_{m-r}}.
\]

For real \(u<1\) and \(0\leq v<1\), define
\begin{equation}
\label{eq:q-negative-binomial-law}
    p_{u,v}^{(q)}(k)
    :=\frac{(v;q)_\infty}{(uv;q)_\infty}
      \frac{(u;q)_k}{(q;q)_k}v^k,
    \qquad k\in\mathbb Z_{\geq0}.
\end{equation}
By the $q$-binomial theorem, \eqref{eq:q-negative-binomial-law} is a
probability distribution.  We call it the $q$-negative binomial distribution
and write \(X\sim q\text{-}\mathrm{NB}(u,v)\) if
\(\mathbb P(X=k)=p_{u,v}^{(q)}(k)\).  When \(u=0\), we abbreviate
\begin{equation}
\label{eq:q-Poisson-law}
    p_v^{(q)}(k):=p_{0,v}^{(q)}(k)
    =(v;q)_\infty\frac{v^k}{(q;q)_k},
    \qquad k\in\mathbb Z_{\geq0},
\end{equation}
and call the corresponding law the $q$-Poisson distribution, denoted by
\(q\text{-}\mathrm{Poi}(v)\).  At \(q=0\), this is the geometric
distribution with parameter \(v\).

\begin{theorem}[\cite{MR4089793}, formula (6.5)]
\label{thm:martin-stationary-queue}
In the two-color setting, set
\begin{equation}
\label{eq:martin-queue-parameter}
    r_{12}:=\frac{\rho_1(1-\rho_2)}{(1-\rho_1)\rho_2}.
\end{equation}
The stationary distribution of the queue length process is
\[
    Q_i\sim q\text{-}\mathrm{Poi}(r_{12}),
    \qquad i\in\mathbb Z.
\]
\end{theorem}

Set
\begin{equation}
\label{eq:two-site-abbreviations}
    \theta:=q+(1-q)\rho_2,
    \qquad
    \phi:=1-(1-q)\rho_1.
\end{equation}

We also need the two neighboring sites joint distribution given in the following table.

\begin{theorem}[\cite{MR4089793}, Lemma~6.3(i)--(iii)]
\label{thm:full-space-two-site-law}
In the two-color setting, let $d\sim\pi^{\rho_1,\rho_2}$.  Then
\begingroup
\footnotesize
\renewcommand{\arraystretch}{1.55}
\setlength{\tabcolsep}{4pt}
\begin{equation}
\label{eq:22-full-space-two-site-table}
\begin{array}{c|ccc}
\mathbb P(d(1)=i,d(2)=j)
    &j=0&j=2&j=1\\
\hline
i=0
    &(1-\rho_2)^2
    &(\rho_2-\rho_1)(1-\rho_2)
    &\rho_1(1-\rho_2)
\\
i=2
    &(\rho_2-\rho_1)(1-\rho_2)\phi
    &(\rho_2-\rho_1)(\rho_2-\rho_1\theta)
    &\rho_1(\rho_2-\rho_1)
\\
i=1
    &\rho_1(1-\rho_2)(\theta+\phi-q)
    &\rho_1(\rho_2-\rho_1)\theta
    &\rho_1^2.
\end{array}
\end{equation}
\endgroup
\end{theorem}

The queue length is generally dependent on the output process.  The
half-space calculations require its conditional law given the colors at two
neighboring sites.

\begin{proposition}
\label{prop:conditional-two-site-queue}
Consider the stationary queue coupling of
Subsection~\ref{subsec:ferrari-martin}, with \(r_{12}\) as in
\eqref{eq:martin-queue-parameter}.  Start with
\(Q_3\sim q\text{-}\mathrm{Poi}(r_{12})\), independently of the arrivals,
services, and rejection variables at sites \(2\) and \(1\).  Process site
\(2\) and then site \(1\), and let \(Q_1\) be the resulting queue length.  For
$i,j\in\{0,1,2\}$, define $Q^{ij}$ by the conditional law
\[
    Q^{ij}\overset d=\bigl(Q_1\mid(d(1),d(2))=(i,j)\bigr).
\]
Then
\begin{equation}
\label{eq:22-queue-given-11}
Q^{11}\sim q\text{-}\mathrm{Poi}(r_{12}),\qquad
Q^{12},Q^{21}\sim q\text{-}\mathrm{Poi}(qr_{12}).
\end{equation}
Moreover, set
\begin{equation}
\label{eq:22-beta-10}
    \beta_{10}:=\frac{\rho_1\theta}{\theta+\phi-q}.
\end{equation}
Then
\begin{equation}
\label{eq:22-queue-given-10}
    Q^{10}\overset{d}{=}Q+\xi^{10},
    \qquad Q\sim q\text{-}\mathrm{Poi}(r_{12}),
    \qquad \xi^{10}\sim\operatorname{Ber}(\beta_{10}),
\end{equation}
where the two random variables on the right-hand side are independent.
\end{proposition}
\begin{proof}
By Theorem~\ref{thm:full-space-two-site-law}, the probabilities of the
conditioning events are known. It therefore suffices to compute the
joint probabilities
\[
\mathbb P(Q_1=m,d(1)=i,d(2)=j),
\qquad (i,j)\in\{(1,1),(1,2),(2,1),(1,0)\}.
\]
Write \(p_m=p_{r_{12}}^{(q)}(m)\), with \(p_{-1}=0\).

We first process site \(2\). An output of color \(1\) arises either
when an arrival uses the service immediately or when there is no
arrival and a waiting customer accepts the service. Color \(2\)
corresponds to an unused service, while color \(0\) means that no
service is available. Hence
\begin{align*}
\mathbb P(Q_2=m,d(2)=1)
&=\rho_1\rho_2p_m+(1-\rho_1)\rho_2(1-q^{m+1})p_{m+1},\\
\mathbb P(Q_2=m,d(2)=2)
&=(1-\rho_1)\rho_2q^mp_m,\\
\mathbb P(Q_2=m,d(2)=0)
&=(1-\rho_2)\bigl[(1-\rho_1)p_m+\rho_1p_{m-1}\bigr].
\end{align*}

We then apply the same queueing rule at site \(1\). In particular,
for the pair \((1,0)\), distinguishing whether an arrival occurs at
site \(1\) gives the following expression for
\(\mathbb P(Q_1=m,d(1)=1,d(2)=0)\):
\[
(1-\rho_2)\rho_2\Bigl\{\rho_1[(1-\rho_1)p_m+\rho_1p_{m-1}]
+(1-\rho_1)(1-q^{m+1})[(1-\rho_1)p_{m+1}+\rho_1p_m]\Bigr\}.
\]

To simplify these probabilities, we use
\[
(1-q^{m+1})p_v^{(q)}(m+1)=vp_v^{(q)}(m),
\qquad
q^mp_v^{(q)}(m)=(1-v)p_{qv}^{(q)}(m).
\]
Together with the definition of \(r_{12}\), these identities reduce
the probabilities of \(\{Q_2=m,d(2)=1\}\) and
\(\{Q_2=m,d(2)=2\}\) to \(\rho_1p_m\) and
\((\rho_2-\rho_1)p_{qr_{12}}^{(q)}(m)\), respectively.
Applying the same identities at site \(1\), and using
\[
(1-q^{m+1})p_m=(1-q)p_m+qr_{12}p_{m-1}
\]
for the pair \((1,0)\), we obtain
\begin{align*}
\mathbb P(Q_1=m,d(1)=1,d(2)=1)
&=\rho_1^2p_m,\\
\mathbb P(Q_1=m,d(1)=1,d(2)=2)
&=\rho_1(\rho_2-\rho_1)\theta\,p_{qr_{12}}^{(q)}(m),\\
\mathbb P(Q_1=m,d(1)=2,d(2)=1)
&=\rho_1(\rho_2-\rho_1)p_{qr_{12}}^{(q)}(m),\\
\mathbb P(Q_1=m,d(1)=1,d(2)=0)
&=\rho_1(1-\rho_2)\Bigl\{[(1-\rho_1)\theta+\phi-q]p_m+\rho_1\theta p_{m-1}\Bigr\}.
\end{align*}

Dividing by the corresponding probabilities in
Theorem~\ref{thm:full-space-two-site-law} proves the three
\(q\)-Poisson laws and gives
\[
\mathbb P(Q^{10}=m)=(1-\beta_{10})p_m+\beta_{10}p_{m-1}.
\]
This is the exactly the distribution of the sum of independent
\(q\text{-}\mathrm{Poi}(r_{12})\) and
\(\operatorname{Ber}(\beta_{10})\) random variables.
\end{proof}

\Needspace{14\baselineskip}
\subsection{Colored bulk vertex weights}
\label{subsec:colored-vertex-model}
We begin with the colored six-vertex model. We use the same priority order as for the particle system:
\[
    1\succ2\succ\cdots\succ n\succ0.
\]

The parameter $q$
is the same asymmetry parameter as in ASEP and is shared by all vertices. 
Each row and each column carries a nonzero complex parameter, called its \emph{rapidity}. 
At a crossing with column rapidity $x$ and row rapidity \(y\), the \emph{spectral parameter} 
$z$ is determined by its row and column rapidity $ z=\frac{x}{y}$ 
\begin{definition}
\label{def:elementary-colored-R}
Let $\mathcal C_n:=\{0,1,\ldots,n\}$. For \(i,j,k,\ell\in\mathcal C_n\), let
\(R_z(i,j;k,\ell)\) be the weight of a vertex whose incoming colors from
below and from the left are \(i,j\), and whose outgoing colors above and to
the right are \(k,\ell\).  If \(i=j\), then
\[
    R_z(i,i;i,i)=1.
\]
For \(i\succ j\), the remaining nonzero weights are
\begin{equation}
\label{eq:elementary-colored-R}
\begin{aligned}
R_z(i,j;i,j)&=\frac{q(1-z)}{1-qz},
&
R_z(i,j;j,i)&=\frac{1-q}{1-qz},
\\
R_z(j,i;j,i)&=\frac{1-z}{1-qz},
&
R_z(j,i;i,j)&=\frac{z(1-q)}{1-qz}.
\end{aligned}
\end{equation}
All other weights vanish.
\end{definition}

For every pair of incoming colors, these weights satisfy the sum-to-one property:
\[
    \sum_{k,\ell\in\mathcal C_n}R_z(i,j;k,\ell)=1.
\]

The integrability of the six-vertex model is expressed through the following Yang–Baxter equation.
\begin{proposition}
\label{prop:elementary-colored-ybe}
For any $i_1,i_2,i_3,k_1,k_2,k_3\in\mathcal C_n$,
\begin{align}
&\sum_{\gamma_1,\gamma_2,\gamma_3\in\mathcal C_n}
R_{u/v}(i_2,i_1;\gamma_2,\gamma_1)
R_{uw}(i_3,\gamma_1;\gamma_3,k_1)
R_{vw}(\gamma_3,\gamma_2;k_3,k_2)
\notag\\
&\qquad=\sum_{\gamma_1,\gamma_2,\gamma_3\in\mathcal C_n}
R_{vw}(i_3,i_2;\gamma_3,\gamma_2)
R_{uw}(\gamma_3,i_1;k_3,\gamma_1)
R_{u/v}(\gamma_2,\gamma_1;k_2,k_1).
\label{eq:elementary-colored-ybe}
\end{align}
\end{proposition}

Before giving its graphical representation, we introduce the conventions used throughout the paper.
\paragraph{Graphical conventions.}
\begin{itemize}
        \item Thin gray lines represent edges carrying a single color. 
        Thick gray lines represent fused edges, which may carry several arrows.
        \item The thick salmon lines carry a hole-indexed fused arrow introduced in Section~\ref{sec:colored-vertex-weights}.
        \item Black dots are boundary vertices.
        \item Rapidities and spin parameters are indicated next to the incoming edges. 
        At a boundary vertex, an incoming rapidity \(x\) becomes \(x^{-1}\) on the reflected edge.
        We omit rapidity and spin labels when their values are \(1\).
    \end{itemize}

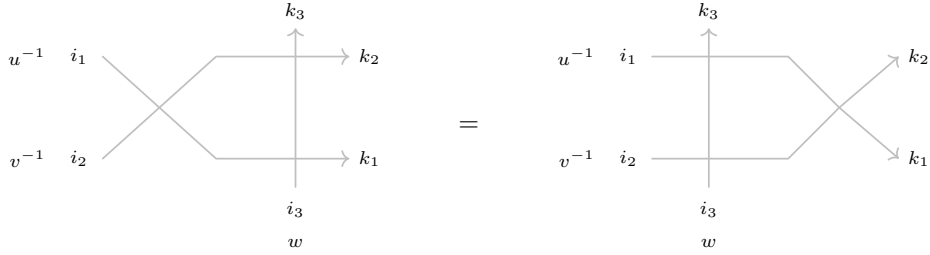
\begin{figure}[htbp]
\centering
\begin{tikzpicture}[
    scale=.75,
    baseline={([yshift=-.5ex]current bounding box.center)}
]
  \draw[paper wire,->]
      (0,2)--(1,1.1)--(2,.2)--(4.35,.2);
  \draw[paper wire,->]
      (0,.2)--(1,1.1)--(2,2)--(4.35,2);
  \draw[paper wire,->]
      (3.4,-.3)--(3.4,2.5);
  \node[left=2pt]  at (0,2)       {\scriptsize $i_1$};
  \node[left=18pt] at (0,2)       {\scriptsize $u^{-1}$};
  \node[left=2pt]  at (0,.2)      {\scriptsize $i_2$};
  \node[left=18pt] at (0,.2)      {\scriptsize $v^{-1}$};
  \node[below=2pt] at (3.4,-.3)   {\scriptsize $i_3$};
  \node[below=16pt] at (3.4,-.3)  {\scriptsize $w$};
  \node[right] at (4.35,.2)   {\scriptsize $k_1$};
  \node[right] at (4.35,2)    {\scriptsize $k_2$};
  \node[above] at (3.4,2.5)   {\scriptsize $k_3$};
\end{tikzpicture}
\hspace{8mm}$=$\hspace{8mm}
\begin{tikzpicture}[
    scale=.75,
    baseline={([yshift=-.5ex]current bounding box.center)}
]
  \draw[paper wire,->]
      (0,2)--(2.4,2)--(3.3,1.1)--(4.35,.2);
  \draw[paper wire,->]
      (0,.2)--(2.4,.2)--(3.3,1.1)--(4.35,2);
  \draw[paper wire,->]
      (1,-.3)--(1,2.5);
  \node[left=2pt]  at (0,2)       {\scriptsize $i_1$};
  \node[left=18pt] at (0,2)       {\scriptsize $u^{-1}$};
  \node[left=2pt]  at (0,.2)      {\scriptsize $i_2$};
  \node[left=18pt] at (0,.2)      {\scriptsize $v^{-1}$};
  \node[below=2pt] at (1,-.3)     {\scriptsize $i_3$};
  \node[below=16pt] at (1,-.3)    {\scriptsize $w$};
  \node[right] at (4.35,.2)   {\scriptsize $k_1$};
  \node[right] at (4.35,2)    {\scriptsize $k_2$};
  \node[above] at (1,2.5)     {\scriptsize $k_3$};
\end{tikzpicture}
\caption{The Yang--Baxter equation.}
\label{fig:elementary-colored-ybe}
\end{figure}

The standard fusion procedure extends the colored six-vertex model by
allowing multiple arrows on each edge; see, for example,
\cite{MR3477349,MR3782413,MR3959822,MR4518478,MR5008991}.
An edge obtained by fusing \(N\) lines carries an arrow configuration
\(A=(A_1,\ldots,A_n)\in\mathbb Z_{\ge0}^n\), where \(A_i\) is the number
of arrows of color \(i\). We write \(|A|:=A_1+\cdots+A_n\le N\) and
set \(A_0:=N-|A|\), the number of holes. The configuration consisting
of a single arrow of color \(i\) is denoted by \(E_i\), while \(E_0\)
denotes the empty configuration:
\[
E_0:=0,\qquad (E_i)_j:=\mathbf1_{\{i=j\}},\qquad 1\le i,j\le n.
\]

Before fusion, the colors on the \(N\) individual lines form a word
\(\boldsymbol c=(c_1,\ldots,c_N)\in\mathcal C_n^N\). Its corresponding
arrow configuration is
\[
|\boldsymbol c|:=(|\boldsymbol c|_1,\ldots,|\boldsymbol c|_n),\qquad
|\boldsymbol c|_i:=\#\{r:c_r=i\},\qquad 1\le i\le n.
\]
Thus we use bold lowercase letters for words and uppercase letters
for arrow configurations.

To construct the fused weights, we average over incoming words with
a fixed arrow configuration according to a \(q\)-exchangeable law.
This law assigns each word a weight determined by its number of
inversions with respect to the priority order:
\[
\operatorname{inv}_{\succ}(\boldsymbol c)
:=\#\{(r,s):1\le r<s\le N,\ c_r\succ c_s\}.
\]
For a fixed arrow configuration \(A\), the normalizing constant is
the \(q\)-multinomial coefficient
\[
\begin{bmatrix}N\\A_0,A_1,\ldots,A_n\end{bmatrix}_{\!q}
:=\frac{(q;q)_N}{\prod_{i=0}^n(q;q)_{A_i}}.
\]
The \(q\)-exchangeable law on words with arrow configuration \(A\)
is therefore
\begin{equation}
\label{eq:colored-q-exchangeable-law}
\pi_{N,A}(\boldsymbol c)
:=\frac{q^{\operatorname{inv}_{\succ}(\boldsymbol c)}}
{\begin{bmatrix}N\\A_0,A_1,\ldots,A_n\end{bmatrix}_{\!q}}
\mathbf1_{\{|\boldsymbol c|=A\}}.
\end{equation}

Fix \(L,M\ge1\). Place \(L\) vertical lines and \(M\) horizontal lines in a rectangle (see Figure~\ref{fig:colored-bulk-fusion-array}). Let $x$ be the rapidity of the rightmost column and $y$ that of the top row, and set $z=x/y$ as above. The vertical rapidities, from left to right, are $q^{L-1}x,q^{L-2}x,\ldots,x,$ and the horizontal rapidities, from top to bottom, are $y,q^{-1}y,\ldots,q^{1-M}y.$  Thus the vertex in row $r$, counted from the top, and column $c$, counted from the left, has spectral parameter $q^{L+r-c-1}z.$

\begin{figure}[htbp]
\centering

\begin{subfigure}[t]{.48\textwidth}
\centering
\begin{tikzpicture}[
  x=.72cm,y=.72cm,
  every node/.style={font=\scriptsize}
]
  \foreach \x in {1,2,3,4,5}
    {\draw[paper wire,->] (\x,0)--(\x,5.7);}
  \foreach \y in {1,2,3,4,5}
    {\draw[paper wire,->] (0,\y)--(5.7,\y);}

  \node[below] at (1,0) {$a_1$};
  \node[below] at (2,0) {$a_2$};
  \node at (3,-.2) {$\cdots$};
  \node at (4,-.2) {$\cdots$};
  \node[below] at (5,0) {$a_L$};
  \node[below=6mm,xshift=-2mm] at (1,0) {$q^{L-1}x$};
  \node[below=6mm,xshift=2mm] at (2,0) {$q^{L-2}x$};
  \node[below=6mm,xshift=5mm] at (3.2,0) {$\cdots$};
  \node[below=6mm] at (5,0) {$x$};

  \node[left] at (0,5) {$b_1$};
  \node[left] at (0,4) {$b_2$};
  \node at (-.25,3) {$\vdots$};
  \node at (-.25,2) {$\vdots$};
  \node[left] at (0,1) {$b_M$};
  \node[left=6mm] at (0,5) {$y$};
  \node[left=6mm] at (0,4) {$q^{-1}y$};
  \node[left=6mm] at (0,2.5) {$\vdots$};
  \node[left=6mm] at (0,1) {$q^{1-M}y$};

  \node[above] at (1,5.7) {$c_1$};
  \node[above] at (2,5.7) {$c_2$};
  \node at (3,5.9) {$\cdots$};
  \node at (4,5.9) {$\cdots$};
  \node[above] at (5,5.7) {$c_L$};
  \node[right] at (5.7,5) {$d_1$};
  \node[right] at (5.7,4) {$d_2$};
  \node at (5.95,3) {$\vdots$};
  \node at (5.95,2) {$\vdots$};
  \node[right] at (5.7,1) {$d_M$};
\end{tikzpicture}
\caption{The vertex model for the bulk fusion.}
\label{fig:colored-bulk-fusion-array}
\end{subfigure}
\hfill
\begin{subfigure}[t]{.48\textwidth}
\centering
\begin{tikzpicture}[
  x=.72cm,y=.72cm,
  every node/.style={font=\scriptsize}
]
  \foreach \r in {1,...,5}{
    \draw[paper wire,->]
      (\r,0)--(\r,\r)
      node[paper boundary] {}
      --(5.7,\r);
  }

  \node[below] at (1,0) {$a_1$};
  \node[below] at (2,0) {$a_2$};
  \node at (3,-.2) {$\cdots$};
  \node at (4,-.2) {$\cdots$};
  \node[below] at (5,0) {$a_L$};
  \node[below=6mm,xshift=-2mm] at (1,0) {$q^{L-1}x$};
  \node[below=6mm,xshift=2mm] at (2,0) {$q^{L-2}x$};
  \node[below=6mm,xshift=5mm] at (3,0) {$\cdots$};
  \node[below=6mm] at (5,0) {$x$};

  \node[right] at (5.7,1) {$b_L$};
  \node[right] at (5.7,2) {$b_{L-1}$};
  \node[right] at (5.7,3) {$\vdots$};
  \node[right] at (5.7,4) {$\vdots$};
  \node[right] at (5.7,5) {$b_1$};
\end{tikzpicture}
\caption{The vertex model  for the boundary fusion.}
\label{fig:colored-boundary-fusion-array}
\end{subfigure}

\caption{The vertex model used in finite bulk and boundary fusion with explicit rapidities.}
\label{fig:finite-fusion-arrays}
\end{figure}
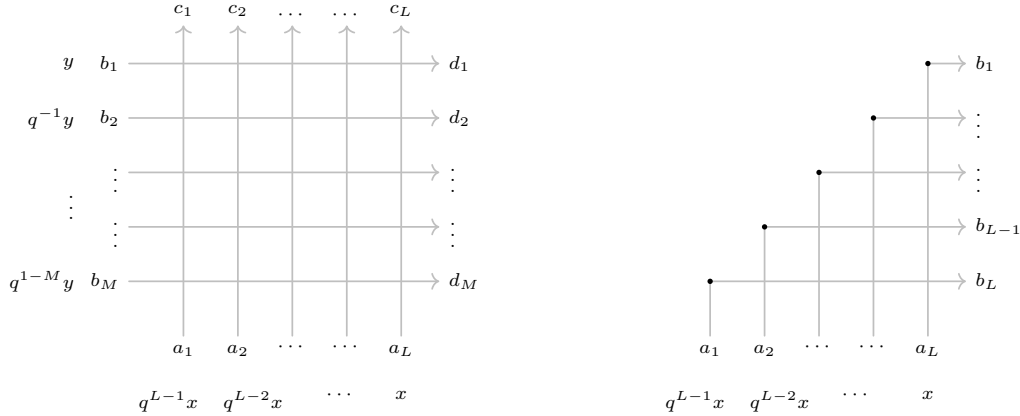

For boundary words
\(\boldsymbol a,\boldsymbol c\in\mathcal C_n^L\) and
\(\boldsymbol b,\boldsymbol d\in\mathcal C_n^M\), read in the orders
shown in Figure~\ref{fig:colored-bulk-fusion-array}, let
\[
\mathcal Z_z^{L,M}
(\boldsymbol a,\boldsymbol b;\boldsymbol c,\boldsymbol d)
\]
denote the partition function of the rectangle.

\begin{definition}
\label{def:finite-colored-bulk-fusion}
Let \(A,C\in\mathbb Z_{\ge0}^n\) satisfy \(|A|,|C|\le L\), and let
\(B,D\in\mathbb Z_{\ge0}^n\) satisfy \(|B|,|D|\le M\). Assume that
\[
A+B=C+D.
\]
Define
\begin{equation}
\label{eq:finite-colored-bulk-definition}
W_z^{L,M}(A,B;C,D)
:=
\sum_{\substack{
|\boldsymbol a|=A,\ |\boldsymbol b|=B\\
|\boldsymbol c|=C,\ |\boldsymbol d|=D}}
\pi_{L,A}(\boldsymbol a)\pi_{M,B}(\boldsymbol b)
\mathcal Z_z^{L,M}
(\boldsymbol a,\boldsymbol b;\boldsymbol c,\boldsymbol d).
\end{equation}
When \(A+B\ne C+D\), set \(W_z^{L,M}(A,B;C,D)=0\).
\end{definition}

The following identity shows that \(q\)-exchangeability is preserved
under fusion.

\begin{proposition}
\label{prop:finite-colored-bulk-q-exchangeability}
Let \(A,C\in\mathbb Z_{\ge0}^n\) satisfy \(|A|,|C|\le L\), and let
\(B,D\in\mathbb Z_{\ge0}^n\) satisfy \(|B|,|D|\le M\). For
\(\boldsymbol c\in\mathcal C_n^L\) and
\(\boldsymbol d\in\mathcal C_n^M\) with
\(|\boldsymbol c|=C\) and \(|\boldsymbol d|=D\), we have
\begin{equation}
\label{eq:finite-colored-bulk-q-exchangeability}
\sum_{\substack{|\boldsymbol a|=A\\|\boldsymbol b|=B}}
\pi_{L,A}(\boldsymbol a)\pi_{M,B}(\boldsymbol b)
\mathcal Z_z^{L,M}
(\boldsymbol a,\boldsymbol b;\boldsymbol c,\boldsymbol d)
=
W_z^{L,M}(A,B;C,D)
\pi_{L,C}(\boldsymbol c)\pi_{M,D}(\boldsymbol d).
\end{equation}
\end{proposition}

The fused weight has an explicit expression in terms of $q$-Hahn weights.  For $U,V\in\mathbb Z_{\ge0}^n$, write $U\leq V$ for
coordinatewise inequality.  When $U\leq V$, set
\begin{equation}
\label{eq:reversed-color-Phi}
\Phi_q(U,V;x,y):=
\left(\frac yx\right)^{|U|}
\frac{(x;q)_{|U|}(y/x;q)_{|V|-|U|}}{(y;q)_{|V|}}
q^{\sum_{1\le i<j\le n}U_i(V_j-U_j)}
\prod_{i=1}^n
\begin{bmatrix}V_i\\U_i\end{bmatrix}_{\!q}.
\end{equation}
Set $\Phi_q(U,V;x,y)=0$ otherwise. 

\begin{proposition}
\label{prop:finite-colored-bulk-fusion}
\begin{align}
W_z^{L,M}(A,B;C,D)
&=\mathbf1_{\{A+B=C+D\}}
 \left(\frac{q^{1-M}}z\right)^{|D|-|B|}
 q^{M|A|-L|D|}
\sum_{0\le P\le B\wedge C}
\notag\\
&\quad
\Phi_q\left(
 C-P,C+D-P;
 \frac{q^{1-L}}z,\frac{q^{1-L-M}}z
\right)
\Phi_q\left(P,B;\frac zq,q^{-M}\right).
\label{eq:finite-colored-bulk-fusion}
\end{align}
Here \(0\le P\le B\wedge C\) means
\(0\le P_i\le\min(B_i,C_i)\) for every \(i\).  Moreover,
\[
    \sum_{C,D}W_z^{L,M}(A,B;C,D)=1.
\]
\end{proposition}

Fix \(0<q<1\). For fixed arrow configurations \(A,B,C,D\), the
explicit formula \eqref{eq:finite-colored-bulk-fusion} depends
rationally on \(q^L\) and \(q^M\). We may therefore regard $s_1=q^L,s_2=q^M$ as independent complex
parameters associated with the vertical and horizontal edges,
respectively. This leads to the following definition.

\begin{definition}
\label{def:normal-continued-W}
For \(A,B,C,D\in\mathbb Z_{\ge0}^n\), define
\begin{align}
W_{z,s_1,s_2}(A,B;C,D)
&:=\mathbf1_{\{A+B=C+D\}}
 \left(\frac q{s_2z}\right)^{|D|-|B|}
 s_2^{|A|}s_1^{-|D|}
\sum_{0\le P\le B\wedge C}
\notag\\
&\quad
\Phi_q\left(C-P,C+D-P;
 \frac q{s_1z},\frac q{s_1s_2z}\right)
\Phi_q\left(P,B;\frac zq,s_2^{-1}\right).
\label{eq:normal-continued-W}
\end{align}
\end{definition}
Rational identities satisfied by the finite-spin weights extend
to these general parameters as well. Indeed, after clearing
denominators, such an identity becomes a polynomial identity in
\(s_1,s_2\), valid at \((s_1,s_2)=(q^L,q^M)\) for any \(L,M\). 
Since each variable independently takes
infinitely many distinct values, the identity holds identically.
Throughout this article, we refer to this argument as the
\emph{standard analytic continuation procedure} and use it without
repeating the justification.

As a special case of the previous definition, when the horizontal edge 
carries a single color and the vertical edge carries a normally indexed fused state, we write
\begin{equation}
\label{eq:def-colored-L-hor}
    L^{\mathrm{hor}}_{z,s}:=W_{z,s,q}.
\end{equation}
For \(A=(A_1,\ldots,A_n)\), write
\[
    |A|:=\sum_{r=1}^n A_r,
    \qquad
    A_{[k,\ell]}:=\sum_{r=k}^{\ell}A_r,
    \qquad
    A_{[k,\ell]}:=0\quad\text{if }k>\ell.
\]
Figure~\ref{fig:colored-L-hor-families} 
lists all six nonzero transition
families and their weights.  

\begin{figure}[htbp]
\centering
\begingroup
\newcommand{\HorConfig}[4]{%
\begin{tikzpicture}[baseline=(current bounding box.center)]
  \draw[gray!70,line width=3.8pt,line cap=round]
    (0,-.62)--(0,.62);
  \draw[gray!70,line width=1.5pt,line cap=round]
    (-.82,0)--(.82,0);
  \node[above=3pt] at (0,.62) {$#1$};
  \node[left=4pt]  at (-.82,0) {$#2$};
  \node[right=4pt] at (.82,0) {$#3$};
  \node[below=3pt] at (0,-.62) {$#4$};
\end{tikzpicture}}
\setlength{\tabcolsep}{0pt}
\setlength{\arrayrulewidth}{.35pt}
\renewcommand{\arraystretch}{1.15}
\begin{tabular}{|
  *{3}{>{\centering\arraybackslash}m{.31\textwidth}|}}
\hline
\rule[-1.12cm]{0pt}{2.30cm}\HorConfig{A}{0}{0}{A}
&
\HorConfig{A}{k}{k}{A}
&
\HorConfig{A_k^-}{0}{k}{A}
\\ \hline
\rule[-.45cm]{0pt}{1.10cm}
{\small\(\displaystyle\frac{q^{|A|}-sz}{1-sz}\)}
&
{\small\(\displaystyle
  \frac{q^{A_{[1,k-1]}}(1-q^{A_k}z)}{1-sz}\)}
&
{\small\(\displaystyle
  \frac{q^{A_{[1,k-1]}}(1-q^{A_k})}{1-sz}\)}
\\ \hline
\rule[-1.12cm]{0pt}{2.30cm}\HorConfig{A_k^+}{k}{0}{A}
&
\HorConfig{A_{k\ell}^{+-}}{k}{\ell}{A}
&
\HorConfig{A_{\ell k}^{+-}}{\ell}{k}{A}
\\ \hline
\rule[-.45cm]{0pt}{1.10cm}
{\small\(\displaystyle\frac{z(q^{|A|}-s)}{1-sz}\)}
&
{\small\(\displaystyle
  \frac{zq^{A_{[1,\ell-1]}}(1-q^{A_\ell})}{1-sz}\)}
&
{\small\(\displaystyle
  \frac{q^{A_{[1,k-1]}}(1-q^{A_k})}{1-sz}\)}
\\ \hline
\end{tabular}
\endgroup
\caption{The six nonzero transition families of
\(L^{\mathrm{hor}}_{z,s}\).  Here \(1\le k<\ell\le n\). $A_k^\pm:=A\pm E_k, A_{k\ell}^{+-}:=A+E_k-E_\ell.$}
\label{fig:colored-L-hor-families}
\end{figure}

The fully fused colored stochastic weights used here originate from the
stochastic $R$-matrix of Bosnjak-Mangazeev \cite{BosnjakMangazeev2016};
see Borodin-Wheeler \cite[Section~2 and Appendix~C]{MR4518478} for
their vertex-model formulation. 

\subsection{Colored boundary weights and finite fusion}
\label{subsec:finite-colored-boundary}
We begin with the unfused uncolored boundary weights of Garbali, de Gier, Mead, and Wheeler.

\begin{definition}[\cite{MR4924418}]
\label{def:basic-boundary-weights}
For $i,j\in\{0,1\}$, let $K_x^{a,c}(i;j)$ denote the boundary weight
with incoming state $i$ and outgoing reflected state $j$, where $a,c$
are boundary parameters. Set
\begin{equation}
\label{eq:def-h-ac}
    h_{a,c}(x):=\frac{ac(1-x^2)}{(a-x)(c-x)}.
\end{equation}
The four boundary weights are
\[
K_x^{a,c}(0;0)=1-h_{a,c}(x),
K_x^{a,c}(0;1)=h_{a,c}(x),
K_x^{a,c}(1;0)=-\frac{h_{a,c}(x)}{ac},
K_x^{a,c}(1;1)=1+\frac{h_{a,c}(x)}{ac}.
\]
\end{definition}

Throughout this paper, we use the creation-only specialization
$c\to\infty$. Write
\[
    K_x(i;j):=\lim_{c\to\infty}K_x^{a,c}(i;j).
\]
Taking this limit gives
\[
    K_x(0;0)=\kappa_0(x),\qquad
    K_x(0;1)=\kappa_1(x),\qquad
    K_x(1;0)=0,\qquad
    K_x(1;1)=1,
\]
where
\begin{equation}
\label{eq:kappa-boundary-functions}
    \kappa_0(x):=\frac{x(ax-1)}{a-x},
    \qquad
    \kappa_1(x):=\frac{a(1-x^2)}{a-x}.
\end{equation}

For the colored model, we allow every color other than $1$ to be replaced by color $1$. 

\begin{definition}
\label{def:elementary-colored-K}
For $i,j\in\mathcal C_n$, define
\begin{equation}
\label{eq:elementary-colored-K}
K_x(i;j)=
\begin{cases}
1, & i=j=1,\\[2pt]
\kappa_0(x), & i=j\ne1,\\[2pt]
\kappa_1(x), & i\ne1,\ j=1,\\[2pt]
0, & \text{otherwise}.
\end{cases}
\end{equation}
The boundary weights satisfy
\[
    \sum_{j\in\mathcal C_n}K_x(i;j)=1.
\]
\end{definition}

A half space analog of the Yang-Baxter equation is the follwoing reflection equation:

\begin{proposition}[Colored reflection equation]
\label{prop:elementary-colored-reflection}
For any $i,j,k,\ell\in\mathcal C_n$,
\begin{align}
&\sum_{\gamma_1,\gamma_2,\gamma_3,\gamma_4\in\mathcal C_n}
R_{x/y}(i,j;\gamma_1,\gamma_2)K_x(\gamma_1;\gamma_3)
R_{xy}(\gamma_2,\gamma_3;\gamma_4,k)K_y(\gamma_4;\ell)
\notag\\
&\qquad=
\sum_{\gamma_1,\gamma_2,\gamma_3,\gamma_4\in\mathcal C_n}
K_y(j;\gamma_1)R_{xy}(i,\gamma_1;\gamma_2,\gamma_3)
K_x(\gamma_2;\gamma_4)R_{x/y}(\gamma_3,\gamma_4;\ell,k).
\label{eq:elementary-colored-reflection}
\end{align}
\end{proposition}

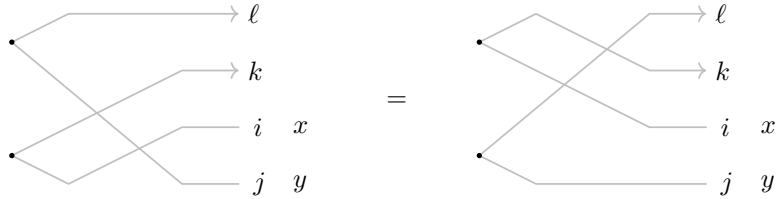
\begin{figure}[htbp]
\centering
\begin{tikzpicture}[
    scale=.75,
    baseline={([yshift=-.5ex]current bounding box.center)}
]
  \draw[paper wire]
      (4,1)--(3,1)--(1,0)--(0,.5);
  \draw[paper wire,->]
      (0,.5)--(3,2)--(4,2);
  \draw[paper wire]
      (4,0)--(3,0)--(0,2.5);
  \draw[paper wire,->]
      (0,2.5)--(1,3)--(4,3);
  \node[paper boundary] at (0,.5)  {};
  \node[paper boundary] at (0,2.5) {};
  \node[right=2pt] at (4,1) {$i$};
  \node[right=17pt] at (4,1) {$x$};
  \node[right=2pt] at (4,0) {$j$};
  \node[right=17pt] at (4,0) {$y$};
  \node[right] at (4,2) {$k$};
  \node[right] at (4,3) {$\ell$};
\end{tikzpicture}
\hspace{8mm}$=$\hspace{8mm}
\begin{tikzpicture}[
    scale=.75,
    baseline={([yshift=-.5ex]current bounding box.center)}
]
  \draw[paper wire]
      (4,1)--(3,1)--(0,2.5);
  \draw[paper wire,->]
      (0,2.5)--(1,3)--(3,2)--(4,2);
  \draw[paper wire]
      (4,0)--(1,0)--(0,.5);
  \draw[paper wire,->]
      (0,.5)--(3,3)--(4,3);
  \node[paper boundary] at (0,.5)  {};
  \node[paper boundary] at (0,2.5) {};
  \node[right=2pt] at (4,1) {$i$};
  \node[right=17pt] at (4,1) {$x$};
  \node[right=2pt] at (4,0) {$j$};
  \node[right=17pt] at (4,0) {$y$};
  \node[right] at (4,2) {$k$};
  \node[right] at (4,3) {$\ell$};
\end{tikzpicture}
\caption{The reflection equation.}
\label{fig:elementary-colored-reflection}
\end{figure}

Fix \(L\ge1\). Place \(L\) vertical lines with rapidities \(q^{L-1}x,q^{L-2}x,\ldots,x\), 
from left to right. Each line is reflected to the right at the diagonal boundary, 
where its rapidity is inverted. The \(r\)th line carries boundary weight \(K_{q^{L-r}x}\), 
and its crossing with a reflected line \(s<r\) carries bulk weight \(R_{q^{2L-r-s}x^2}\), 
following our column-to-row rapidity convention.

Let \(\mathcal K_x^L(\boldsymbol a;\boldsymbol b)\) denote the partition
function of Figure~\ref{fig:colored-boundary-fusion-array}, where
\(\boldsymbol a=(a_1,\ldots,a_L)\) is read from left to right along the
bottom and \(\boldsymbol b=(b_1,\ldots,b_L)\) is read from top to bottom
on the right.

\begin{definition}
\label{def:finite-colored-boundary-fusion}
For \(A,B\in\mathbb Z_{\ge0}^n\), with \(|A|,|B|\le L\), define
\begin{equation}
\label{eq:finite-colored-boundary-definition}
K_x^L(A;B):=
\sum_{(|\boldsymbol a|,|\boldsymbol b|)=(A,B)}
\pi_{L,A}(\boldsymbol a)
\mathcal K_x^L(\boldsymbol a;\boldsymbol b).
\end{equation}
\end{definition}

As in bulk fusion, conditioning on the outgoing arrow configuration
$B$ leaves a $q$-exchangeable word.

\begin{proposition}
\label{prop:finite-colored-boundary-q-exchangeability}
For arrow configurations \(A,B\) with \(|A|,|B|\leq L\) and every
outgoing word \(\boldsymbol b\) with \(|\boldsymbol b|=B\), we have
\begin{equation}
\label{eq:finite-colored-boundary-q-exchangeability}
\sum_{|\boldsymbol a|=A}
\pi_{L,A}(\boldsymbol a)
\mathcal K_x^L(\boldsymbol a;\boldsymbol b)
=
K_x^L(A;B)\pi_{L,B}(\boldsymbol b).
\end{equation}
The same $q$-exchangeability property allows us to fuse the $L$ incoming
lines and their reflected lines. The resulting boundary vertex is
described by the numbers of incoming and outgoing arrows of each color.
Its weight is $K_x^L(A;B)$, where $A$ and $B$ are the incoming and
outgoing arrow configurations, respectively.
\end{proposition}

To state the explicit weight, put
\[
    A_0:=L-|A|,
    \qquad
    B_0:=L-|B|,
    \qquad
    T:=L-B_1=B_0+\sum_{i=2}^nB_i,
\]
and
\begin{equation}
\label{eq:finite-boundary-Xi}
    \Xi(A,B)
    :=\sum_{2\le i<j\le n}(A_i-B_i)B_j
      +\sum_{i=2}^n(A_i-B_i)B_0.
\end{equation}

\begin{proposition}
\label{prop:finite-colored-boundary-fusion}
The weight vanishes unless \(B_r\le A_r\) for every
\(r\in\{0,2,\ldots,n\}\).  On this support,
\begin{equation}
\label{eq:finite-colored-boundary-fusion}
K_x^L(A;B)=
q^{A_1T+\binom T2+\Xi(A,B)}
\left(-\frac xa\right)^T
\frac{(ax;q)_T(q^{A_1+T}x^2;q)_{L-A_1-T}}
     {(q^{A_1}x/a;q)_{L-A_1}}
\prod_{r\in\{0,2,\ldots,n\}}
\begin{bmatrix}A_r\\B_r\end{bmatrix}_{\!q}.
\end{equation}
Moreover, the sum to one property is satisfied
\[
    \sum_BK_x^L(A;B)=1.
\]
\end{proposition}
\begin{remark}
\label{rem:finite-boundary-literature}
After translating the color and parameter conventions, this is He's
$t=0$ fused boundary weight \cite[Theorem~6.5]{MR5008991}.
\end{remark}

\subsection{Vertex model realization of Martin's construction}
\label{subsec:martin-yang-baxter}

Aggarwal, Nicoletti, and Petrov obtained Martin's construction from
the colored fused vertex model; see \cite[Sections~2.4 and~7.1]{MR4968006}.
An important notion in their work is the following \emph{queue specialization}.

\begin{lemma}
\label{lem:queue-specialization}
Assume \(A_r=A'_r=0\) for \(r>m\).
Fix the coordinates of \(A,A'\) away from \(m\), and let
their \(m\)th coordinates be finite during the limiting procedure.  For
\(i,j\in\{0,\ldots,m\}\), impose color conservation in the form
\(
A'_m=A_m+\mathbf1_{\{i=m\}}-\mathbf1_{\{j=m\}}
\).
Then, as \(A_m\to\infty\), both vertical states contain infinitely many
arrows of color \(m\), and the limit
\begin{equation}
\mathbb W^{(m)}_{z,s}(A,i;A',j)
:=\lim_{A_m\to\infty}
L^{\mathrm{hor}}_{z,s}(A,i;A',j)
\label{eq:def-queue-L}
\end{equation}
exists.  All nonzero values belong to the following six families:
\begingroup
\small
\begin{equation}
\label{eq:queue-L-families}
\begin{alignedat}{2}
\mathbb W^{(m)}_{z,s}(A,0;A,0)
 &=\frac{-sz}{1-sz},
&\qquad
\mathbb W^{(m)}_{z,s}(A,k;A+E_k,0)
 &=\frac{-sz}{1-sz},
\\[2mm]
\mathbb W^{(m)}_{z,s}(A,k;A,k)
 &=\frac{q^{A_{[1,k-1]}}(1-zq^{A_k})}{1-sz},
&
\mathbb W^{(m)}_{z,s}(A,0;A-E_k,k)
 &=\frac{q^{A_{[1,k-1]}}(1-q^{A_k})}{1-sz},
\\[2mm]
\mathbb W^{(m)}_{z,s}
 (A,k;A+E_k-E_\ell,\ell)
 &=\frac{zq^{A_{[1,\ell-1]}}(1-q^{A_\ell})}{1-sz},
&
\mathbb W^{(m)}_{z,s}
 (A,\ell;A+E_\ell-E_k,k)
 &=\frac{q^{A_{[1,k-1]}}(1-q^{A_k})}{1-sz}.
\end{alignedat}
\end{equation}
\endgroup
Here \(1\leq k<\ell\leq m\).  For each incoming state the weights sum to one. 
\end{lemma}

\begin{proof}
Take \(A_m\to\infty\) in the six weights of
Figure~\ref{fig:colored-L-hor-families}.  For \(0<q<1\), one has
\(q^{A_m}\to0\), so this gives \eqref{eq:queue-L-families}. 
\end{proof}

\begin{definition}
Fix \(N\geq1\) and parameters \(s_m<0\), \(0<z_m\leq1\),
\(1\leq m\leq n\).  The queue transfer matrix
\(\mathfrak Q^N_{\mathbf z,\mathbf s}\) is the transfer operator of the
vertex model with \(N\) rows and \(n\) columns in
Figure~\ref{fig:queue-transfer-matrix}: every vertex
in column \(m\) has weight \(\mathbb W^{(m)}_{z_m,s_m}\), and each matrix
element maps a bottom-and-left boundary configuration to a top-and-right
boundary configuration and is obtained by summing the products of local
weights over all internal edges.
We write simply \(\mathfrak Q^N\) when the parameters are understood.

For the queue construction, put color $0$ on every left boundary edge
and the infinitely many arrows of color m $E_m^\infty$ at the bottom of column $m$.
Let \(Q_m^N\) be the state at the top of column \(m\), and let
\(d^N(i)\) be the color on the right boundary in row \(i\), counted from the
top.  The operator together with these
boundary data induces measure on the top and right boundary
variables.
\end{definition}

\begin{figure}[htbp]
\centering
\begin{tikzpicture}[
  x=.82cm,
  y=.74cm,
  baseline={([yshift=-.5ex]current bounding box.center)},
  every node/.style={font=\scriptsize}
]
  
  \foreach \x in {1.35,2.90,4.45,7.60}
    \draw[paper wire,->] (\x,-.45)--(\x,5.65);

  \foreach \y in {.75,2.75,4.75}
    \draw[paper normal,->] (8.45,\y)--(0,\y);

  \node[paper label,anchor=south west] at (1.52,.87)
    {$\mathbb W^{(1)}_{z_1,s_1}$};

  \node[paper label,anchor=south west] at (3.07,2.87)
    {$\mathbb W^{(2)}_{z_2,s_2}$};

  \node[paper label,anchor=south east] at (7.43,4.87)
    {$\mathbb W^{(n)}_{z_n,s_n}$};

  \node[below=2pt] at (1.35,-.45) {$0$};
  \node[below=2pt] at (2.90,-.45) {$0$};
  \node[below=2pt] at (4.45,-.45) {$0$};
  \node[below=2pt] at (6.03,-.45) {$\cdots$};
  \node[below=2pt] at (7.60,-.45) {$0$};

  \node[above=2pt] at (1.35,5.65) {$d^N(1)$};
  \node[above=2pt] at (2.90,5.65) {$d^N(2)$};
  \node[above=2pt] at (4.45,5.65) {$d^N(3)$};
  \node[above=2pt] at (6.03,5.65) {$\cdots$};
  \node[above=2pt] at (7.60,5.65) {$d^N(N)$};

  \node[left=2pt] at (0,.75) {$Q_1^N$};
  \node[left=2pt] at (0,2.75) {$Q_2^N$};
  \node[left=2pt] at (0,3.75) {$\vdots$};
  \node[left=2pt] at (0,4.75) {$Q_n^N$};

  \node[right=2pt] at (8.45,.75)
    {$E_1^\infty\;(z_1,s_1)$};

  \node[right=2pt] at (8.45,2.75)
    {$E_2^\infty\;(z_2,s_2)$};

  \node[right=2pt] at (8.45,3.75)
    {$\vdots$};

  \node[right=2pt] at (8.45,4.75)
    {$E_n^\infty\;(z_n,s_n)$};
\end{tikzpicture}

\caption{The vertex model \(\mathfrak Q^N\). With the displayed empty bottom data and right-boundary \(E_m^\infty\) data, the model induces a probability measure on the top and left boundary once the stochastic condition is satisfied.}
\label{fig:queue-transfer-matrix}
\end{figure}
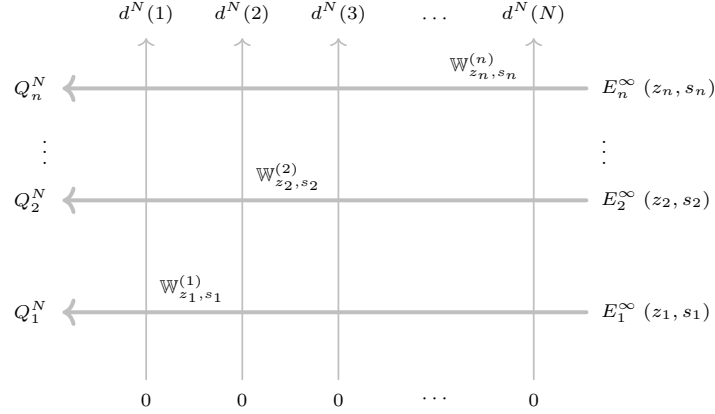

\begin{proposition}[\cite{MR4968006}, Proposition~7.4]
\label{prop:martin-yang-baxter-steady-state}
Fix \(0<\rho_1<\cdots<\rho_n<1\). For each finite \(N\), specialize
the vertex weights in Figure~\ref{fig:queue-transfer-matrix} by taking
\begin{equation}
\label{eq:martin-density-odds}
 z_m\downarrow0,
 \qquad
 s_mz_m=-\frac{1-\rho_m}{\rho_m},
 \qquad 1\leq m\leq n.
\end{equation}
The resulting vertex model realizes Martin's construction with service
densities \(\rho_1,\ldots,\rho_n\).

More precisely, for every fixed \(\ell\geq1\),
\begin{equation}
\label{eq:martin-anp-steady-state-convergence}
 \bigl(Q_1^N,\ldots,Q_n^N,d^N(1),\ldots,d^N(\ell)\bigr)
 \xrightarrow[N\to\infty]{\mathrm d}
 \bigl(Q_1,\ldots,Q_n,d(1),\ldots,d(\ell)\bigr),
\end{equation}
where the limit consists of the stationary queue states and their
departure process in Martin's construction. The limiting departure
process has the stationary full-space \(n\)-species ASEP law.

For \(n=2\), the limiting first-class queue length in the second
column satisfies
\begin{equation}
\label{eq:martin-M2-law}
 Q_2(1)\sim q\text{-}\mathrm{Poi}(r_{12}).
\end{equation}
\end{proposition}

\section{Hole-indexed continuation of boundary and bulk weights}
\label{sec:colored-vertex-weights}

To construct stationary measures for the multispecies ASEP on the half-line,
we use a hole-indexed analytic continuation of the finite fused bulk and
boundary weights.  On the reflected lines, this continuation keeps the
occupations of colors \(0,2,\ldots,n\) fixed, with color \(1\) treated as the
implicit background.  Related constructions involving dual vertex weights,
fusion, and analytic continuation appear, e.g., in Borodin-Wheeler
\cite{BorodinWheeler2021} and Bufetov-Mucciconi-Petrov
\cite{BufetovMucciconiPetrov2021}.  In this section, we derive explicit
formulas for the continued weights, establish nonnegativity and the
sum-to-one property for the boundary and spin-one bulk weights, and show
that the continued weights satisfy the Yang-Baxter and reflection
equations.  These weights and identities will be used in
Section~\ref{sec:queue-specialization-stationarity} to obtain the queue
specialization and construct the stationary measures.

Set
\[
    \widehat{\mathcal C}_n:=\mathcal C_n\setminus\{1\}.
\]
A hole-indexed state is a vector
\(
H=(H_0,H_2,\ldots,H_n)
\in\mathbb Z_{\geq0}^{\widehat{\mathcal C}_n}
\); we define \(|H|:=H_0+\sum_{i=2}^n H_i\).  For
\(r\in\widehat{\mathcal C}_n\), let \(\widehat E_r\) be the
corresponding standard basis vector and set \(\widehat E_1:=0\).
At finite capacity \(M\), the state \(H\) represents the ordinary
particle vector
\begin{equation}
\label{eq:hole-to-normal-embedding}
    (M-|H|,H_2,\ldots,H_n).
\end{equation}
Thus the unrecorded coordinate is the number of color-\(1\) arrows.  When
holes are listed separately, we refer to colors \(2,\ldots,n\) as the weak
colors.

\paragraph{Stochastic convention.}
We take \(0<q<1\), unless stated otherwise.  For stochasticity statements
in this and subsequent sections, we use the boundary parameter range
\begin{equation}
\label{eq:continued-K-stochastic-range}
 0<x<1,\qquad a<0,\qquad s<0,\qquad 0<\frac{sx}{a}<1,
\end{equation}
and the spin-one bulk parameter range
\begin{equation}
\label{eq:spin-one-bulk-stochastic-range}
 s<0,\qquad 0<z\leq1.
\end{equation}

\subsection{Continuation of the fused boundary weight}
\label{subsec:continued-boundary-weight}
We continue the finite boundary weights recalled in
Subsection~\ref{subsec:finite-colored-boundary}, using the hole coordinates
defined above.

The boundary maps a normally indexed state to a hole-indexed state.  Write
\(G=(G_1,\ldots,G_n)\) and
\(H=(H_0,H_2,\ldots,H_n)\).  Define
\begin{equation}
\label{eq:continued-K-normalization}
    \mathcal N^K_{x,s}
    :=\frac{(x^2;q)_\infty(sx/a;q)_\infty}
    {(x/a;q)_\infty(sx^2;q)_\infty}.
\end{equation}

\begin{proposition}
\label{prop:continued-colored-K}
For fixed \(G=(G_1,\ldots,G_n)\) and
\(H=(H_0,H_2,\ldots,H_n)\), consider, for \(L\geq |G|+|H|\), the
finite-spin boundary weights
\[
 K_x^L\bigl(G;(L-|H|,H_2,\ldots,H_n)\bigr).
\]
Regarded as values at \(s=q^L\), they admit an analytic continuation in
\(s\) wherever the expression below is defined.  The continued weight
vanishes unless \(H_i\le G_i\) for \(2\le i\le n\); on this support, it is
given by
\begin{align}
K_{x,s}(G;H)
&=\mathcal N^K_{x,s}
\frac{(x/a;q)_{G_1}(ax;q)_{|H|}(q^{|G|}/s;q)_{H_0}}
     {(x^2;q)_{G_1+|H|}(q;q)_{H_0}}
\left(\frac{sx}{a}\right)^{H_0}
\left(-\frac xa\right)^{H_{[2,n]}}
\notag\\
&\quad\times q^{\sum_{j=2}^n\left(G_{[1,j-1]}H_j+\binom{H_j}{2}\right)}
\prod_{i=2}^n\begin{bmatrix}G_i\\H_i\end{bmatrix}_{\!q}.
\label{eq:continued-colored-K}
\end{align}
At \(s=q^L\), with \(L\geq |G|\), the weight vanishes if
\(H_0>L-|G|\).
\end{proposition}

\begin{proof}
Fix \(L\ge |G|\) and set \(B=(L-|H|,H_2,\ldots,H_n)\), so that
\(B_0=H_0\).  On the support \(H_0\le L-|G|\) and \(H_i\le G_i\)
for \(i\ge2\), we have \(G_1+|H|\le L\).  Substituting \(A=G\) and
this \(B\) into Proposition~\ref{prop:finite-colored-boundary-fusion} gives
\begin{equation}
\label{eq:continued-K-finite-substitution}
K_x^L(G;B)=
q^{G_1|H|+\binom{|H|}{2}+\Xi(G,B)}
\left(-\frac xa\right)^{|H|}
\frac{(ax;q)_{|H|}(q^{G_1+|H|}x^2;q)_{L-G_1-|H|}}
     {(q^{G_1}x/a;q)_{L-G_1}}
\begin{bmatrix}L-|G|\\H_0\end{bmatrix}_{\!q}
\prod_{i=2}^n
\begin{bmatrix}G_i\\H_i\end{bmatrix}_{\!q}.
\end{equation}

First, the definition of \(\Xi\) and the identity \(B_0=H_0\) imply
\[
\Xi(G,B)
=
\sum_{2\le i<j\le n}(G_i-H_i)H_j
+H_0\sum_{i=2}^n(G_i-H_i).
\]
Consequently,
\begin{equation}
\label{eq:continued-K-exponent-splitting}
G_1|H|+\binom{|H|}{2}+\Xi(G,B)
=|G|H_0+\binom{H_0}{2}
+\sum_{j=2}^n\left(G_{[1,j-1]}H_j+\binom{H_j}{2}\right).
\end{equation}
Second,
\begin{equation}
\label{eq:continued-K-binomial-conversion}
\begin{bmatrix}L-|G|\\H_0\end{bmatrix}_{\!q}
q^{|G|H_0+\binom{H_0}{2}}
\left(-\frac xa\right)^{H_0}
=
\frac{(q^{|G|-L};q)_{H_0}}{(q;q)_{H_0}}
\left(\frac{q^Lx}{a}\right)^{H_0}.
\end{equation}
Finally,
\[
    \mathcal N^K_{x,q^L}
    =\frac{(x^2;q)_L}{(x/a;q)_L},
\]
and hence
\begin{equation}
\label{eq:continued-K-product-conversion}
\frac{(q^{G_1+|H|}x^2;q)_{L-G_1-|H|}}
     {(q^{G_1}x/a;q)_{L-G_1}}
=
\mathcal N^K_{x,q^L}
\frac{(x/a;q)_{G_1}}{(x^2;q)_{G_1+|H|}}.
\end{equation}
Combining
\eqref{eq:continued-K-finite-substitution}--%
\eqref{eq:continued-K-product-conversion}
gives \eqref{eq:continued-colored-K} at \(s=q^L\).
The support and specialization claims follow immediately from
\eqref{eq:continued-colored-K}.

The apparent singularity at \(s=0\) is removable, since
\begin{equation}
\label{eq:continued-K-removable-product}
(q^{|G|}/s;q)_{H_0}
\left(\frac{sx}{a}\right)^{H_0}
=
\left(\frac xa\right)^{H_0}
\prod_{r=0}^{H_0-1}(s-q^{|G|+r}).
\end{equation}
The right-hand side of \eqref{eq:continued-colored-K} is therefore analytic
at the neighborhood of \(s=0\), and the standard analytic continuation procedure completes the proof.
\end{proof}

The following proposition establishes stochasticity of the continued
boundary weights.

\begin{proposition}
\label{prop:continued-K-stochasticity}
Assume the boundary stochastic range
\eqref{eq:continued-K-stochastic-range}.
Then, for every normally indexed state \(G\),
\begin{equation}
\label{eq:continued-K-stochasticity}
 K_{x,s}(G;H)\geq0,
 \qquad
 \sum_{H}K_{x,s}(G;H)=1.
\end{equation}
The sum is absolutely convergent and locally uniformly convergent when
\(s\) ranges over compact subsets of the pole-free domain
\(|sx/a|<1\).
\end{proposition}

\begin{proof}
Nonnegativity follows directly from the signs of the factors in
\eqref{eq:continued-colored-K}.  For fixed \(G\), the weak output
coordinates range over a finite set, while the weights decay geometrically
in \(H_0\), locally uniformly on pole-free compact subsets of
\(|sx/a|<1\).  Together with the removable singularity at \(s=0\)
established above, this shows that the row sum is holomorphic.  At
\(s=q^L\), for all sufficiently large \(L\), it reduces to the finite-spin
boundary sum, which equals one.  The identity theorem therefore gives the
sum-to-one property on the pole-free component containing zero, which
contains the stated stochastic range.
\end{proof}

For the use in the boundary specializations below, we record the one- and
two-color forms explicitly.

\begin{example}
\label{ex:one-color-continued-boundary}
For \(n=1\), write \(g:=G_1\) and \(h:=H_0\).  Then
Proposition~\ref{prop:continued-colored-K} gives
\begin{equation}
\label{eq:one-color-continued-K}
K_{x,s}(g;h)
=
\mathcal N^K_{x,s}
\frac{(x/a;q)_g(ax;q)_h}{(x^2;q)_{g+h}}
\left(\frac{sx}{a}\right)^h
\frac{(q^g/s;q)_h}{(q;q)_h},
\qquad g,h\in\mathbb Z_{\geq0}.
\end{equation}
\end{example}

\begin{example}
\label{ex:two-color-continued-boundary}
For \(n=2\), Proposition~\ref{prop:continued-colored-K} reads
\begin{align}
K_{x,s}\bigl((G_1,G_2);(H_0,H_2)\bigr)
&=\mathcal N^K_{x,s}
\frac{(x/a;q)_{G_1}(ax;q)_{H_0+H_2}}
     {(x^2;q)_{G_1+H_0+H_2}}
\left(\frac{sx}{a}\right)^{H_0}
\left(-\frac xa\right)^{H_2}
\notag\\
&\quad\times
q^{G_1H_2+\binom{H_2}{2}}
\begin{bmatrix}G_2\\H_2\end{bmatrix}_{\!q}
\frac{(q^{G_1+G_2}/s;q)_{H_0}}{(q;q)_{H_0}}.
\label{eq:two-color-continued-K}
\end{align}
\end{example}

\begin{remark}
Identifying color \(2\) with the hole \(0\) amounts to fixing
\(h=H_0+H_2\) and summing \eqref{eq:two-color-continued-K} over all
compatible pairs \((H_0,H_2)\).  With \(g=G_1\), the resulting weight is
\eqref{eq:one-color-continued-K} and does not depend on \(G_2\).
\end{remark}

\subsection{The horizontal hole-index continuation}
We next apply the hole-indexed continuation to the fused bulk weights.

Let
\[
G=(G_1,\ldots,G_n),
\quad
G'=(G'_1,\ldots,G'_n),
\quad
H=(H_0,H_2,\ldots,H_n),
\quad
H'=(H'_0,H'_2,\ldots,H'_n).
\]
Given \(P=(P_1,\ldots,P_n)\), write \(p:=|P|\).  Define
\begin{equation}
\label{eq:hat-W-normalization}
    \mathcal N_{z,s_1,s_2}
    :=\frac{(z;q)_\infty(s_1s_2z;q)_\infty}
    {(s_1z;q)_\infty(s_2z;q)_\infty}.
\end{equation}

\begin{proposition}
\label{prop:hole-continued-W}
For fixed exterior states, the finite fused weights obtained from the
finite-capacity representation \eqref{eq:hole-to-normal-embedding} admit an
analytic continuation in the spin variables, with \(s_1=q^L\) and
\(s_2=q^M\) at finite spin.  Denote this weight by
\(
\widehat W_{z,s_1,s_2}
(G,H;G',H')
\).  It vanishes unless
\begin{equation}
\label{eq:hat-W-conservation}
    G_1+|H'|=G'_1+|H|,
    \qquad
    G_i+H_i=G'_i+H'_i,
    \quad 2\le i\le n.
\end{equation}
When \eqref{eq:hat-W-conservation} holds, it is given by
\begin{align}
&\widehat W_{z,s_1,s_2}(G,H;G',H')
\notag\\
&\quad=\mathcal N_{z,s_1,s_2}
\left(s_1z\right)^{H'_0}
\frac{(q;q)_{H_0}}{(q;q)_{H'_0}}
\times\sum_{P}(-1)^{p+H'_0}
q^{H_0p+\binom{H'_0}{2}+\binom{p+1}{2}}
\frac{(z/q;q)_p}{(z;q)_{H_0+p}}
\notag\\
&\qquad\times
\left(\frac{q^{1+|G|-H_0-p}}{s_1z};q\right)_{H'_0}
q^{\sum_{1\le i<j\le n}\left[(G'_i-P_i)H'_j+P_i(H_j-P_j)\right]}
\notag\\
&\qquad\times
\frac{(s_2q^{1-|H'|};q)_{G'_1-P_1}
      (s_2q^{1-|H|-P_1};q)_{P_1}}
     {(q;q)_{G'_1-P_1}(q;q)_{P_1}}
\prod_{i=2}^n
\begin{bmatrix}G'_i+H'_i-P_i\\G'_i-P_i\end{bmatrix}_{\!q}
\begin{bmatrix}H_i\\P_i\end{bmatrix}_{\!q}.
\label{eq:hole-continued-W}
\end{align}
Here the finite sum is over
\begin{equation}
\label{eq:hat-W-summation-range}
    0\le P_1\le G'_1,
    \qquad
    0\le P_i\le\min(H_i,G'_i),
    \quad 2\le i\le n.
\end{equation}
\end{proposition}

\begin{proof}
Substitute
\[
B=(M-|H|,H_2,\ldots,H_n),
\qquad
D=(M-|H'|,H'_2,\ldots,H'_n)
\]
into \eqref{eq:finite-colored-bulk-fusion}.  Its summation vector remains
finite.  The original first-coordinate bound is
\[
    P_1\le\min(M-|H|,G'_1).
\]
After the scalar factors have been rearranged, this range may be extended
to \(P_1\le G'_1\): at \(s_2=q^M\), every added term has
\[
    (s_2q^{1-|H|-P_1};q)_{P_1}
    =(q^{M+1-|H|-P_1};q)_{P_1}=0.
\]
For \(i\ge2\), the bound is
\(P_i\le\min(H_i,G'_i)\).  The two color-order exponents
and the associated Gaussian binomial coefficients give
\[
\begin{aligned}
&q^{\sum_{1\le i<j\le n}\left[(G'_i-P_i)H'_j+P_i(H_j-P_j)\right]}
\\
&\quad\times
\frac{(q^{M+1-|H'|};q)_{G'_1-P_1}
      (q^{M+1-|H|-P_1};q)_{P_1}}
     {(q;q)_{G'_1-P_1}(q;q)_{P_1}}
\prod_{i=2}^n
\begin{bmatrix}G'_i+H'_i-P_i\\G'_i-P_i\end{bmatrix}_{\!q}
\begin{bmatrix}H_i\\P_i\end{bmatrix}_{\!q}.
\end{aligned}
\]

Using the identities
\begin{equation}
\label{eq:negative-q-Pochhammer-identities}
(q^{-N};q)_r
=(-1)^rq^{-Nr+\binom r2}
  \frac{(q;q)_N}{(q;q)_{N-r}},
\quad
(u;q)_{r+t}=(u;q)_r(uq^r;q)_t,
\quad
\frac{(u;q)_\infty}{(uq^r;q)_\infty}=(u;q)_r,
\end{equation}
we obtain \eqref{eq:hole-continued-W} at finite spins.  The apparent \(s_1=0\)
singularity is removable term by term, because
\[
 (s_1z)^{H'_0}
 \left(\frac{q^{1+|G|-H_0-p}}{s_1z};q\right)_{H'_0}
 =\prod_{r=0}^{H'_0-1}
   \left(s_1z-q^{1+|G|-H_0-p+r}\right).
\]
Since the sum in \eqref{eq:hole-continued-W} is finite, the standard
analytic continuation procedure proves the claim.
\end{proof}

The following specializations are used later in the one-sided and
mixed-crossing computations.

\begin{example}
We first specialize the bottom--top edge to spin one while retaining a
hole-indexed fused left--right edge:
\begin{equation}
\label{eq:def-colored-L-ver}
    \widehat L^{\mathrm{ver}}_{z,s}:=\widehat W_{z,q,s}.
\end{equation}
For a hole-indexed state \(H=(H_0,H_2,\ldots,H_n)\), write
\[
    H_k^\pm:=H\pm\widehat E_k,
    \qquad
    H_{k\ell}^{+-}:=H+\widehat E_k-\widehat E_\ell.
\]
Figure~\ref{fig:colored-L-ver-families} lists all six nonzero transition
families and their weights. 
\end{example}

\begin{figure}[htbp]
\centering
\begingroup
\newcommand{\VerConfig}[4]{%
\begin{tikzpicture}[baseline=(current bounding box.center)]
  \draw[gray!70,line width=1.5pt,line cap=round]
    (0,-.62)--(0,.62);
  \draw[holecolor,line width=3.8pt,line cap=round]
    (-.82,0)--(.82,0);
  \node[above=3pt] at (0,.62) {$#1$};
  \node[left=4pt,text=holecolor] at (-.82,0) {$#2$};
  \node[right=4pt,text=holecolor] at (.82,0) {$#3$};
  \node[below=3pt] at (0,-.62) {$#4$};
\end{tikzpicture}}
\setlength{\tabcolsep}{0pt}
\setlength{\arrayrulewidth}{.35pt}
\renewcommand{\arraystretch}{1.15}
\begin{tabular}{|
  *{3}{>{\centering\arraybackslash}m{.31\textwidth}|}}
\hline
\rule[-1.12cm]{0pt}{2.30cm}\VerConfig{1}{H}{H}{1}
&
\VerConfig{k}{H}{H}{k}
&
\VerConfig{k}{H}{H_k^-}{1}
\\ \hline
\rule[-.45cm]{0pt}{1.10cm}
{\small\(\displaystyle\frac{q^{|H|}-sz}{1-sz}\)}
&
{\small\(\displaystyle
  \frac{q^{\sum_{r\prec k}H_r}(1-q^{H_k}z)}{1-sz}\)}
&
{\small\(\displaystyle
  \frac{q^{\sum_{r\prec k}H_r}(1-q^{H_k})}{1-sz}\)}
\\ \hline
\rule[-1.12cm]{0pt}{2.30cm}\VerConfig{1}{H}{H_k^+}{k}
&
\VerConfig{\ell}{H}{H_{k\ell}^{+-}}{k}
&
\VerConfig{k}{H}{H_{\ell k}^{+-}}{\ell}
\\ \hline
\rule[-.45cm]{0pt}{1.10cm}
{\small\(\displaystyle\frac{z(q^{|H|}-s)}{1-sz}\)}
&
{\small\(\displaystyle
  \frac{zq^{\sum_{r\prec\ell}H_r}(1-q^{H_\ell})}
       {1-sz}\)}
&
{\small\(\displaystyle
  \frac{q^{\sum_{r\prec k}H_r}(1-q^{H_k})}{1-sz}\)}
\\ \hline
\end{tabular}
\endgroup
\caption{The six nonzero transition families of
\(\widehat L^{\mathrm{ver}}_{z,s}\).  Here
\(k,\ell\in\widehat{\mathcal C}_n\) and \(k\prec\ell\).}
\label{fig:colored-L-ver-families}
\end{figure}

The following proposition establishes stochasticity of the spin-one
weights displayed above.

\begin{proposition}
\label{prop:hat-L-ver-stochasticity}
Assume the spin-one bulk stochastic range
\eqref{eq:spin-one-bulk-stochastic-range}.
Then the weights in Figure~\ref{fig:colored-L-ver-families} are
nonnegative and, for every input color \(i\) and hole-indexed state
\(H\),
\begin{equation}
\label{eq:hat-L-ver-stochasticity}
 \sum_{j,H'}
 \widehat L^{\mathrm{ver}}_{z,s}
 (i,H;j,H')=1.
\end{equation}
The sum is finite.  The same statement holds at \(q=0\) by continuity.
\end{proposition}

\begin{proof}
All six families are nonnegative in the stated range.  Conservation
determines \(H'\) from the output color \(j\), so the row sum contains at
most \(n+1\) terms and is rational in \(s\).  The sum-to-one property
follows from that of the finite bulk weights in
Proposition~\ref{prop:finite-colored-bulk-fusion} and the standard analytic
continuation procedure.  The statement at \(q=0\) follows by continuity.
\end{proof}

\begin{example}
\label{ex:two-color-hole-spin-one}
Let \(1\succ2\succ0\), and write
\(H=(H_0,H_2)\).  With the input and output
colors ordered as \(0,2,1\), the weights
\(\widehat L^{\mathrm{ver}}_{z,s}\) are
\begingroup
\renewcommand{\arraystretch}{1.65}
\begin{equation}
\label{eq:two-color-L-ver-table}
\begin{array}{c|ccc}
i\backslash j&0&2&1\\ \hline
0&\dfrac{1-q^{H_0}z}{1-sz}
 &\dfrac{q^{H_0}z(1-q^{H_2})}{1-sz}
 &\dfrac{z(q^{|H|}-s)}{1-sz}\\[7pt]
2&\dfrac{1-q^{H_0}}{1-sz}
 &\dfrac{q^{H_0}(1-q^{H_2}z)}{1-sz}
 &\dfrac{z(q^{|H|}-s)}{1-sz}\\[7pt]
1&\dfrac{1-q^{H_0}}{1-sz}
 &\dfrac{q^{H_0}(1-q^{H_2})}{1-sz}
 &\dfrac{q^{|H|}-sz}{1-sz}.
\end{array}
\end{equation}
\endgroup
Here
\(H'=H_{ij}^{+-}\), where
\(\widehat E_0=(1,0)\),
\(\widehat E_2=(0,1)\), and
\(\widehat E_1=(0,0)\).
\end{example}

\begin{example}
\label{ex:two-color-doubly-fused-hole}
For \(n=2\), write
\[
G=(G_1,G_2),\qquad
H=(H_0,H_2),\qquad
G'=(G'_1,G'_2),\qquad
H'=(H'_0,H'_2).
\]
The weight vanishes unless
\begin{equation}
\label{eq:two-color-doubly-fused-conservation}
G_1+|H'|=G'_1+|H|,
\qquad
G_2+H_2=G'_2+H'_2.
\end{equation}
On this support,
\begin{align}
&\widehat W_{z,s_1,s_2}
 \bigl(G,H;G',H'\bigr)
\notag\\
&\quad =
\mathcal N_{z,s_1,s_2}(s_1z)^{H'_0}
\frac{(q;q)_{H_0}}{(q;q)_{H'_0}}
\sum_{P_1=0}^{G'_1}
\sum_{P_2=0}^{\min(H_2,G'_2)}
(-1)^{p+H'_0}
q^{H_0p+\binom{H'_0}{2}+\binom{p+1}{2}}
\frac{(z/q;q)_p}{(z;q)_{H_0+p}}
\notag\\
&\quad\times
\left(\frac{q^{1+|G|-H_0-p}}{s_1z};q\right)_{H'_0}
q^{(G'_1-P_1)H'_2+P_1(H_2-P_2)}
\notag\\
&\quad\times
\begin{bmatrix}G'_2+H'_2-P_2\\G'_2-P_2\end{bmatrix}_{\!q}
\begin{bmatrix}H_2\\P_2\end{bmatrix}_{\!q}
\times
\frac{(s_2q^{1-|H'|};q)_{G'_1-P_1}}
     {(q;q)_{G'_1-P_1}}
\frac{(s_2q^{1-|H|-P_1};q)_{P_1}}
     {(q;q)_{P_1}}.
\label{eq:22-fully-fused-weight}
\end{align}
\end{example}

\subsection{Continued Yang-Baxter and reflection equations}
\label{subsec:continued-colored-identities}

Recall that, in the normalization of
Figure~\ref{fig:colored-bulk-fusion-array}, the local parameter in row
\(r\) and column \(c\) is \(q^{L+r-c-1}z\).  Consequently, the three
pairwise spectral parameters in the fully fused Yang--Baxter equation
below are \(qy/(s_2x)\), \(qy/s_2\), and \(qx/s_1\).

\begin{proposition}
\label{prop:continued-colored-ybe}
First, let
\(
A,B,I_1,I_2,C_1,C_2\in\mathbb Z_{\ge0}^n
\)
be normally indexed states.  Then
\begin{align}
&\sum_{X,X_1,X_2}
W_{qy/(s_2x),s_1,s_2}(I_1,I_2;X_1,X_2)
W_{qy/s_2,s_3,s_2}(A,X_2;X,C_2)
W_{qx/s_1,s_3,s_1}(X,X_1;B,C_1)
\notag\\
&\qquad=
\sum_{X,X_1,X_2}
W_{qx/s_1,s_3,s_1}(A,I_1;X,X_1)
W_{qy/s_2,s_3,s_2}(X,I_2;B,X_2)
W_{qy/(s_2x),s_1,s_2}(X_1,X_2;C_1,C_2).
\label{eq:continued-normal-colored-ybe}
\end{align}
Here \(X,X_1,X_2\in\mathbb Z_{\ge0}^n\).

Second, let \(A,B,I_1,C_1\) be normally indexed states and let
\(I_2,C_2\) be hole-indexed states.  Then
\begin{align}
&\sum_{X,X_1,X_2}
\widehat W_{qy/(s_2x),s_1,s_2}(I_1,I_2;X_1,X_2)
\widehat W_{qy/s_2,s_3,s_2}(A,X_2;X,C_2)
W_{qx/s_1,s_3,s_1}(X,X_1;B,C_1)
\notag\\
&\qquad=
\sum_{X,X_1,X_2}
W_{qx/s_1,s_3,s_1}(A,I_1;X,X_1)
\widehat W_{qy/s_2,s_3,s_2}(X,I_2;B,X_2)
\widehat W_{qy/(s_2x),s_1,s_2}(X_1,X_2;C_1,C_2).
\label{eq:continued-colored-ybe-components}
\end{align}
In the second identity, \(X,X_1\) are normally indexed and \(X_2\) is
hole-indexed.  All sums are finite by the conservation laws.
\end{proposition}

\begin{figure}[htbp]
\centering
\begingroup
\begin{tikzpicture}[
  x=.74cm,y=.74cm,
  baseline={([yshift=-.5ex]current bounding box.center)},
  identity state/.style={font=\scriptsize,inner sep=.7pt}
]
  \begin{scope}
    \draw[paper normal,->]
      (0,.2)--(1,1.1)--(2,2)--(4.4,2);
    \draw[paper hole,->]
      (0,2)--(1,1.1)--(2,.2)--(4.4,.2);
    \draw[paper normal,->]
      (3.4,-.3)--(3.4,2.5);
    \node[identity state,left=2pt] at (0,.2) {$I_1$};
    \node[identity state,left=22pt] at (0,.2) {$(x,s_1)$};
    \node[identity state,left=2pt] at (0,2) {$I_2$};
    \node[identity state,left=22pt] at (0,2) {$(y,s_2)$};
    \node[identity state,below=2pt] at (3.4,-.3) {$A$};
    \node[identity state,below=16pt] at (3.4,-.3) {$(1,s_3)$};
    \node[identity state,right=2pt] at (4.4,2) {$C_1$};
    \node[identity state,right=2pt] at (4.4,.2) {$C_2$};
    \node[identity state,above=2pt] at (3.4,2.5) {$B$};
  \end{scope}

  \node[font=\normalsize] at (6.35,1.1) {$=$};
    
  \begin{scope}[shift={(9.0,0)}]
    \draw[paper normal,->]
      (1,-.3)--(1,2.5);
    \draw[paper normal,->]
      (0,.2)--(2.4,.2)--(3.3,1.1)--(4.4,2);
    \draw[paper hole,->]
      (0,2)--(2.4,2)--(3.3,1.1)--(4.4,.2);
    \node[identity state,left=2pt] at (0,.2) {$I_1$};
    \node[identity state,left=22pt] at (0,.2) {$(x,s_1)$};
    \node[identity state,left=2pt] at (0,2) {$I_2$};
    \node[identity state,left=22pt] at (0,2) {$(y,s_2)$};
    \node[identity state,below=2pt] at (1,-.3) {$A$};
    \node[identity state,below=16pt] at (1,-.3) {$(1,s_3)$};
    \node[identity state,right=2pt] at (4.4,2) {$C_1$};
    \node[identity state,right=2pt] at (4.4,.2) {$C_2$};
    \node[identity state,above=2pt] at (1,2.5) {$B$};
  \end{scope}
\end{tikzpicture}
\endgroup
\caption{The Yang-Baxter equation for continued vertices \eqref{eq:continued-colored-ybe-components}.}
\label{fig:continued-colored-ybe}
\end{figure}
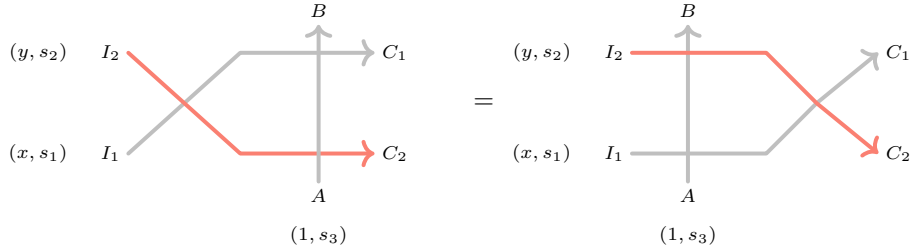

The continued bulk and boundary weights also satisfy the following
reflection equation.

\begin{proposition}
\label{prop:continued-colored-reflection}
Let \(G\) be an incoming normally indexed state, \(i\) an incoming color,
\(H\) an outgoing hole-indexed state, and \(j\) an outgoing
color.  Then
\begin{align}
&\sum_{\gamma_1,\gamma_2,G^{(1)},H^{(1)}}
L^{\mathrm{hor}}_{x/y,s}(G,i;G^{(1)},\gamma_1)\,
K_{x,s}(G^{(1)};H^{(1)})\,
\widehat L^{\mathrm{ver}}_{xy,s}
(\gamma_1,H^{(1)};\gamma_2,H)\,
K_y(\gamma_2;j)
\notag\\
&\qquad=
\sum_{\gamma_1,\gamma_2,G^{(1)},H^{(1)}}
K_y(i;\gamma_1)\,
L^{\mathrm{hor}}_{xy,s}(G,\gamma_1;G^{(1)},\gamma_2)\,
K_{x,s}(G^{(1)};H^{(1)})\,
\widehat L^{\mathrm{ver}}_{x/y,s}
(\gamma_2,H^{(1)};j,H).
\label{eq:continued-colored-reflection-components}
\end{align}
The sums are finite by conservation.
\end{proposition}

\begin{proof}[Proof of Propositions~\ref{prop:continued-colored-ybe}
and~\ref{prop:continued-colored-reflection}]
The identities follow from the elementary Yang--Baxter and reflection
equations, the \(q\)-exchangeability statements in
Propositions~\ref{prop:finite-colored-bulk-q-exchangeability}
and~\ref{prop:finite-colored-boundary-q-exchangeability}, and the standard
analytic continuation procedure.
\end{proof}

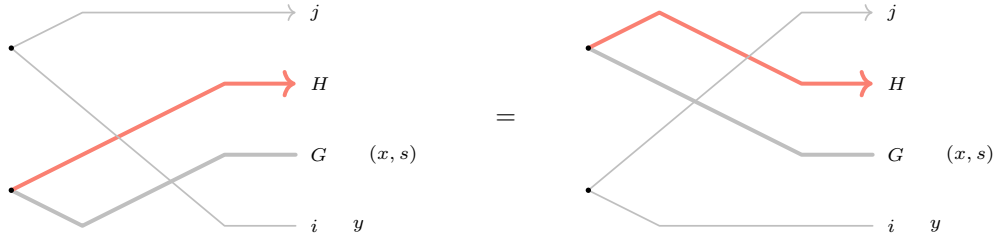
\begin{figure}[htbp]
\centering
\begingroup
\begin{tikzpicture}[scale=.94,baseline={([yshift=-.5ex]current bounding box.center)}]
  \draw[paper normal]
    (4,1)--(3,1)--(1,0)--(0,.5);
  \draw[paper hole,->]
    (0,.5)--(3,2)--(4,2);
  \draw[paper wire,->]
    (4,0)--(3,0)--(0,2.5)--(1,3)--(4,3);
  \node[paper boundary] at (0,.5)  {};
  \node[paper boundary] at (0,2.5) {};
  \node[right=2pt] at (4,0) {\scriptsize \(i\)};
  \node[right=18pt] at (4,0) {\scriptsize \(y\)};
  \node[right=2pt] at (4,1) {\scriptsize \(G\)};
  \node[right=24pt] at (4,1) {\scriptsize \((x,s)\)};
  \node[right=2pt] at (4,2) {\scriptsize \(H\)};
  \node[right=2pt] at (4,3) {\scriptsize \(j\)};
\end{tikzpicture}
\hspace{8mm}\(=\)\hspace{8mm}
\begin{tikzpicture}[scale=.94,baseline={([yshift=-.5ex]current bounding box.center)}]
  \draw[paper normal]
    (4,1)--(3,1)--(0,2.5);
  \draw[paper hole,->]
    (0,2.5)--(1,3)--(3,2)--(4,2);
  \draw[paper wire,->]
    (4,0)--(1,0)--(0,.5)--(3,3)--(4,3);
  \node[paper boundary] at (0,.5)  {};
  \node[paper boundary] at (0,2.5) {};
  \node[right=2pt] at (4,0) {\scriptsize \(i\)};
  \node[right=18pt] at (4,0) {\scriptsize \(y\)};
  \node[right=2pt] at (4,1) {\scriptsize \(G\)};
  \node[right=24pt] at (4,1) {\scriptsize \((x,s)\)};
  \node[right=2pt] at (4,2) {\scriptsize \(H\)};
  \node[right=2pt] at (4,3) {\scriptsize \(j\)};
\end{tikzpicture}
\endgroup
\caption{The reflection equation for continued vertices
\eqref{eq:continued-colored-reflection-components}.}
\label{fig:continued-colored-reflection}
\end{figure}

\section{Queue specialization and stationarity}
\label{sec:queue-specialization-stationarity}

In this section we prove Theorem~\ref{thm:main-stationary-construction}.
The general idea behind the proof is essentially the same as in \cite{MR4968006}; it is short and follows a well-known pattern. Namely, the Yang-Baxter equation provides an ''engine'' --- a vertex model of a particular form --- that commutes with the ASEP generator. Constructing the engine in a right way, we can force it to produce a nontrivial output from the trivial stationary measure of ASEP (such as the completely filled line by first-class particles). Due to commutation, this ensures that the nontrivial output is also a stationary measure of ASEP. Thus, the engine provides non-trivial stationary measures from trivial ones. 

While our proof fully follows this idea, we were not able to find its elegant and short form. The main technical difficulty is that the commutation relations seem to work better for finite volumes, while we are constructing a genuinely infinite object. This explains why Theorem \ref{thm:intro-vertex-form} has an asymptotic form, and a significant effort has to be devoted to prove the stabilization of the limiting procedure. It would be interesting to realize an idea above in a more transparent and shorter way.  

The proof in this section has the following structure. In Subsection \ref{subsec:half-space-ASEP-double-row} we recall the vertex model (fundamental double row), which has an ASEP limit. In Subsection \ref{subsec:boundary-queue-specialization} we introduce a queue specialization of vertices from Section \ref{sec:colored-vertex-weights}, recovering weights from Figures \ref{fig:intro-normal-weights}, \ref{fig:intro-return-weights}, and \ref{fig:intro-boundary-weights}. The next Subsections record their properties, which survive the queue specialization. The key Subsection \ref{subsec:stable-engine-carriers} establishes the stabilization of our vertex model / queue construction, as the size of the model goes to infinity. Finally, in the last subsection we prove stationarity, following the general idea above. 

\subsection{The fundamental double row and the half-space ASEP limit}
\label{subsec:half-space-ASEP-double-row}

Set \(\Omega_n=\mathcal C_n^{\mathbb N}\).  For \(\eta\in\Omega_n\), let \(\eta^{j,j+1}\) exchange
sites \(j,j+1\), and let \(\eta^{1\to1}\) replace the first color by
\(1\).  Write
\begin{equation}
\label{eq:colored-ASEP-rate}
 \mathfrak c_q(i,j)=\mathbf1_{\{i\succ j\}}+
                       q\mathbf1_{\{j\succ i\}}.
\end{equation}
The process defined in Subsection~\ref{subsec:intro-stationary-construction}
has the following generator on cylinder functions:
\begin{equation}
 (\mathcal L_\alpha f)(\eta)
 =\sum_{j\geq1}\mathfrak c_q(\eta_j,\eta_{j+1})
       \bigl(f(\eta^{j,j+1})-f(\eta)\bigr)
  +(1-q)\alpha\mathbf1_{\{\eta_1\ne1\}}
       \bigl(f(\eta^{1\to1})-f(\eta)\bigr).
 \label{eq:multispecies-half-space-ASEP-generator}
\end{equation}
Its Harris construction defines a Feller semigroup
\begin{equation}
\label{eq:ASEP-semigroup}
 (\mathfrak P_t f)(\eta)=\mathbb E_\eta[f(\eta(t))].
\end{equation}
Probability laws act on the right: \((\mu\mathfrak P_t)(f)=\mu(\mathfrak P_t f)\).

For $\alpha>0$, $\alpha\ne1$, set
\begin{equation}
\label{eq:a-alpha-map}
 a=\frac{\alpha}{\alpha-1}.
\end{equation}
Thus $a<0$ for $\alpha<1$ and $a>1$ for $\alpha>1$.
Writing the elementary boundary weights in terms of $\alpha$ gives
\begin{equation}
\label{eq:fundamental-boundary-alpha-form}
 \kappa_1(u)=\frac{\alpha(1-u^2)}{\alpha+(1-\alpha)u},\qquad
 \kappa_0(u)=\frac{u(1-\alpha+\alpha u)}{\alpha+(1-\alpha)u}.
\end{equation}

Let us define a double row operator, see Figure \ref{fig:finite-fundamental-double-row} for a graphical representation. The incoming
lower auxiliary color is set to be \(a_N=c\), and its outgoing upper color is
\(b_N=d\).  In column \(j\), the lower and upper parameters are,
respectively,
\begin{equation}
\label{eq:fundamental-double-row-rapidities}
 u/y_j,\qquad uy_j.
\end{equation}
\begin{definition}
\label{def:finite-fundamental-double-row}
For words \(\boldsymbol\eta,\boldsymbol\xi\in\mathcal C_n^N\), define
\begin{equation}
\label{eq:unclosed-fundamental-double-row}
 \mathfrak T^{N;\mathrm{open}}_{u;\mathbf y}(c,\boldsymbol\eta;d,\boldsymbol\xi)
 =\sum_{\substack{\boldsymbol v,a_0,\ldots,a_{N-1}\\b_0,\ldots,b_{N-1}}}
 K_u(a_0;b_0)
 \prod_{j=1}^N
 R_{u/y_j}(a_j,\eta_j;a_{j-1},v_j)
 R_{uy_j}(v_j,b_{j-1};\xi_j,b_j).
\end{equation}
Set
\[
 \mathfrak T^{N,c}_{u;\mathbf y}(\boldsymbol\eta,\boldsymbol\xi)
 =\sum_{d\in\mathcal C_n}\mathfrak T^{N;\mathrm{open}}_{u;\mathbf y}
                  (c,\boldsymbol\eta;d,\boldsymbol\xi),
 \qquad \mathfrak T^N_{u;\mathbf y}=\mathfrak T^{N,0}_{u;\mathbf y}.
\]
\end{definition}
The lower row is sampled from right to left, then the boundary vertex is sampled,
and finally the upper row is sampled from left to right. The procedure is stochastic.
In the sequel, we use the homogeneous parameters
\begin{equation}
\label{eq:ASEP-column-rapidities}
 y_1=\cdots=y_N=1,
\end{equation}
and suppress \(\mathbf y\) in notation. Note that we sum over the output $d$, cf.
\cite[Definition~3.2 and Section~6]{MR4924418}.

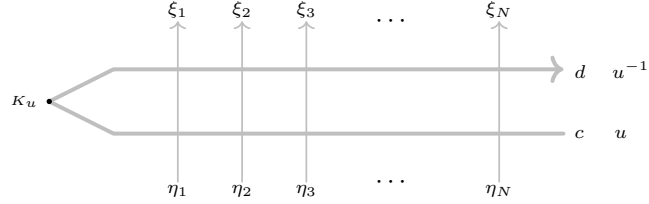
\begin{figure}[htbp]
\centering
\begin{tikzpicture}[
baseline={([yshift=-.5ex]current bounding box.center)},
scale=.85
]

\foreach \xx/\bottomstate/\topstate in {
  2/{\eta_1}/{\xi_1},
  3/{\eta_2}/{\xi_2},
  4/{\eta_3}/{\xi_3},
  7/{\eta_N}/{\xi_N}}
{
  \draw[paper wire,->] (\xx,-.75)--(\xx,1.75);
  \node[below,inner sep=1pt] at (\xx,-.75)
    {$\scriptstyle\bottomstate$};
  \node[above,inner sep=1pt] at (\xx,1.75)
    {$\scriptstyle\topstate$};
}
\node at (5.35,-.75) {$\cdots$};
\node at (5.35,1.75) {$\cdots$};

\draw[paper normal]
  (8,0)--(1,0)--(0,.5);
\draw[paper normal,->]
  (0,.5)--(1,1)--(8,1);
\node[paper boundary] at (0,.5) {};

\node[paper label,left=3pt] at (0,.5)
  {$\scriptstyle K_u$};
\node[right=3pt,inner sep=1pt] at (8,0)
  {$\scriptstyle c$};
\node[right=18pt,inner sep=1pt] at (8,0)
  {$\scriptstyle u$};
\node[right=3pt,inner sep=1pt] at (8,1)
  {$\scriptstyle d \ \ \ u^{-1}$};

\end{tikzpicture}
\caption{The finite fundamental double row
\(\mathfrak T^{N,c}_{u;\mathbf 1}\).}
\label{fig:finite-fundamental-double-row}
\end{figure}

For a finite word write \(\boldsymbol\eta^{N\to c}\) for the replacement of
its rightmost color by \(c\).  Let the finite space ASEP generator \(\mathcal L^{\rm int}_{\alpha,N}\) consist
of the bonds \(1,\ldots,N-1\) and the left boundary term of
\eqref{eq:multispecies-half-space-ASEP-generator}, and set
\begin{equation}
\label{eq:finite-out-generator}
 \mathcal L^{\rm out}_{\alpha,N}f(\boldsymbol\eta)
 =\mathcal L^{\rm int}_{\alpha,N}f(\boldsymbol\eta)
 +\mathbf1_{\{\eta_N\ne0\}}
       \bigl(f(\boldsymbol\eta^{N\to0})-f(\boldsymbol\eta)\bigr).
\end{equation}

\begin{proposition}
\label{prop:double-row-to-half-space-ASEP}
Set
\begin{equation}
\label{eq:ASEP-row-scaling}
 u_\varepsilon=1-(1-q)\varepsilon.
\end{equation}
On the finite state space of pairs \((c,\boldsymbol\eta)\),
\(\mathfrak T^{N;\mathrm{open}}_{u_\varepsilon}=\mathrm{Id}+2\varepsilon\mathcal G_N+
O_N(\varepsilon^2)\), where
\begin{equation}
 \mathcal G_N F(c,\boldsymbol\eta)
 =\mathcal L^{\rm int}_{\alpha,N}F(c,\cdot)(\boldsymbol\eta)
 +\mathfrak c_q(\eta_N,c)
 \bigl(F(\eta_N,\boldsymbol\eta^{N\to c})-F(c,\boldsymbol\eta)\bigr).
 \label{eq:unclosed-double-row-generator}
\end{equation}
Consequently,
\begin{equation}
\label{eq:ASEP-row-first-order}
 \mathfrak T^{N,0}_{u_\varepsilon}
 =\mathrm{Id}+2\varepsilon\mathcal L^{\rm out}_{\alpha,N}
       +O_N(\varepsilon^2).
\end{equation}
The remainders are in the sup-norm operator norm at each fixed $N$.
\end{proposition}

\begin{proof}
At \(u=1\), \(R_1\) swaps its two inputs and \(K_1\) is the identity.
In \eqref{eq:unclosed-fundamental-double-row}, the lower and upper passes
therefore undo each other, including the auxiliary color.  A first-order
transition contains exactly one exception to this deterministic pattern.
We have
\[
 \frac{1-u_\varepsilon}{1-qu_\varepsilon}
     =\varepsilon+O(\varepsilon^2),\qquad
 \frac{q(1-u_\varepsilon)}{1-qu_\varepsilon}
     =q\varepsilon+O(\varepsilon^2),
\]
and
$\kappa_1(u_\varepsilon)=2(1-q)\alpha\varepsilon+O(\varepsilon^2)$
by \eqref{eq:fundamental-boundary-alpha-form}.
A bulk bond is visited in both passes, giving the coefficient \(2\).
At the right endpoint an exception in either pass exchanges \(\eta_N\)
and \(c\), with coefficient \(\mathfrak c_q(\eta_N,c)\); this gives
\eqref{eq:unclosed-double-row-generator}.  For \(c=0\) the exchange is
the exit term in \eqref{eq:finite-out-generator}.  This is the colored,
endpoint-retaining form of the expansion in
\cite[Lemma~6.2]{MR4924418}.
\end{proof}

At fixed $N$, we have
\begin{equation}
\label{eq:finite-row-Euler-limit}
 \lim_{M\to\infty}
 \left(\mathfrak T^{N,0}_{1-(1-q)t/(2M)}\right)^M
 =\exp\bigl(t\mathcal L^{\mathrm{out}}_{\alpha,N}\bigr).
\end{equation}
For a bounded cylinder $f$ supported on $[1,k]$, let $f_N$ be its restriction
to $N$ sites.  Coupling the finite and half-line processes with the same
Harris clocks yields
\begin{equation}
\label{eq:ASEP-row-continuous-time-limit}
 \lim_{N\to\infty}\lim_{M\to\infty}
 \left\|\left(\mathfrak T^{N,0}_{1-(1-q)t/(2M)}\right)^M f_N
             -\mathfrak P_t f\right\|_\infty=0.
\end{equation}

\subsection{Queue boundary specialization}
\label{subsec:boundary-queue-specialization}

There are two boundary regimes.  For $0<\alpha<1$, we first take the
queue limit (infinitely many arrows of one color) and then specialize the
density parameters.  For $\alpha>1$, the boundary kernel instead becomes
deterministic at $x=a^{-1}$, where $a$ is given by
\eqref{eq:a-alpha-map}.
We treat the deterministic case first.

\paragraph{Deterministic boundary.}
The deterministic parameter range is
\begin{equation}
\label{eq:queue-empty-cap-range}
 0<x<1,\qquad a=x^{-1}>1,\qquad s<0.
\end{equation}
At this specialization, the boundary output is independent of the
incoming normal state.

\begin{lemma}
\label{prop:empty-cap-boundary}
Under \eqref{eq:queue-empty-cap-range}, for every usually indexed state
$G$ and hole-indexed state $H$,
\begin{equation}
\label{eq:empty-cap-boundary}
 K_{x,s}(G;H)=\mathbf1_{\{H=0\}}.
\end{equation}
The same formula holds after making any one usually indexed coordinate
infinite.  The local identities of
Propositions~\ref{prop:continued-colored-ybe}
and~\ref{prop:continued-colored-reflection} are valid at this
specialization whenever their other weights are regular.
\end{lemma}
\begin{proof}
Consider \eqref{eq:continued-colored-K}. If $|H|>0$, the factor
$(ax;q)_{|H|}=(1;q)_{|H|}$ vanishes.  If $H=0$, the normalization
\eqref{eq:continued-K-normalization} and the ratio
$(x/a;q)_{G_1}/(x^2;q)_{G_1}$ both equal one.
Since this formula does not depend on the incoming
multiplicities, it also holds in the queue limit.
\end{proof}

\paragraph{Stochastic boundary.}
We now work in the stochastic parameter range
\eqref{eq:continued-K-stochastic-range}.  The queue specialization sends
the number of normally indexed arrows of color~$m$ to infinity, while
the hole-indexed boundary output $H$ remains finite.  We use the laws
\(p_{u,v}^{(q)}\), \(q\text{-}\mathrm{NB}(u,v)\), \(p_v^{(q)}\), and
\(q\text{-}\mathrm{Poi}(v)\) of
\eqref{eq:q-negative-binomial-law}--\eqref{eq:q-Poisson-law}.

\begin{lemma}
\label{lem:queue-boundary-specialization}
Assume \eqref{eq:continued-K-stochastic-range}.  Let
\[
 A=(A_1,\ldots,A_{m-1},\infty,A_{m+1},\ldots,A_n)
\]
be a state with infinitely many arrows of color~\(m\).  The limit
\begin{equation}
\label{eq:def-queue-K}
 \mathbb K^{(m)}_{x,s}(A;H)
 :=\lim_{M\to\infty}
 K_{x,s}\bigl(
 (A_1,\ldots,A_{m-1},M,A_{m+1},\ldots,A_n);H\bigr)
\end{equation}
exists for every fixed hole-indexed output~$H$.  For $m=1$, all
weak-color outputs vanish, and
\begin{equation}
\label{eq:queue-K-m-one}
 \mathbb K^{(1)}_{x,s}(A;H)
 =\mathbf1_{\{H_2=\cdots=H_n=0\}}
   p_{ax,sx/a}^{(q)}(H_0).
\end{equation}
For $m\geq2$, the output is supported on
\begin{equation}
\label{eq:queue-K-support}
 0\leq H_i\leq A_i\quad(2\leq i<m),
 \qquad H_i=0\quad(i>m),
\end{equation}
with $H_0,H_m\geq0$ unrestricted.  Thus only the $m$th weak-color
multiplicity can be arbitrarily large.  Write $T:=H_{[2,m]}$; on this
support,
\begin{equation}
\label{eq:queue-K-m-general}
\begin{aligned}
\mathbb K^{(m)}_{x,s}(A;H)
&=\mathcal N^K_{x,s}
  \frac{(x/a;q)_{A_1}(ax;q)_{H_0+T}(sx/a)^{H_0}}
       {(x^2;q)_{A_1+H_0+T}(q;q)_{H_0}(q;q)_{H_m}}
  \left(-\frac xa\right)^T\\
&\quad\times
  q^{A_1T+\binom T2+
     \sum_{2\leq i<j\leq m}(A_i-H_i)H_j}
  \prod_{i=2}^{m-1}\begin{bmatrix}A_i\\H_i\end{bmatrix}_{\!q}.
\end{aligned}
\end{equation}
The weights are nonnegative and
\begin{equation}
\label{eq:queue-K-normalization}
       \sum_{H}\mathbb K^{(m)}_{x,s}(A;H)=1.
\end{equation}
\end{lemma}

\begin{proof}
Suppose first that $m\geq2$.  In
\eqref{eq:continued-colored-K}, put
\[
 G=(A_1,\ldots,A_{m-1},M,A_{m+1},\ldots,A_n).
\]
For fixed \(H\),
\[
 (q^{M+\sum_{a\ne m}A_a}/s;q)_{H_0}\longrightarrow1,
 \qquad
 \begin{bmatrix}M\\H_m\end{bmatrix}_{\!q}
 \longrightarrow\frac1{(q;q)_{H_m}}.
\]
If some \(H_j>0\) with \(j>m\), the exponent contains a factor
\(q^{MH_j}\), so the limit is zero.  Collecting the remaining powers gives
\eqref{eq:queue-K-m-general}.  For \(m=1\), every term with
\(H_{[2,n]}>0\) contains a vanishing power of \(q^M\).  The remaining factors
reduce to
\eqref{eq:queue-K-m-one}.  This is a \(q\)-negative binomial law, whose
sum-to-one property follows from the \(q\)-binomial theorem.
\end{proof}

We next send $x$ to zero while keeping $sx/a$ fixed.  The resulting
boundary law has no weak-color output: only the hole count $H_0$
remains random.

\begin{corollary}
\label{cor:queue-K-density-specialization}
Let \(0<\alpha<1\), \(0<\rho<1\), and \(\alpha+\rho>1\), and set
\begin{equation}
\label{eq:queue-density-boundary-parameters}
 a=\frac{\alpha}{\alpha-1},
 \qquad
 s_x=-\frac{1-\rho}{\rho x},
 \qquad
 \beta:=\frac{s_xx}{a}
 =\frac{(1-\alpha)(1-\rho)}{\alpha\rho}\in(0,1).
\end{equation}
The condition $\alpha+\rho>1$ is equivalent to $\beta<1$.
For the queue specialization in any color~$m$ and every fixed finite
choice of the remaining coordinates of~$A$,
\begin{equation}
\label{eq:queue-K-density-limit}
 \mathbb K^{(m)}_{x,s_x}(A;H)
 \xrightarrow[x\downarrow0]{}
 \mathbf1_{\{H_2=\cdots=H_n=0\}}p_\beta^{(q)}(H_0).
\end{equation}
\end{corollary}

\begin{proof}
For fixed incoming and outgoing states, substitute
\eqref{eq:queue-density-boundary-parameters} in
\eqref{eq:queue-K-m-one} or \eqref{eq:queue-K-m-general}.
For $m\geq2$, any nonzero term with $H_{[2,n]}>0$ has
$H_{[2,m]}>0$, and the factor $(-x/a)^{H_{[2,m]}}$ forces it to zero;
for $m=1$, such terms already vanish.  When $H_{[2,n]}=0$, the finite Pochhammer factors tend to
one, and
\[
 \mathcal N^K_{x,s_x}
 =\frac{(x^2;q)_\infty(\beta;q)_\infty}
 {(x/a;q)_\infty(\beta ax;q)_\infty}
 \longrightarrow(\beta;q)_\infty.
\]
This proves the claim. 
\end{proof}

\paragraph{Local Yang--Baxter and reflection identities.}
The finite-color identities from Section~\ref{sec:colored-vertex-weights},
continued in Subsection~\ref{subsec:continued-colored-identities}, also
survive the queue limit.  At fixed exterior states, each component
identity involves only finitely many internal configurations; thus
passing to an infinite color multiplicity can be done directly. This establishes the following proposition. 

\begin{proposition}
\label{thm:queue-ybe-reflection}
\label{prop:queue-ybe-reflection}
Assume either \eqref{eq:continued-K-stochastic-range} or
\eqref{eq:queue-empty-cap-range}; in the latter case put
$\mathbb K^{(m)}_{x,s}(A;H)=\mathbf1_{\{H=0\}}$.
The normally indexed Yang--Baxter equation of
Proposition~\ref{prop:continued-colored-ybe}, with one fused line and
two fundamental lines, remains valid after the color-$m$ queue limit:
its fused vertices become $\mathbb W^{(m)}$.  The corresponding mixed
Yang-Baxter equation on the reflected segment uses a hole-indexed
fused line with vertices $\widehat L^{\rm ver}$.

The queue reflection equation is as follows.  Let $A$ have infinitely
many normally indexed arrows of color~$m$, let $H$ be an exterior
hole-indexed state, and take $i,j\in\mathcal C_n$.  Then
\begin{align}
&\sum_{\gamma_1,\gamma_2,A^{(1)},H^{(1)}}
\mathbb W^{(m)}_{x/y,s}(A,i;A^{(1)},\gamma_1)
\mathbb K^{(m)}_{x,s}(A^{(1)};H^{(1)})
\widehat L^{\mathrm{ver}}_{xy,s}
 (\gamma_1,H^{(1)};\gamma_2,H)K_y(\gamma_2;j)
\notag\\
&\qquad=
\sum_{\gamma_1,\gamma_2,A^{(1)},H^{(1)}}
K_y(i;\gamma_1)
\mathbb W^{(m)}_{xy,s}(A,\gamma_1;A^{(1)},\gamma_2)
\mathbb K^{(m)}_{x,s}(A^{(1)};H^{(1)})
\widehat L^{\mathrm{ver}}_{x/y,s}
 (\gamma_2,H^{(1)};j,H).
\label{eq:queue-reflection-components}
\end{align}
In both sums, $A^{(1)}$ and $H^{(1)}$ are internal normally and
hole-indexed states, respectively, and $\gamma_1,\gamma_2$ are internal
colors.
\end{proposition}

\subsection{Combination of engines}
\label{subsec:half-line-engine}

We use the conventions of Subsection~\ref{subsec:intro-stationary-construction},
padding carrier states with zeros to $n=r$ colors.  Write
$\mathcal C_m=\{0,\ldots,m\}$ and let $\mathbf0^{(m)}$ be the state
with infinitely many arrows of color~$m$ and no others.  To formulate
the commutation identities, we extend the $m$th engine to physical inputs
in $\mathcal C_m^N$, rather than just $\mathcal C_{m-1}^N$.

\begin{definition}
\label{def:half-line-engine}
Assume either \eqref{eq:continued-K-stochastic-range} or
\eqref{eq:queue-empty-cap-range}. Let $G$ be a normally indexed state
with $G_m=\infty$ and $G_k=0$ for $k>m$, let $H$ be a hole-indexed
state with $H_k=0$ for $k>m$, and take
$\boldsymbol\eta,\boldsymbol\xi\in\mathcal C_m^N$.  Define the
\emph{open engine} by
\begin{align}
 &\mathfrak E^{N,m;\mathrm{open}}_{x,s}(G,\boldsymbol\eta;
                         H,\boldsymbol\xi)\notag\\*
 &\quad=\sum_{\substack{\boldsymbol v,A^{(0)},\ldots,A^{(N-1)}\\*
                        H^{(0)},\ldots,H^{(N-1)}}}
 \mathbb K^{(m)}_{x,s}(A^{(0)};H^{(0)})
 \prod_{j=1}^N
 \mathbb W^{(m)}_{x,s}(A^{(j)},\eta_j;A^{(j-1)},v_j)
 \widehat L^{\rm ver}_{x,s}(v_j,H^{(j-1)};\xi_j,H^{(j)}),
 \label{eq:unclosed-engine-kernel}
\end{align}
where $A^{(N)}=G$ and $H^{(N)}=H$.  The normal carrier runs from
site $N$ to site $1$, and the return carrier runs from $1$ to $N$.
Summing over the return-carrier endpoint gives the closed kernels
\[
 \mathfrak E^{N,m;G}_{x,s}(\boldsymbol\eta,\boldsymbol\xi)
 =\sum_H\mathfrak E^{N,m;\mathrm{open}}_{x,s}
                 (G,\boldsymbol\eta;H,\boldsymbol\xi),
 \qquad
 \mathfrak E^{N,m}_{x,s}
 =\mathfrak E^{N,m;\mathbf0^{(m)}}_{x,s}.
\]
The open endpoints will be needed for the construction in
Subsection~\ref{subsec:engine-commutation}; the corresponding vertex model
is depicted in Figure~\ref{fig:finite-half-line-engine}.
\end{definition}

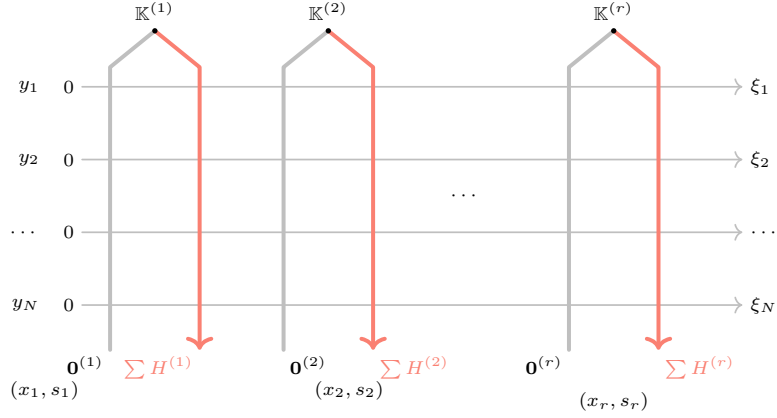
\begin{figure}[htbp]
\centering
\begin{tikzpicture}[
baseline={([yshift=-.5ex]current bounding box.center)},
 x=.74cm,y=.74cm,
 engstate/.style={font=\scriptsize,inner sep=.8pt}
]

\foreach \yy/\outstate/\rap in {
  .9/{\xi_N}/{y_N},
  2.2/{\cdots}/{\cdots},
  3.5/{\xi_2}/{y_2},
  4.8/{\xi_1}/{y_1}}
{
  \draw[paper wire,->] (.35,\yy)--(12.15,\yy);
  \node[engstate,left=2pt] at (.35,\yy) {$0$};
  \node[engstate,left=15pt] at (.35,\yy) {$\rap$};
  \node[engstate,right=2pt] at (12.15,\yy) {$\outstate$};
}

\draw[paper normal] (.85,.10)--(.85,5.15)--(1.65,5.80);
\draw[paper hole,->] (1.65,5.80)--(2.45,5.15)--(2.45,.10);
\node[paper boundary] at (1.65,5.80) {};
\node[engstate,below left=2pt] at (.85,.10) {$\mathbf0^{(1)}$};
\node[engstate,below left=2pt,text=holecolor] at (2.45,.10)
  {$\sum H^{(1)}$};
\node[engstate,below left=15pt] at (.85,.10) {$(x_1,s_1)$};
\node[engstate,above=3pt] at (1.65,5.80) {$\mathbb K^{(1)}$};

\draw[paper normal] (3.95,.10)--(3.95,5.15)--(4.75,5.80);
\draw[paper hole,->] (4.75,5.80)--(5.55,5.15)--(5.55,.10);
\node[paper boundary] at (4.75,5.80) {};
\node[engstate,below right=2pt] at (3.95,.10) {$\mathbf0^{(2)}$};
\node[engstate,below right=2pt,text=holecolor] at (5.55,.10)
  {$\sum H^{(2)}$};
\node[engstate,below right=15pt] at (3.95,.10) {$(x_2,s_2)$};
\node[engstate,above=3pt] at (4.75,5.80) {$\mathbb K^{(2)}$};

\draw[paper normal] (9.05,.10)--(9.05,5.15)--(9.85,5.80);
\draw[paper hole,->] (9.85,5.80)--(10.65,5.15)--(10.65,.10);
\node[paper boundary] at (9.85,5.80) {};
\node[engstate,below left=2pt] at (9.05,.10) {$\mathbf0^{(r)}$};
\node[engstate,below right=2pt,text=holecolor] at (10.65,.10)
  {$\sum H^{(r)}$};
\node[engstate,below=15pt] at (9.85,.10)
  {$(x_r,s_r)$};
\node[engstate,above=3pt] at (9.85,5.80) {$\mathbb K^{(r)}$};
\node[engstate] at (7.20,2.85) {$\cdots$};
\end{tikzpicture}
\caption{The combination of engines \eqref{eq:finite-engine-array}, after a clockwise rotation. The engines of colors \(3,\ldots,r-1\) are omitted between the
second and last engines.}
\label{fig:finite-half-line-engine}
\end{figure}

For $0<\alpha<1$, we use the density limit
\begin{equation}
\label{eq:engine-density-specialization}
 x_m\downarrow0,\qquad s_m=-\frac{\bar\rho_m}{\rho_m x_m},\qquad
 a=\frac{\alpha}{\alpha-1}.
\end{equation}
For $\alpha>1$, use instead the exact specialization
\begin{equation}
\label{eq:engine-large-rate-specialization}
 a=\frac{\alpha}{\alpha-1},\qquad
 x_m=a^{-1}=1-\alpha^{-1},\qquad
 s_m=-\frac{1-\rho_m}{\rho_m x_m}.
\end{equation}
We also use the density-specialized vertices
$\mathbb W^{(m)}_{\rho_m}$, $\mathbb K^{(m)}_{\alpha,\rho_m}$, and
$\widehat L^{\mathrm{ver}}_{\rho_m}$ of
Subsection~\ref{subsec:intro-stationary-construction}, and denote
the corresponding closed engines by $\mathfrak E^{N,m}$.  Their combination produces
\begin{equation}
\label{eq:finite-engine-array}
 \mu^N_{\mathrm{eng}}
 =\delta_{\boldsymbol0_N}\mathfrak E^{N,1}\cdots\mathfrak E^{N,r}.
\end{equation}
At $\alpha=1$ the two regimes have the common limit $x=0$ with a
zero boundary state.  For both regimes, set
\begin{equation}
\label{eq:all-rate-boundary-parameter}
 \beta_m:=\frac{(1-\alpha)_+(1-\rho_m)}{\alpha\rho_m}.
\end{equation}
The boundary hole count has law $q\text{-}\mathrm{Poi}(\beta_m)$, with
$\beta_m=0$ (hence a deterministic zero boundary) when $\alpha\ge1$.

\begin{proposition}
\label{prop:density-array-convergence}
For every fixed size and permitted incoming state,
\eqref{eq:engine-density-specialization} converges under the queue specialization in total variation
to the vertex model of Definition~\ref{def:intro-queue-array} for $0<\alpha<1$.
For $\alpha>1$, \eqref{eq:engine-large-rate-specialization} gives those
queues exactly.  The boundary output is independent of the incoming
normally indexed state, and the resulting density kernels are
continuous in total variation at $\alpha=1$ and $q=0$.
\end{proposition}

\begin {proof}
The usual row weights of Subsection \ref{subsec:intro-stationary-construction} match 
\eqref{eq:queue-L-families} and Figure~\ref{fig:colored-L-ver-families},
with $\rho=(1-sx)^{-1}$.
The boundary specializations are Corollary~\ref{cor:queue-K-density-specialization}
and Lemma~\ref{prop:empty-cap-boundary}. 
\end{proof}

\begin{lemma}
\label{lem:saturated-color-invisibility}
For every normally indexed state $A$ in the color-$m$ queue limit,
\begin{equation}
\label{eq:saturated-color-invisibility}
 \mathbb W^{(m)}_{z,s}(A,0;A',j)
 =\mathbb W^{(m)}_{z,s}(A,m;A',j).
\end{equation}
Consequently, the first engine sends any word in $\{0,1\}^N$ to
the same output law as the vacuum.  Both assertions hold also after
density specialization.
\end{lemma}
\begin{proof}
This is immediate from \eqref{eq:queue-L-families}: weights distinguishing
inputs $0$ and $m$ contain the vanishing factor $q^{A_m}$.
\end{proof}

\subsection{Finite commutation relation}
\label{subsec:engine-commutation}

We have already established that the queue vertices satisfy the Yang-Baxter and reflection equations. Repeating the argument of 
\cite[Propositions~3.7 and~3.9]{MR4924418} with the use of queue vertices, we arrive to the following proposition. 
See Figures~\ref{fig:finite-boundary-train}-\ref{fig:finite-boundary-train-return} for a graphical form of the proof.

\begin{proposition}
\label{prop:finite-boundary-train}
Assume either \eqref{eq:continued-K-stochastic-range} or
\eqref{eq:queue-empty-cap-range}, and let $x<u<1$.
For arrow configurations $G,H$,
$i,j\in\mathcal C_m$, and
$\boldsymbol\eta,\boldsymbol\xi\in\mathcal C_m^N$, one has
\begin{align}
 &\sum_{G^{(1)},H^{(1)},\gamma_1,\gamma_2,\boldsymbol\nu}
 \mathbb W^{(m)}_{x/u,s}(G,i;G^{(1)},\gamma_1)
 \mathfrak E^{N,m;\mathrm{open}}_{x,s}(G^{(1)},\boldsymbol\eta;
                                 H^{(1)},\boldsymbol\nu)\notag\\*
 &\hspace{15mm}\times
 \widehat L^{\rm ver}_{xu,s}(\gamma_1,H^{(1)};\gamma_2,H)
 \mathfrak T^{N;\mathrm{open}}_u(\gamma_2,\boldsymbol\nu;j,\boldsymbol\xi)
 \notag\\*
 &=\sum_{G^{(1)},H^{(1)},\gamma_1,\gamma_2,\boldsymbol\nu}
 \mathfrak T^{N;\mathrm{open}}_u(i,\boldsymbol\eta;\gamma_1,\boldsymbol\nu)
 \mathbb W^{(m)}_{xu,s}(G,\gamma_1;G^{(1)},\gamma_2)\notag\\*
 &\hspace{15mm}\times
 \mathfrak E^{N,m;\mathrm{open}}_{x,s}(G^{(1)},\boldsymbol\nu;
                                 H^{(1)},\boldsymbol\xi)
 \widehat L^{\rm ver}_{x/u,s}(\gamma_2,H^{(1)};j,H).
 \label{eq:finite-augmented-train}
\end{align}
\end{proposition}

\begin{figure}[H]
	\centering
	\begin{tikzpicture}[
		x=.70cm,y=.63cm,
		train state/.style={font=\scriptsize,inner sep=.6pt}
		]
		\foreach \xx in {1.70,2.70,3.70,4.70}{
			\draw[paper wire,->] (\xx,-.45)--(\xx,3.45);
		}
		\node[train state,above=1pt] at (1.70,3.45) {$\xi_1$};
		\node[train state,above=1pt] at (2.70,3.45) {$\xi_2$};
		\node[train state,above=1pt] at (3.70,3.45) {$\cdots$};
		\node[train state,above=1pt] at (4.70,3.45) {$\xi_N$};
		\node[train state,below=1pt] at (1.70,-.45) {$\eta_1$};
		\node[train state,below=1pt] at (2.70,-.45) {$\eta_2$};
		\node[train state,below=1pt] at (3.70,-.45) {$\cdots$};
		\node[train state,below=1pt] at (4.70,-.45) {$\eta_N$};
		
		\draw[paper normal]
		(7.2,1)--(6.55,1)--(5.55,0)--(.50,0)--(0,.50);
		\draw[paper hole,->]
		(0,.50)--(.50,1)--(4.95,1)--(5.95,2)--(7.2,2);
		\node[paper boundary] at (0,.50) {};
		
		\draw[paper wire,->]
		(7.2,0)--(6.55,0)--(4.95,2)--(.50,2)--(0,2.50)--(.50,3)--(7.2,3);
		\node[paper boundary] at (0,2.50) {};
		
		\node[train state,anchor=west] at (7.24,0) {$0\,;\,u$};
		\node[train state,anchor=west] at (7.24,1) {$\mathbf0^{(m)}\,;\,(x,s)$};
		\node[train state,anchor=west,text=holecolor] at (7.24,2)
		{$\sum H'$};
		\node[train state,anchor=west] at (7.24,3)
		{$\sum c'$};
	\end{tikzpicture}
	\caption{Initial diagram for the commutation identity, with crossings of parameters
		$x/u$ and $xu$.  The auxiliary inputs are $0$ and $\mathbf0^{(m)}$,
		the exterior outputs are summed.}
	\label{fig:finite-boundary-train}
\end{figure}
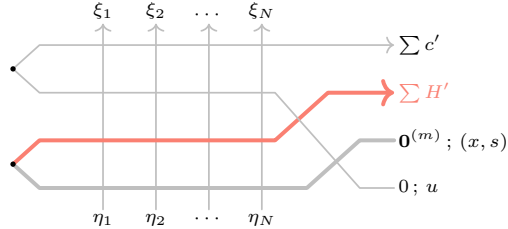

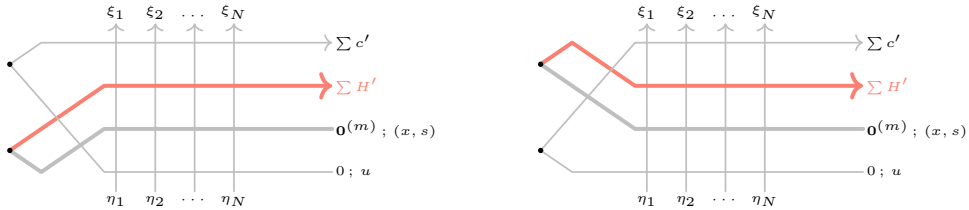
\begin{figure}[H]
	\centering
	\begin{tikzpicture}[
		x=.52cm,y=.57cm,
		train state/.style={font=\tiny,inner sep=.6pt}
		]
		\begin{scope}
			\foreach \xx in {1.70,2.70,3.70,4.70}{
				\draw[paper wire,->] (\xx,-.45)--(\xx,3.45);
			}
			\node[train state,above=1pt] at (1.70,3.45) {$\xi_1$};
			\node[train state,above=1pt] at (2.70,3.45) {$\xi_2$};
			\node[train state,above=1pt] at (3.70,3.45) {$\cdots$};
			\node[train state,above=1pt] at (4.70,3.45) {$\xi_N$};
			\node[train state,below=1pt] at (1.70,-.45) {$\eta_1$};
			\node[train state,below=1pt] at (2.70,-.45) {$\eta_2$};
			\node[train state,below=1pt] at (3.70,-.45) {$\cdots$};
			\node[train state,below=1pt] at (4.70,-.45) {$\eta_N$};
			
			\draw[paper normal]
			(7.2,1)--(1.40,1)--(-.20,0)--(-1.00,.50);
			\draw[paper hole,->]
			(-1.00,.50)--(-.20,1)--(1.40,2)--(7.2,2);
			\draw[paper wire,->]
			(7.2,0)--(1.40,0)--(-1.00,2.50)
			--(-.20,3)--(7.2,3);
			\node[paper boundary] at (-1.00,.50) {};
			\node[paper boundary] at (-1.00,2.50) {};
			\node[train state,anchor=west] at (7.24,0) {$0\,;\,u$};
			\node[train state,anchor=west] at (7.24,1) {$\mathbf0^{(m)}\,;\,(x,s)$};
			\node[train state,anchor=west,text=holecolor] at (7.24,2)
			{$\sum H'$};
			\node[train state,anchor=west] at (7.24,3)
			{$\sum c'$};
		\end{scope}
		
		\begin{scope}[shift={(13.5,0)}]
			\foreach \xx in {1.70,2.70,3.70,4.70}{
				\draw[paper wire,->] (\xx,-.45)--(\xx,3.45);
			}
			\node[train state,above=1pt] at (1.70,3.45) {$\xi_1$};
			\node[train state,above=1pt] at (2.70,3.45) {$\xi_2$};
			\node[train state,above=1pt] at (3.70,3.45) {$\cdots$};
			\node[train state,above=1pt] at (4.70,3.45) {$\xi_N$};
			\node[train state,below=1pt] at (1.70,-.45) {$\eta_1$};
			\node[train state,below=1pt] at (2.70,-.45) {$\eta_2$};
			\node[train state,below=1pt] at (3.70,-.45) {$\cdots$};
			\node[train state,below=1pt] at (4.70,-.45) {$\eta_N$};
			\draw[paper normal]
			(7.2,1)--(1.40,1)--(-.20,2)--(-1.00,2.50);
			\draw[paper hole,->]
			(-1.00,2.50)--(-.20,3)--(1.40,2)--(7.2,2);
			\draw[paper wire,->]
			(7.2,0)--(-.20,0)--(-1.00,.50)--(1.40,3)--(7.2,3);
			\node[paper boundary] at (-1.00,.50) {};
			\node[paper boundary] at (-1.00,2.50) {};
			\node[train state,anchor=west] at (7.24,0) {$0\,;\,u$};
			\node[train state,anchor=west] at (7.24,1) {$\mathbf0^{(m)}\,;\,(x,s)$};
			\node[train state,anchor=west,text=holecolor] at (7.24,2)
			{$\sum H'$};
			\node[train state,anchor=west] at (7.24,3)
			{$\sum c'$};
		\end{scope}
		
	\end{tikzpicture}
	\caption{Yang-Baxter moves through columns $N,N-1,\ldots,1$ bring the crossings
		to the boundary (left), followed by the reflection equation (right).
	}
	\label{fig:finite-boundary-train-boundary}
\end{figure}

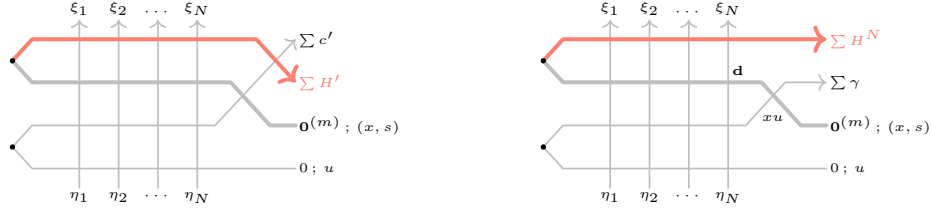
\begin{figure}[H]
	\centering
	\begin{tikzpicture}[
		x=.52cm,y=.57cm,
		train state/.style={font=\tiny,inner sep=.6pt}
		]
		\begin{scope}
			\foreach \xx in {1.70,2.70,3.70,4.70}{
				\draw[paper wire,->] (\xx,-.45)--(\xx,3.45);
			}
			\node[train state,above=1pt] at (1.70,3.45) {$\xi_1$};
			\node[train state,above=1pt] at (2.70,3.45) {$\xi_2$};
			\node[train state,above=1pt] at (3.70,3.45) {$\cdots$};
			\node[train state,above=1pt] at (4.70,3.45) {$\xi_N$};
			\node[train state,below=1pt] at (1.70,-.45) {$\eta_1$};
			\node[train state,below=1pt] at (2.70,-.45) {$\eta_2$};
			\node[train state,below=1pt] at (3.70,-.45) {$\cdots$};
			\node[train state,below=1pt] at (4.70,-.45) {$\eta_N$};
			
			\draw[paper wire,->]
			(7.2,0)--(.50,0)--(0,.50)--(.50,1)--(5.15,1)--(7.2,3);
			\node[paper boundary] at (0,.50) {};
			\draw[paper normal]
			(7.2,1)--(6.55,1)--(5.55,2)--(.50,2)--(0,2.50);
			\draw[paper hole,->] (0,2.50)--(.50,3)--(6.2,3)--(7.2,2);
			\node[paper boundary] at (0,2.50) {};
			\node[train state,anchor=west] at (7.24,0) {$0\,;\,u$};
			\node[train state,anchor=west] at (7.24,1) {$\mathbf0^{(m)}\,;\,(x,s)$};
			\node[train state,anchor=west,text=holecolor] at (7.24,2)
			{$\sum H'$};
			\node[train state,anchor=west] at (7.24,3)
			{$\sum c'$};
		\end{scope}
		
		\begin{scope}[shift={(13.5,0)}]
			\foreach \xx in {1.70,2.70,3.70,4.70}{
				\draw[paper wire,->] (\xx,-.45)--(\xx,3.45);
			}
			\node[train state,above=1pt] at (1.70,3.45) {$\xi_1$};
			\node[train state,above=1pt] at (2.70,3.45) {$\xi_2$};
			\node[train state,above=1pt] at (3.70,3.45) {$\cdots$};
			\node[train state,above=1pt] at (4.70,3.45) {$\xi_N$};
			\node[train state,below=1pt] at (1.70,-.45) {$\eta_1$};
			\node[train state,below=1pt] at (2.70,-.45) {$\eta_2$};
			\node[train state,below=1pt] at (3.70,-.45) {$\cdots$};
			\node[train state,below=1pt] at (4.70,-.45) {$\eta_N$};
            
			\draw[paper wire,->]
			(7.2,0)--(.50,0)--(0,.50)--(.50,1)--(5.15,1)--(6.15,2)--(7.2,2);
			\node[paper boundary] at (0,.50) {};
			\draw[paper normal]
			(7.2,1)--(6.55,1)--(5.55,2)--(.50,2)--(0,2.50);
			\draw[paper hole,->] (0,2.50)--(.50,3)--(7.2,3);
			\node[paper boundary] at (0,2.50) {};
			\node[train state,anchor=west] at (7.24,0) {$0\,;\,u$};
			\node[train state,anchor=west] at (7.24,1)
			{$\mathbf0^{(m)}\,;\,(x,s)$};
			\node[train state,anchor=west] at (7.24,2)
			{$\sum\gamma$};
			\node[train state,above=2pt] at (4.95,2) {$\mathbf d$};
			\node[train state,fill=white,below=2pt] at (5.84,1.42) {$xu$};
			\node[train state,anchor=west,text=holecolor] at (7.24,3)
			{$\sum H^N$};
		\end{scope}
		
	\end{tikzpicture}
	\caption{Yang-Baxter moves through columns $1,2,\ldots,N$ (left), followed by
		summing the two outgoing states of the hole-indexed vertex (right).
		The normally indexed vertex of parameter $xu$ remains, and its output
		$\mathbf d$ enters the engine.}
	\label{fig:finite-boundary-train-return}
\end{figure}

\subsection{Infinite engine: stabilization}
\label{subsec:stable-engine-carriers}
\label{subsec:engine-locality}

In the double row operator, the usual row is read from right to left and the return row from
left to right. To construct an infinite engine, we must show that
changing the usual carrier sufficiently far to the right does not
change the colors in a fixed window on the left.  We do this by sampling different
carrier states with the \emph{same random choices at every vertex}.

\begin{proposition}
\label{def:engine-graphical-realization}
Recall $x=(1-\alpha^{-1})_+$.  Independently at each vertex, define random variables
\[
 S\sim\operatorname{Ber}(\rho_m),\qquad
 F\sim\operatorname{Ber}(x),\qquad
 \mathbb P(R=j)=(1-q)q^{j-1},\quad j\ge1.
\]
Then the following rules produce precisely the weights
in Figures~\ref{fig:intro-normal-weights} and~\ref{fig:intro-return-weights}.

\smallskip
\noindent\emph{Usual vertex.}
For $m\ge2$, list the arrows of the incoming carrier $G$ in the order
\[
 \underbrace{1,\ldots,1}_{G_1},\
 \underbrace{2,\ldots,2}_{G_2},\ \ldots,\
 \underbrace{m-1,\ldots,m-1}_{G_{m-1}},\
 m,m,\ldots .
\]
Let $c_R$
denote the color of the arrow in position $R$.  Given the incoming
physical color $i$, choose the outgoing physical color $j$ as follows:
\begin{enumerate}
 \item If $S=0$, set $j=0$.
 \item If $S=1$ and $i$ is not a recorded carrier color
       (in particular, if $i=0$), set $j=c_R$.
 \item If $S=1$ and $i=k\in\{1,\ldots,m-1\}$, set
       $b=G_1+\cdots+G_{k-1}$.  Set $j=c_R$ when $R\le b$.
       When $R>b$, set $j=k$ if $F=0$, and $j=c_R$ if $F=1$.
\end{enumerate}
The remaining colors determine the outgoing carrier $G'$.

\smallskip
\noindent\emph{Return vertex.}
We use exactly the same procedure, but list the incoming
carrier $H$ in the different order:
\[
 \underbrace{0,\ldots,0}_{H_0},\
 \underbrace{m,\ldots,m}_{H_m},\
 \underbrace{m-1,\ldots,m-1}_{H_{m-1}},\ \ldots,\
 \underbrace{2,\ldots,2}_{H_2},\
 1,1,\ldots .
\]
Let $c_R$ be
again the color in position $R$ in such an ordered list.  Given the incoming physical
color $i$, the outgoing physical color $j$ is determined via:
\begin{enumerate}
 \item If $S=0$, set $j=1$.
 \item If $S=1$ and $i=1$, set $j=c_R$.
 \item If $S=1$ and $i=k$ is recorded ($k=0$ or $2\le k\le m$), let
       $b=\sum_{a\prec k}H_a$, the number of recorded carrier arrows
       preceding color $k$ in the displayed order.  Set $j=c_R$ when
       $R\le b$.  When $R>b$, set $j=k$ if $F=0$, and $j=c_R$ if $F=1$.
\end{enumerate}
\end{proposition}
\begin{proof}
A direct check.
\end{proof}

For either row, the number of \emph{recorded carrier arrows} is the
sum of its finite color multiplicities; the infinitely repeated
color is excluded. Our goal is to show that the engine converges, in the sense of local convergence, as $N\to\infty$, and the limit is independent from the input at the bottom.
The next key technical lemma shows
that sufficiently distant vertices erase dependence on this input.

\begin{lemma}
\label{lem:scalar-clearing}
Let $A_i \in \{0,1 \}\}$ be the indicator of the event that the physical input at site $i$ is a non-zero
color, and let $h$ be the number non-zero arrows entering a vertex.
The number of arrows leaving it is
\begin{equation}
\label{eq:monotone-workload-recursion}
 \Phi_i(h,A_i)=h+A_i-S_i\bigl(\mathbf1_{\{R_i\le h\}}
       +A_i(1-F_i)\mathbf1_{\{R_i>h\}}\bigr).
\end{equation}
For fixed randomness at each vertex, $\Phi_i$ is nondecreasing in $h$.

Suppose that $(A_i)$ is stationary ergodic with mean
$\lambda<\rho_m$. For $N>d$,
let $C_N^{(N)}=0$ and
$C_{i-1}^{(N)}=\Phi_i(C_i^{(N)},A_i)$ for $i=N,\ldots,1$.
Then, at every fixed cut $d$,
\[
 C_d^{(N)}\uparrow C_d<\infty\quad\text{almost surely}.
\]
The process $(C_d)$ is stationary and ergodic, with
$\mathbb P(C_d=0)>0$.  For every integer $b\ge1$, almost surely
there are arbitrarily large cuts $d$ and a finite block immediately
to their left such that the carrier is empty at $d$ and the block
empties \emph{any} carrier entering with at most $b$ recorded arrows,
regardless of their colors.
\end{lemma}
\begin{proof}
Equation~\eqref{eq:monotone-workload-recursion} follows from arrow
conservation and Proposition~\ref{def:engine-graphical-realization}.
When $S_i=1$, its two maps are
$h\mapsto h-\mathbf1_{\{R_i\le h\}}$ for $A_i=0$ and
$h\mapsto h+F_i\mathbf1_{\{R_i>h\}}$ for $A_i=1$.
These formulas imply monotonicity.
Increasing $N$ replaces a zero carrier at the old endpoint by a
nonnegative one; hence $C_d^{(N)}$ increases in $N$.

Choose $B\ge1$ so large that $p:=\rho_m(1-q^B)>\lambda$, and set
$T_i=S_i\mathbf1_{\{R_i\le B\}}$.  At a vertex with at least $B$
recorded arrows, $T_i=1$ forces one to be emitted.  If there are fewer
than $B$, the outgoing carrier contains at most $B$ arrows.  Thus
\[
 \Phi_i(h,A_i)\le\max\{B,h+A_i-T_i\},\qquad
 C_d^{(N)}\le B+\max_{d\le j\le N}\sum_{i=d+1}^j(A_i-T_i).
\]
The $T_i$ are i.i.d. Bernoulli($p$), independent of $(A_i)$.
By the ergodic theorem, the sums on the right have negative linear
drift $\lambda-p<0$, and their supremum is finite almost surely.
This proves $C_d<\infty$.  On a two-sided extension, $(C_d)$ is a
shift-invariant function of the stationary ergodic inputs and
independent vertex randomness, hence is itself stationary and ergodic.

We claim that $\mathbb P(C_d=0)>0$.  Otherwise let $h\ge1$ be the
smallest value with $\mathbb P(C_d=h)>0$.  Since $\lambda<1$, almost
every input has infinitely many sites with $A_i=0$ to the left of $d$.
For some $L$, therefore, the event that $C_d=h$ and at least one
of $A_{d-L+1},\ldots,A_d$ vanishes has positive probability.
On these $L$ vertices impose $S_i=1$, $F_i=0$, $R_i=1$.
These choices have positive probability independently of $C_d$ and
the physical inputs.  By~\eqref{eq:monotone-workload-recursion},
they leave the carrier size unchanged when $A_i=1$ and
reduce it by one, if positive, when $A_i=0$.  Hence
$C_{d-L}<h$ with positive probability, contrary to the minimality
of $h$.

Finally, fix $b$.  On the event $C_d=0$, there are almost surely
infinitely many $A_i=0$ to the left.  We may therefore choose a
deterministic $L$ such that the event
\begin{equation}
\label{eq:scalar-clearing-block}
 C_d=0,\qquad \sum_{i=d-L+1}^d(1-A_i)\ge b,
 \qquad S_i=1,\ F_i=0,\ R_i=1\quad(d-L+1\le i\le d)
\end{equation}
has positive probability.  Indeed, the prescribed vertex choices
are independent of $C_d$ and of the physical inputs on this block.
By ergodicity, such blocks occur arbitrarily far to the right.
Each vertex with $A_i=0$ removes one recorded arrow, if present,
while the other vertices leave the population unchanged; thus the
block empties every entrance with at most $b$ recorded arrows.
\end{proof}

The lemma also holds when the input indicators agree, outside an
almost surely finite prefix, with a stationary ergodic sequence of
mean less than $\rho_m$. Indeed,
far enough to the right the scalar constructions coincide.  For a
fixed window and a bound $b$ on the right-entrance population, take
an erasing block~\eqref{eq:scalar-clearing-block} to the right of
that window and of the exceptional prefix.  At its right end, every
zero-entrance truncation has an empty carrier. The colored carriers and outputs therefore agree to
its left, \emph{uniformly} over all choices of the $b$ incoming
arrows.

\begin{proposition}
\label{prop:qualitative-localization}
\label{prop:engine-binary-projection}
Combination of engines
\eqref{eq:finite-engine-array}, sampled with common vertex randomness with all possible inputs from the bottom,
are eventually identical on every fixed output window, almost
surely, as $N \to \infty$. 
\end{proposition}
\begin{proof}
First observe that the occupation indicators in the output of engine
$m$ depend only on its own random choices, not on its physical input.
Indeed, the usual row emits a hole exactly when $S_i^{m,-}=0$, and
holes precede all other recorded colors in the return carrier.
Writing $H_i$ for the number of holes after its $i$th return vertex,
we obtain the one-color rule
\begin{equation}
\label{eq:autonomous-hole-carrier}
 H_i=\Phi_i^{m,+}(H_{i-1},1-S_i^{m,-}),\qquad
 \mathbf1_{\{\xi_i=0\}}=H_{i-1}+1-S_i^{m,-}-H_i.
\end{equation}
Here the physical hole inputs $1-S_i^{m,-}$ are i.i.d.
Bernoulli($1-\rho_m$) and are independent of the return-vertex
choices and the almost surely finite boundary carrier $H_0$.

For this one-color return row the scalar update is a birth--death
chain.  Its probabilities of gaining and losing a hole from state $h$
are, respectively,
\[
 (1-\rho_m)(1-\rho_m+\rho_m xq^h),\qquad
 \rho_m^2(1-q^h).
\]
Since $\rho_m>1/2$, the ratio of these probabilities tends to
$((1-\rho_m)/\rho_m)^2<1$.  The chain is positive recurrent and
has an invariant distribution with a finite mean.  Start a second
copy in that distribution, using the same physical inputs and
vertex choices.  The larger of the two initial carrier populations
dominates both copies and hits zero almost surely; from then on
they coincide.  The stationary copy has a stationary ergodic output
occupation sequence.  By arrow conservation and equality of the
mean incoming and outgoing hole counts in stationarity, its mean
occupation is $\rho_m$.  Thus the actual occupation output agrees
outside a finite random prefix with a stationary ergodic binary
sequence of mean $\rho_m$.

Engine $1$ is insensitive to its input by
Lemma~\ref{lem:saturated-color-invisibility}.  For $m\ge2$, the
recorded-input indicators of engine $m$ are precisely the occupation
indicators of engine $m-1$.  They are fixed across truncations by
the input-independent recursion~\eqref{eq:autonomous-hole-carrier},
and have a stationary ergodic tail of mean
$\rho_{m-1}<\rho_m$, independent of the new engine's vertex choices.
Lemma~\ref{lem:scalar-clearing} and the erasing-block coupling above
therefore construct the infinite usual row and give exact local
agreement of its truncations; the return row is then read from left
to right using the common boundary sample and vertex choices.
Induction over the finitely many engines proves the asserted local
limit. 
\end{proof}

\subsection{Stationarity}
\label{subsec:engine-stationarity}

In this section we prove the first main result of the paper, Theorem \ref{thm:intro-vertex-form}. 

In the next lemma, we prove an analog of a stabilization result from the previous section for a fundamental double row, which will become an ASEP generator in the limit. 

\begin{lemma}
\label{lem:infinite-fundamental-double-row}
For $x=(1-\alpha^{-1})_+<u<1$ set
$d(u)=(1-u)/(1-qu)<1$.  The finite kernels
$\mathfrak T_u^{N,c}$ converge, independently of the incoming
auxiliary color $c$, to a Markov kernel $\mathfrak T_u$
on $\Omega_n$.  For $f$ supported on $[1,k]$ and $N>k$,
\begin{equation}
\label{eq:fundamental-geometric-forgetting}
 \sup_{\eta,c}\left|
  \mathfrak T_u^{N,c}f(\eta|_{[1,N]})-\mathfrak T_uf(\eta)
 \right|\le 2\|f\|_\infty d(u)^{N-k}.
\end{equation}
Moreover,
\begin{equation}
\label{eq:full-configuration-fundamental-invariant}
 \delta_{\boldsymbol1}\mathfrak T_u=\delta_{\boldsymbol1}.
\end{equation}
\end{lemma}
\begin{proof}
At each fundamental vertex, the incoming colors can be exchanged
with probability at least $1-d(u)=u(1-q)/(1-qu)$.  Sample an
independent event of that probability on which the exchange is
performed for \emph{all} inputs.  In the lower right-to-left pass,
such an exchange at site $j$ makes the auxiliary color continuing
to the left independent of its distant right entrance.  Two
coupled truncations therefore agree on $[1,k]$ unless every lower
vertex at sites $k+1,\ldots,N$ fails this event; the probability of
failure is $d(u)^{N-k}$.  This proves the bound and the existence
of the limit. 
\end{proof}

\begin{proposition}
\label{prop:engine-preserves-double-row}
Fix $x<u<1$.  Suppose $\nu\mathfrak T_u=\nu$ and, for $m\ge2$,
the occupation indicators under $\nu$ agree outside an almost surely
finite prefix with a stationary ergodic binary sequence of mean
$\lambda<\rho_m$, independently of the $m$th engine's vertex
choices.  Then
\begin{equation}
\label{eq:engine-preserves-fundamental-invariance}
 (\nu\mathfrak E_m)\mathfrak T_u=\nu\mathfrak E_m.
\end{equation}
For $m=1$ the same holds for any invariant input law on
$\mathcal C_1^{\mathbb N}$.
\end{proposition}
\begin{proof}
Apply the identity~\eqref{eq:finite-augmented-train}
with empty usual carrier and incoming fundamental color $0$, sum
over the outgoing auxiliary states, and make the density
specialization of Proposition~\ref{prop:density-array-convergence}.
The left-hand side consists of an engine followed by a fundamental
double row with possibly random auxiliary entrance; the second
consists of the fundamental double row with entrance $0$, followed
by an engine with at most one additional recorded arrow in its
right-entrance carrier.  The final return crossing disappears on
summing by stochasticity.

By~\eqref{eq:fundamental-geometric-forgetting}, the first side
converges to $(\nu\mathfrak E_m)\mathfrak T_u$, regardless of the
random auxiliary entrance.  For the second, couple the finite
fundamental output $\zeta^N$ with the infinite output
$\zeta=\mathfrak T_u\eta$, where $\eta\sim\nu$.  For fixed $L$,
\[
 \mathbb P\bigl(\zeta^N|_{[1,N-L]}\ne
                  \zeta|_{[1,N-L]}\bigr)\le d(u)^L.
\]
On the agreement event, the last $L$ sites introduce at most $L+1$
recorded carrier arrows.  To their left the physical inputs agree.
As $\zeta\sim\nu$ and the next engine uses independent vertex
choices, the erasing-block coupling following
Lemma~\ref{lem:scalar-clearing} removes every such bounded entrance
before any fixed output window, uniformly in its colors.
The limits of the two sides therefore differ, on a cylinder
function $f$, by at most $2\|f\|_\infty d(u)^L$.
Letting $L\to\infty$ proves the assertion. 
\end{proof}

\begin{proof}[Proof of Theorem~\ref{thm:main-stationary-construction}]
Let $\nu_0=\delta_{\boldsymbol1}$ and
$\nu_m=\nu_{m-1}\mathfrak E_m$.  The local construction above
identifies $\nu_r=\mu_{\mathrm{eng}}$ with the limit of the output of the combination of engines, as $N \to \infty$. The occupation indicators under
$\nu_m$ agree outside a finite prefix with a stationary ergodic
binary sequence of mean $\rho_m$.  Since
$\nu_0\mathfrak T_u=\nu_0$ by
\eqref{eq:full-configuration-fundamental-invariant}, the preceding
proposition, applied successively with $\rho_{m-1}<\rho_m$, gives
\begin{equation}
\label{eq:engine-infinite-double-row-invariance}
 \mu_{\mathrm{eng}}\mathfrak T_u=\mu_{\mathrm{eng}},\qquad x<u<1.
\end{equation}
For a cylinder function $f$ supported on $[1,k]$, take
$u_\varepsilon=1-(1-q)\varepsilon$.  Since
$d(u_\varepsilon)=\varepsilon/(1+q\varepsilon)$,
\eqref{eq:fundamental-geometric-forgetting} allows us to truncate
at $k+2$, with error $O_f(\varepsilon^2)$.  By
Proposition~\ref{prop:double-row-to-half-space-ASEP},
\[
 \mathfrak T_{u_\varepsilon}f
 =\mathfrak T_{u_\varepsilon}^{k+2,0}f+O_f(\varepsilon^2)
 =f+2\varepsilon\mathcal L_\alpha f+O_f(\varepsilon^2).
\]
The finite right-end term does not act on $f$.  Using
\eqref{eq:engine-infinite-double-row-invariance}, dividing by
$2\varepsilon$, and sending $\varepsilon\downarrow0$ gives
$\mu_{\mathrm{eng}}(\mathcal L_\alpha f)=0$, which is the desired stationarity.
\end{proof}

\section{Boundary correlations}
This section proves Theorem~\ref{thm:two-site-marginal} by computing
the joint color distribution at the first two sites of the stationary
two-species ASEP on the half-line, using the engine construction of
Section~\ref{sec:queue-specialization-stationarity}.
Subsection~\ref{subsec:asep-local-specializations} collects the local
transition probabilities and moment identities needed for these
computations.  Subsection~\ref{subsec:low-high-asep} treats the low--high
density regime.  For the high--high regime,
Subsection~\ref{subsec:two-color-mixed-engine} introduces the mixed-crossing
representation and determines the crossing's output distribution;
Subsection~\ref{subsec:high-high} then uses these results to compute the
two-site probabilities.  Finally,
Subsection~\ref{subsec:two-species-large-alpha} treats \(\alpha\geq1\)
using the exact specialization constructed in
Section~\ref{sec:queue-specialization-stationarity}.

In each case, color projection reduces the computation of the full
two-site marginal to two probabilities.  We give this common reduction
here and use it throughout the section.

\paragraph{Reduction to two probabilities.}
We use the notation \(\bar\rho\), \(J(\rho)\), and \(b_{u,v}\) from
Subsection~\ref{subsec:half-space-mpa}.
The color projections merging \(\{1,2\}\) and \(\{0,2\}\) have stationary
laws \(\mu^{\alpha,\rho_2}\) and \(\mu^{\alpha,\rho_1}\), respectively,
whose two-site marginals are given by
Proposition~\ref{prop:one-species-left-edge}.
After normalization, the nine probabilities \(p_{ij}\) have eight degrees
of freedom, and the two projections give six independent constraints.
For convenience, we compute \(p_{22}\) and \(p_{20}\); all other entries
then follow from the two projected marginals.  It therefore suffices
to establish
\begin{equation}
\label{eq:two-colored-targets}
 p_{22}=\frac{J_1-J_2}{\alpha^2}
       \bigl(J_1\theta^2-qJ_2\bigr),
 \qquad
 p_{20}=\frac{J_1-J_2}{\alpha^2}
       \bigl(J_1\theta(1-\theta)+qJ_2\bigr).
\end{equation}

\subsection{ASEP specializations of the boundary and row vertex weights}
\label{subsec:asep-local-specializations}

We collect the identities for the laws defined in
\eqref{eq:q-negative-binomial-law} and
\eqref{eq:q-Poisson-law} that will be used repeatedly below.

\begin{lemma}
\label{lem:22-q-negative-binomial-moments}
Let \(0\leq q<1\), \(s<1\), \(0\leq b<1\), and
\(\upsilon\sim q\text{-}\mathrm{NB}(s,b)\).  For \(|bz|<1\),
\begin{equation}
\label{eq:22-q-negative-binomial-pgf}
    \mathbb E[z^\upsilon]
    =\frac{(b;q)_\infty(sbz;q)_\infty}
           {(sb;q)_\infty(bz;q)_\infty}.
\end{equation}
Consequently, for \(m\in\mathbb Z_{\geq0}\),
\begin{equation}
\label{eq:22-q-negative-binomial-moments}
    \mathbb E[q^{m\upsilon}]
    =\frac{(b;q)_m}{(sb;q)_m}.
\end{equation}
More generally, for \(\ell\in\mathbb Z_{\geq1}\),
\begin{equation}
\label{eq:22-q-negative-binomial-factorial-moments}
\mathbb E\!\left[
 \mathbf 1_{\{\upsilon\geq\ell\}}
 \prod_{r=0}^{\ell-1}(1-q^{\upsilon-r})
\right]
=\frac{b^\ell(s;q)_\ell}{(sb;q)_\ell}.
\end{equation}

In particular, if
\(\upsilon\sim q\text{-}\mathrm{Poi}(b)\), then
\begin{equation}
\label{eq:22-q-Poisson-useful-moments}
\begin{gathered}
    \mathbb E[q^{m\upsilon}]=(b;q)_m,
    \qquad
    \mathbb E[1-q^\upsilon]=b,\\
    \mathbb E[q^\upsilon(1-q^\upsilon)]=qb(1-b),
    \qquad
    \mathbb E\!\left[
      \mathbf 1_{\{\upsilon\geq2\}}(1-q^\upsilon)(1-q^{\upsilon-1})
    \right]=b^2.
\end{gathered}
\end{equation}
\end{lemma}

\begin{proof}
Equation~\eqref{eq:22-q-negative-binomial-pgf} is the \(q\)-binomial
theorem.  Setting \(z=q^m\) proves
\eqref{eq:22-q-negative-binomial-moments}.  Multiplying
\(p_{s,b}^{(q)}(k)\) by
\(\mathbf 1_{\{k\geq\ell\}}
 \prod_{r=0}^{\ell-1}(1-q^{k-r})\), shifting \(k\) by \(\ell\), and
applying the same theorem proves
\eqref{eq:22-q-negative-binomial-factorial-moments}.  The final identities
are the specialization \(s=0\).
\end{proof}

We use the boundary kernels \(\mathbb K^{(m)}_{\alpha,\rho}\) obtained by
the ASEP density specialization \eqref{eq:engine-density-specialization}.
The following statement records their output laws for \(m=1,2\).

\begin{lemma}
\label{lem:22-two-boundary-specialization}
Let \(0<\alpha,\rho<1\) with \(\alpha+\rho>1\).  Under the ASEP density
specialization, the boundary output \(H=(h_0,h_2)\) satisfies
\[
 h_0\sim q\text{-}\mathrm{Poi}(b_{\alpha,\rho}),
 \qquad h_2=0\quad\text{almost surely},
\]
for input \((\infty,0)\) when \(m=1\), and for input \((g,\infty)\) when
\(m=2\), for every \(g\in\mathbb Z_{\geq0}\).  In the latter case, the
output law does not depend on \(g\); consequently, for a random input
\(g\), the boundary output is independent of \(g\).
\end{lemma}

\begin{proof}
This follows from Corollary~\ref{cor:queue-K-density-specialization}
with \(n=2\) and \(m=1,2\), respectively.
\end{proof}

In both density regimes, the reflected row uses the kernel
\(\widehat L^{\mathrm{ver}}_\rho\) from
Section~\ref{sec:queue-specialization-stationarity}.
Its two-color transition probabilities are as follows.

\begin{lemma}
\label{lem:22-local-transition}
Fix \(0\leq q<1\) and \(\rho\in(0,1)\).  Given \(i\in\mathcal C_2\) and a
horizontal queue state \((h_0,h_2)\in\mathbb Z_{\geq0}^2\), let
\(\mathsf p_{q,\rho}(i,j\mid h_0,h_2)\) be the probability of output color
\(j\in\mathcal C_2\).  Rows and columns in the following table are displayed
in the order \(0,2,1\):
\begin{equation}
\label{eq:22-local-transition-table}
\begingroup
\renewcommand{\arraystretch}{1.55}
\begin{array}{c|ccc}
\mathsf p_{q,\rho}(i,j\mid h_0,h_2)
    &j=0&j=2&j=1\\
\hline
i=0
    &\rho
    &0
    &\bar\rho
\\
i=2
    &\rho(1-q^{h_0})
    &\rho q^{h_0}
    &\bar\rho
\\
i=1
    &\rho(1-q^{h_0})
    &\rho q^{h_0}(1-q^{h_2})
    &\bar\rho+\rho q^{h_0+h_2}
\end{array}
\endgroup
\end{equation}
With the hole-state vectors \(\widehat E_i\)
from Example~\ref{ex:two-color-hole-spin-one}, the arrow configuration is
updated deterministically according to
\[
 (h_0',h_2')=(h_0,h_2)+\widehat E_i-\widehat E_j.
\]
\end{lemma}

\subsection{Two-species low--high-density stationary measure}
\label{subsec:low-high-asep}
Throughout this subsection we assume regime (L) of
Theorem~\ref{thm:two-site-marginal}; in particular,
\[
 \rho_1=\alpha,\qquad 0<\alpha\leq\frac12,\qquad 1-\alpha<\rho_2<1.
\]
We use the construction in Figure~\ref{fig:low-high-engine}, in which the
engine for color~\(2\) receives a Bernoulli input of density \(\alpha\).
The following lemma obtains this representation from the two-engine
construction by taking the first engine to its critical density.

\begin{lemma}
\label{lem:critical-bernoulli-reduction}
For \(1-\alpha<\rho<\rho_2\), let \(\mu_{N,\rho}\) be the output law of the
finite zero-input array with \(N\) sites and engine densities
\((\rho,\rho_2)\).  For each fixed \(N\), the output law of the first engine
converges to \(\operatorname{Ber}(\alpha)^{\otimes N}\) as
\(\rho\downarrow1-\alpha\).  Moreover, the two iterated limits
\[
 \lim_{\rho\downarrow1-\alpha}\lim_{N\to\infty}\mu_{N,\rho}
 =\lim_{N\to\infty}\lim_{\rho\downarrow1-\alpha}\mu_{N,\rho}
 =\mu^{\alpha,\alpha,\rho_2}
\]
exist and agree locally.  Their common limit is stationary for the
two-species half-space ASEP and is given by the single-engine construction
in Figure~\ref{fig:low-high-engine}, with input law \(\pi^\alpha\) independent
of the randomness of the remaining engine.
\end{lemma}

\begin{figure}[htbp]
\centering
\begin{tikzpicture}[
baseline={([yshift=-.5ex]current bounding box.center)},
scale=.85
]
\foreach \x/\instate/\midstate/\outstate in {
  2/{\eta_1}/{\pi_1}/{\mu_1},
  3/{\eta_2}/{\pi_2}/{\mu_2},
  4/{\eta_3}/{\pi_3}/{\mu_3},
  5/{\cdots}/{\cdots}/{\cdots},
  6/{\cdots}/{\cdots}/{\cdots},
  7/{\cdots}/{\cdots}/{\cdots}}
{
  \draw[paper wire,->] (\x,-.75)--(\x,1.75);
  \node[below] at (\x,-.75) {$\instate$};
  \node[anchor=east,xshift=-2pt,inner sep=0pt]
      at (\x,.5) {$\midstate$};
  \node[above] at (\x,1.75) {$\outstate$};
}

\draw[paper normal] (8,0) -- (1,0) -- (0,.5);
\draw[paper hole,->] (0,.5) -- (1,1) -- (8,1);
\node[paper boundary] at (0,.5) {};

\node[right=3pt] at (8,0) {$(0,\infty)$};
\node[right=36pt] at (8,0) {$\scriptstyle\rho_2$};
\node[right=3pt,text=holecolor] at (8,1) {$\sum H$};
\end{tikzpicture}
\caption{The limiting construction in
Lemma~\ref{lem:critical-bernoulli-reduction}.  The first engine
is replaced by the Bernoulli input \(\eta\sim\pi^\alpha\), and the remaining
engine has density \(\rho_2\).}
\label{fig:low-high-engine}
\end{figure}
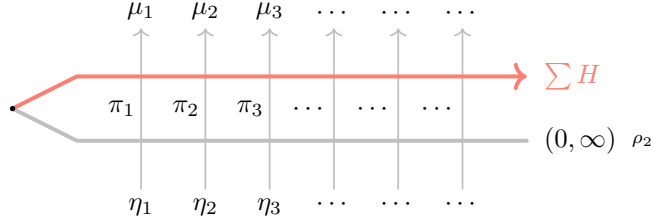

\begin{proof}
The first engine starts its reflected row with
\[
 h_0\sim q\text{-}\mathrm{Poi}(b_{\alpha,\rho}).
\]
As \(\rho\downarrow1-\alpha\), we have \(b_{\alpha,\rho}\uparrow1\), so
\(h_0\) tends to infinity in probability.  At fixed \(N\), the hole count
can decrease by at most \(N\), and hence tends to infinity throughout the
reflected row.  In the one-color restriction of
\eqref{eq:22-local-transition-table}, the probability of output color~\(1\)
consequently tends to \(1-\rho\to\alpha\), regardless of the incoming
color.  The independent randomness at different vertices therefore gives
the product law \(\operatorname{Ber}(\alpha)^{\otimes N}\).  Applying the
second engine yields the finite version of Figure~\ref{fig:low-high-engine}.

To justify the exchange of limits, choose
\(1-\alpha<\bar u<\rho_2\) and restrict to
\(1-\alpha<\rho\leq\bar u\).  Since \(\rho>\alpha\), we have
\(b_{\alpha,\rho}\geq b_{\rho,\rho}\).  The stationary one-color Martin
coupling with densities \(1-\rho\) and \(\rho\) has carrier law
\(q\text{-}\mathrm{Poi}(b_{\rho,\rho})\) and Bernoulli product occupation
output of density \(\rho\), by
Theorem~\ref{thm:martin-stationary-queue} and
Proposition~\ref{prop:martin-yang-baxter-steady-state}.  Since the
\(q\)-Poisson laws are stochastically increasing in their parameter, the
common graphical realization couples the first-engine boundary carrier
above this stationary carrier and its occupation output below the
stationary Bernoulli sequence.  It is therefore also dominated by a
Bernoulli product sequence of density \(\bar u\).  Since
\(\bar u<\rho_2\), the erasing-block coupling following
Lemma~\ref{lem:scalar-clearing} gives a finite-engine approximation for
the second engine on every fixed output window, uniformly in \(\rho\).
For each fixed truncation, the Bernoulli convergence established above
applies.  Removing the truncation proves that the two iterated limits
agree.

For every \(1-\alpha<\rho<\rho_2\), the limit as \(N\to\infty\) is
stationary by Theorem~\ref{thm:engine-array-stationarity}.  Since the
ASEP generator maps cylinder functions to cylinder functions, the
stationarity identity passes to the local limit
\(\rho\downarrow1-\alpha\).  Consequently, the limiting measure is
stationary.
\end{proof}

Let \(\eta\), \(\pi\), and \(\mu\) denote the configurations below, between,
and above the two rows in Figure~\ref{fig:low-high-engine}, respectively.
The input \(\eta\) has Bernoulli product law \(\pi^\alpha\).  Reading from
right to left, the unreflected row realizes Martin's queue with arrival
density \(\alpha\) and service density \(\rho_2\).
Proposition~\ref{prop:martin-yang-baxter-steady-state} therefore gives
\[
 \pi\sim\pi^{\alpha,\rho_2},
\]
where the full-space law is restricted to the positive sites.  Its
two-site marginal is given by Theorem~\ref{thm:full-space-two-site-law}.

At the left boundary, the unreflected row supplies the input
\((g,\infty)\).  Given the entire unreflected row, the boundary kernel
depends only on this input and, by
Lemma~\ref{lem:22-two-boundary-specialization}, returns \((h_0,0)\), where
\[
 h_0\sim q\text{-}\mathrm{Poi}(b_{\alpha,\rho_2}).
\]
This conditional output law is the same for every \(g\), so \(h_0\) is
independent of \(\pi\).  Starting from \((h_0,0)\), the reflected row is
sampled from left to right using fresh, independent vertex randomness.
At site \(i\), the vertical input is \(\pi(i)\) and the output color
\(\mu(i)\) is sampled according to \eqref{eq:22-local-transition-table},
with \(\rho=\rho_2\) and the current horizontal arrow configuration.
The latter is updated by the conservation rule in
Lemma~\ref{lem:22-local-transition}.  By
Lemma~\ref{lem:critical-bernoulli-reduction}, the resulting configuration
has law \(\mu\sim\mu^{\alpha,\alpha,\rho_2}\).

For the two direct calculations, we use the elementary identities
\begin{equation}
\label{eq:22-cross-parameter-identities}
    1-b_{u,v}=\frac{u+v-1}{uv},
    \qquad
    1-qb_{u,v}=\frac{uv-q(1-u)(1-v)}{uv}
\end{equation}
and the low--high identity
\begin{equation}
J_1-J_2
 =\alpha(1-\alpha)-\rho_2\bar\rho_2
  =(\rho_2-\alpha)(\alpha+\rho_2-1).
\label{eq:low-high-current-difference}
\end{equation}

We first compute \(p_{22}\).

\begin{proposition}
\label{prop:low-high-p22}
The first identity in \eqref{eq:two-colored-targets} holds under
regime (L) of Theorem~\ref{thm:two-site-marginal}.
\end{proposition}

\begin{proof}
Since the boundary output has no color--\(2\) arrows,
\eqref{eq:22-local-transition-table} forces \(\pi(1)=\pi(2)=2\) whenever
\(\mu(1)=\mu(2)=2\).  Each \(2\to2\) transition has weight
\(\rho_2q^{h_0}\) and leaves \((h_0,0)\) unchanged, so independence of
\(\pi\) and \(h_0\) gives
\begin{equation}
\label{eq:low-high-p22-unsimplified}
p_{22}=\rho_2^2
  \mathbb P\bigl(\pi(1)=2,\pi(2)=2\bigr)
  \mathbb E[q^{2h_0}].
\end{equation}
Theorem~\ref{thm:full-space-two-site-law} gives
\[
\mathbb P\bigl(\pi(1)=2,\pi(2)=2\bigr)
=(\rho_2-\alpha)
 \bigl((1-\alpha)\rho_2-q\alpha\bar\rho_2\bigr).
\]
Substituting this and the moment formula
\eqref{eq:22-q-Poisson-useful-moments} into
\eqref{eq:low-high-p22-unsimplified}, and using
\eqref{eq:low-high-current-difference} and
\eqref{eq:22-cross-parameter-identities}, proves the claim since
\(\theta=\rho_2+q\bar\rho_2\) and
\[
\bigl(\alpha\rho_2-q(1-\alpha)\bar\rho_2\bigr)
 \bigl((1-\alpha)\rho_2-q\alpha\bar\rho_2\bigr)
=\alpha(1-\alpha)(\rho_2+q\bar\rho_2)^2
 -q\rho_2\bar\rho_2=J_1\theta^2-qJ_2.
\]
\end{proof}

We next compute \(p_{20}\).

\begin{proposition}
\label{prop:low-high-p20}
The second identity in \eqref{eq:two-colored-targets} holds under
regime (L) of Theorem~\ref{thm:two-site-marginal}.
\end{proposition}

\begin{proof}
As in the preceding proof, the first output \(2\) requires \(\pi(1)=2\)
and leaves \((h_0,0)\) unchanged.  Separating the cases \(\pi(2)=0\) and
\(\pi(2)\in\{1,2\}\), we obtain from \eqref{eq:22-local-transition-table}
and independence of \(\pi\) and \(h_0\)
\begin{align}
p_{20}
&=\rho_2^2\Bigl\{
  \mathbb P\bigl(\pi(1)=2,\pi(2)=0\bigr)\mathbb E[q^{h_0}]
  +\mathbb P\bigl(\pi(1)=2,\pi(2)\in\{1,2\}\bigr)
   \mathbb E\bigl[q^{h_0}(1-q^{h_0})\bigr]\Bigr\}.
\label{eq:low-high-p20-unsimplified}
\end{align}
Theorem~\ref{thm:full-space-two-site-law} gives
\begin{align*}
\mathbb P\bigl(\pi(1)=2,\pi(2)\in\{1,2\}\bigr)
 &=(\rho_2-\alpha)\bigl(\rho_2+\alpha(1-\theta)\bigr),\\
\mathbb P\bigl(\pi(1)=2,\pi(2)=0\bigr)
 &=(\rho_2-\alpha)\bigl[1-\rho_2-\alpha(1-\theta)\bigr].
\end{align*}
Substituting these probabilities and the moments
\eqref{eq:22-q-Poisson-useful-moments} into
\eqref{eq:low-high-p20-unsimplified} yields
\begin{equation}
p_{20}=\rho_2^2(1-b_{\alpha,\rho_2})(\rho_2-\alpha)
  \Bigl[1-(1-qb_{\alpha,\rho_2})
    \bigl(\rho_2+\alpha(1-\theta)\bigr)\Bigr].
\end{equation}
Applying \eqref{eq:low-high-current-difference},
\eqref{eq:22-cross-parameter-identities}, and
\(1-\theta=(1-q)\bar\rho_2\) gives the second identity in
\eqref{eq:two-colored-targets}.
\end{proof}

Propositions~\ref{prop:low-high-p22} and~\ref{prop:low-high-p20}
establish \eqref{eq:two-colored-targets} in regime (L).  The reduction
at the beginning of this section therefore proves regime (L) of
Theorem~\ref{thm:two-site-marginal}.

\subsection{The mixed crossing and the two-engine construction}
\label{subsec:two-color-mixed-engine}

Direct computation from the original two-engine construction is difficult,
since the upper engine receives the lower engine's half-space output.
Introducing a mixed crossing allows us to bring the two unreflected rows
together, where they form Martin's two-level queue.  We can then use its
known limiting output law and conditional queue distributions,
substantially simplifying the calculation.

Throughout this subsection, assume
\[
 0<\alpha<1,\qquad
 \frac12<\rho_1<\rho_2<1,\qquad
 \alpha+\rho_1>1.
\]
The lower \(m=1\) and upper \(m=2\) engines have densities \(\rho_1\) and
\(\rho_2\), respectively.  For a row with spectral parameter \(x\) and spin
\(s\), the service density is \(\rho=(1-sx)^{-1}\).  The ASEP specialization
\eqref{eq:engine-density-specialization} keeps this density fixed as
\(x\downarrow0\), which determines the spin parameters below.  Write
\((x_1,\tau)\) and \((x_2,\sigma)\) for the parameters of the two engines,
and set \(a=\alpha/(\alpha-1)\).  The resulting specialization is
\begin{equation}
\label{eq:22-queue-specialization}
 z=x_1x_2,
 \qquad
 \tau:=-\frac{\bar\rho_1}{\rho_1x_1},
 \qquad
 \sigma:=-\frac{\bar\rho_2}{\rho_2x_2},
 \qquad x_1,x_2\downarrow0.
\end{equation}
The following proposition shows that this rearrangement preserves the
local output law in the infinite-volume limit.

\begin{proposition}
\label{prop:22-mixed-crossing-representation}
Under the assumptions above and the specialization
\eqref{eq:22-queue-specialization}, the two constructions in
Figure~\ref{fig:mixed-engine-limit} have the same local limiting output law.
\end{proposition}
\begin{figure}[htbp]
\centering
\begin{tikzpicture}[
baseline={([yshift=-.5ex]current bounding box.center)},
scale=.58,
thickgray/.style={paper normal},
thicksalmon/.style={paper hole},
thingray/.style={paper wire},
every node/.style={inner sep=1.2pt}
]

\foreach \x in {2,...,7}
  \draw[thingray,->] (\x,-.7) -- (\x,3.7);

\draw[thickgray] (8,2) -- (1,2) -- (0,2.5);
\draw[thicksalmon,->] (0,2.5) -- (1,3) -- (8,3);
\fill[black] (0,2.5) circle (1.7pt);

\draw[thickgray] (8,0) -- (1,0) -- (0,.5);
\draw[thicksalmon,->] (0,.5) -- (1,1) -- (8,1);
\fill[black] (0,.5) circle (1.7pt);

\foreach \x in {2,3,4,6,7}
  \node[below] at (\x,-.7) {$0$};
\node[below] at (5,-.7) {$\cdots$};
\foreach \x/\lab in {2/\mu_1,3/\mu_2,4/\mu_3,5/\cdots,6/\cdots,7/\cdots}
  \node[above] at (\x,3.7) {$\lab$};

\node[align=center] at (10,2.45)
  {\scriptsize same local\\[-1pt]\scriptsize output law};
\node at (10,1.35) {$\Longleftrightarrow$};

\foreach \x in {14,...,19}
  \draw[thingray,->] (\x,-.7) -- (\x,3.7);

\draw[thickgray] (20,1) -- (13,1) -- (12,2.25);
\draw[thicksalmon,->] (12,2.25) -- (13,3) -- (20,3);
\fill[black] (12,2.25) circle (1.7pt);

\draw[thickgray] (20,0) -- (13,0) -- (12,.75);
\draw[thicksalmon,->] (12,.75) -- (13,2) -- (20,2);
\fill[black] (12,.75) circle (1.7pt);
\fill[black] (12.6,1.5) circle (1.9pt);

\foreach \x in {14,15,16,18,19}
  \node[below] at (\x,-.7) {$0$};
\node[below] at (17,-.7) {$\cdots$};
\foreach \x/\lab in {14/\mu_1,15/\mu_2,16/\mu_3,17/\cdots,18/\cdots,19/\cdots}
  \node[above] at (\x,3.7) {$\lab$};
\end{tikzpicture}
\caption{Local limits of the two diagrams in Figure~\ref{fig:mixed-engine}:
the original two-engine construction (left) and the representation with a
mixed vertex between the boundary rows (right).  Both have the same
stationary output law.}
\label{fig:mixed-engine-limit}
\end{figure}
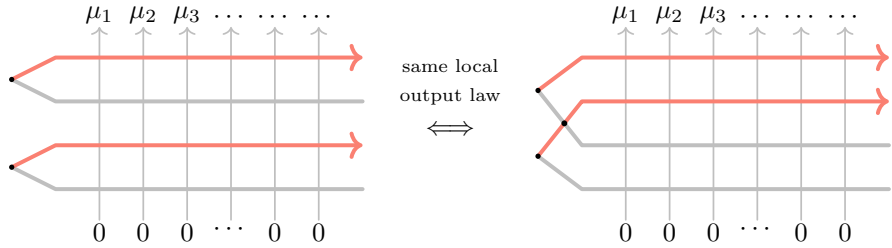

\begin{proof}
Fix \(N\), and let \(\widetilde{\mathfrak H}^{(N)}_{\mathrm R}\) and
\(\widetilde{\mathfrak H}^{(N)}_{\mathrm L}\) denote the partition functions
in Figure~\ref{fig:mixed-engine}, with the crossing at the right and left
endpoints, respectively.  For fixed exposed arrow configurations,
successive applications of the continued Yang--Baxter identity of
Proposition~\ref{prop:continued-colored-ybe} give
\begin{equation}
\label{eq:finite-two-engine-crossing-train}
 \widetilde{\mathfrak H}^{(N)}_{\mathrm R}
 =\widetilde{\mathfrak H}^{(N)}_{\mathrm L}.
\end{equation}
At fixed \(N\), we now take the specialization
\eqref{eq:22-queue-specialization}.  The limiting kernels are stochastic by
Propositions~\ref{prop:density-array-convergence}
and~\ref{prop:22-infinite-stack-specialization}, so
\eqref{eq:finite-two-engine-crossing-train} gives equality of the
finite-volume output laws.
\begin{figure}[htbp]
\centering
\begin{tikzpicture}[
baseline={([yshift=-.5ex]current bounding box.center)},
scale=0.58,
every node/.style={inner sep=1.2pt}
]

\foreach \x in {2,...,7}
    \draw[paper wire,->] (\x,-.7) -- (\x,3.7);

\draw[paper normal]
    (8.5,1) -- (8,1.5) -- (7.5,2)
    -- (1,2) -- (0,2.5);

\draw[paper hole,->]
    (0,2.5) -- (1,3) -- (8,3);

\draw[paper normal]
    (8,0) -- (1,0) -- (0,.5);

\draw[paper hole,->]
    (0,.5) -- (1,1) -- (7.5,1)
    -- (8,1.5) -- (8.5,2);

\node[paper boundary] at (0,2.5) {};
\node[paper boundary] at (0,.5)  {};

\node[right=2pt] at (8,0) {$(\infty,0)$};
\node[right=36pt] at (8,0) {$(x_1,\tau)$};
\node[right=2pt] at (8.5,1) {$(0,\infty)$};
\node[right=36pt] at (8.5,1) {$(x_2,\sigma)$};
\node[right=2pt,text=holecolor] at (8.5,2)
    {$\sum H^{(1)}$};
\node[right=2pt,text=holecolor] at (8,3)
    {$\sum H^{(2)}$};

\node[below] at (2,-.7) {$0$};
\node[below] at (3,-.7) {$0$};
\node[below] at (4,-.7) {$0$};
\node[below] at (5,-.7) {$\cdots$};
\node[below] at (6,-.7) {$\cdots$};
\node[below] at (7,-.7) {$0$};

\node[above] at (2,3.7) {$\mu_1$};
\node[above] at (3,3.7) {$\mu_2$};
\node[above] at (4,3.7) {$\mu_3$};
\node[above] at (5,3.7) {$\cdots$};
\node[above] at (6,3.7) {$\cdots$};
\node[above] at (7,3.7) {$\mu_N$};

\node at (11.8,1.5) {$=$};

\begin{scope}[xshift=1.25cm]

\foreach \x in {14,...,19}
    \draw[paper wire,->] (\x,-.7) -- (\x,3.7);

\draw[paper normal]
    (20,1) -- (13,1) -- (12,2.25);

\draw[paper hole,->]
    (12,2.25) -- (13,3) -- (20,3);

\draw[paper normal]
    (20,0) -- (13,0) -- (12,.75);

\draw[paper hole,->]
    (12,.75) -- (13,2) -- (20,2);
\node[paper boundary] at (12,2.25) {};
\node[paper boundary] at (12,.75)  {};

\node[right=2pt] at (20,0) {$(\infty,0)$};
\node[right=36pt] at (20,0) {$(x_1,\tau)$};
\node[right=2pt] at (20,1) {$(0,\infty)$};
\node[right=36pt] at (20,1) {$(x_2,\sigma)$};
\node[right=2pt,text=holecolor] at (20,2)
    {$\sum H^{(1)}$};
\node[right=2pt,text=holecolor] at (20,3)
    {$\sum H^{(2)}$};

\node[below] at (14,-.7) {$0$};
\node[below] at (15,-.7) {$0$};
\node[below] at (16,-.7) {$0$};
\node[below] at (17,-.7) {$\cdots$};
\node[below] at (18,-.7) {$\cdots$};
\node[below] at (19,-.7) {$0$};

\node[above] at (14,3.7) {$\mu_1$};
\node[above] at (15,3.7) {$\mu_2$};
\node[above] at (16,3.7) {$\mu_3$};
\node[above] at (17,3.7) {$\cdots$};
\node[above] at (18,3.7) {$\cdots$};
\node[above] at (19,3.7) {$\mu_N$};

\end{scope}
\end{tikzpicture}
\caption{The finite-volume Yang--Baxter identity at cutoff \(N\).
The mixed vertex is moved from the right endpoint (left) to the position
between the two boundary rows (right).}
\label{fig:mixed-engine}
\end{figure}
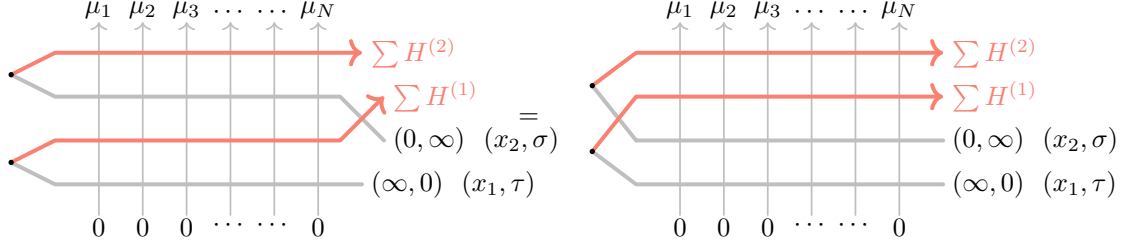

In the diagram with the crossing at the right endpoint, its inputs are
\(G=(0,\infty)\) and \(H=(h_0,0)\).  The case \(g_1=0\) of
Proposition~\ref{prop:22-infinite-stack-specialization} gives
\[
 g_1'=h_0'\sim q\text{-}\mathrm{Poi}(b_{\rho_1,\rho_2}),
\]
independently of \(h_0\).  Since \(b_{\rho_1,\rho_2}<1\) and this law does
not depend on \(N\), the laws of \(g_1'\) form a tight family.

The crossing changes the upper engine's entrance from \((0,\infty)\) to
\((g_1',\infty)\).  Since the laws of \(g_1'\) are tight,
given \(\varepsilon>0\), choose \(B\) so that
\(\mathbb P(g_1'>B)<\varepsilon\).  On \(\{g_1'\leq B\}\), the
erasing-block coupling following Lemma~\ref{lem:scalar-clearing} removes
the additional entrance before any fixed output window, uniformly over
its colors.  Letting \(N\to\infty\) and then \(\varepsilon\downarrow0\)
shows that this change does not affect the limiting output law.  We may
therefore let \(N\to\infty\) in
\eqref{eq:finite-two-engine-crossing-train} and conclude that the two
constructions in Figure~\ref{fig:mixed-engine-limit} have the same local
limiting output law.
\end{proof}

In the right panel of Figure~\ref{fig:mixed-engine-limit}, the mixed
vertex receives \(G=(g_1,\infty)\) from Martin's queue and \(H=(h_0,0)\)
from the lower boundary.  We first compute its transition law for fixed
\(g_1,h_0\), by applying the specialization
\eqref{eq:22-queue-specialization} to the weight
\(\widehat W_{z,\sigma,\tau}\) from
Example~\ref{ex:two-color-doubly-fused-hole}.

\begin{proposition}
\label{prop:22-infinite-stack-specialization}
Let \(\rho_1,\rho_2\) and the parameters be as in
\eqref{eq:22-queue-specialization}, and fix inputs
\(G=(g,\infty)\) and \(H=(h,0)\), with \(g,h\in\mathbb Z_{\geq0}\).
Under this specialization, the output is obtained by sampling
\[
 h_0'\sim q\text{-}\mathrm{Poi}(b_{\rho_1,\rho_2}),
 \qquad
 \mathbb P(\psi=t)=\mathsf M_q(t\mid g,h),
\]
independently, where
\begin{equation}
\label{eq:22-matching-kernel}
\mathsf M_q(t\mid g,h)
:=
\mathbf1_{\{0\leq t\leq\min(g,h)\}}
q^{(g-t)(h-t)}
\frac{(q;q)_g(q;q)_h}
{(q;q)_t(q;q)_{g-t}(q;q)_{h-t}}.
\end{equation}
The output vectors are \(G'=(g_1',\infty)\) and \(H'=(h_0',h_2')\), where
\begin{equation}
\label{eq:22-infinite-stack-output-map}
 h_2'=h-\psi,
 \qquad
 g_1'=h_0'+g-\psi.
\end{equation}
\end{proposition}

\begin{proof}
Fix integers \(r\geq0\) and \(0\leq t\leq\min(g,h)\).  At a finite
cutoff \(R\geq h-t\), the corresponding arrow configurations are
\[
\begin{aligned}
 G&=(g,R),& H&=(h,0),\\
 G'&=(r+g-t,R-h+t),& H'&=(r,h-t).
\end{aligned}
\]
We first send \(R\to\infty\) and write
\begin{align}
&\mathbb{W}^{[2]}_{z,\sigma,\tau}
\bigl((g,\infty),(h,0);(r+g-t,\infty),(r,h-t)\bigr)
\notag\\
&\quad:=\lim_{R\to\infty}
\widehat W_{z,\sigma,\tau}
\bigl((g,R),(h,0);(r+g-t,R-h+t),(r,h-t)\bigr).
\label{eq:def-22-infinite-stack-crossing}
\end{align}
We will show that the subsequent density specialization gives
\begin{equation}
\label{eq:22-infinite-stack-kernel}
\lim_{x_1,x_2\downarrow0}
\mathbb{W}^{[2]}_{z,\sigma,\tau}
\bigl((g,\infty),(h,0);(r+g-t,\infty),(r,h-t)\bigr)
=p_{b_{\rho_1,\rho_2}}^{(q)}(r)
\mathsf M_q(t\mid g,h).
\end{equation}
For fixed finite values of all output coordinates other than the infinite
coordinate, the limiting weight vanishes unless the output has this form.

Since \(h_2=0\), only \(P_2=0\) contributes to
\eqref{eq:22-fully-fused-weight}.  In the summation notation of that
formula, \(p=P_1+P_2=P_1\).  Sending \(R\to\infty\) gives
\[
\mathbb{W}^{[2]}_{z,\sigma,\tau}
\bigl((g,\infty),(h,0);(r+g-t,\infty),(r,h-t)\bigr)
=
\mathcal N_{z,\sigma,\tau}(\sigma z)^r
\frac{(q;q)_h}{(q;q)_r(q;q)_{h-t}}
S_{g,h,t,r}(z,\tau),
\]
where
\begin{align*}
S_{g,h,t,r}(z,\tau)
&:=\sum_{p=0}^{r+g-t}
(-1)^{p+r}
q^{\binom r2+(h-t)(r+g-t)+tp+\binom{p+1}2}
\frac{(z/q;q)_p}{(z;q)_{h+p}}
\notag\\
&\quad\times
\frac{(\tau q^{1-r-h+t};q)_{r+g-t-p}}
     {(q;q)_{r+g-t-p}}
\frac{(\tau q^{1-h-p};q)_p}{(q;q)_p}.
\end{align*}
The terminating \(q\)-Chu--Vandermonde identity gives the coefficient
extraction
\[
 \lim_{\tau\to-\infty}\tau^{-r}S_{g,h,t,r}(0,\tau)
 =q^{(g-t)(h-t)}\begin{bmatrix}g\\t\end{bmatrix}_{\!q}.
\]
Equivalently, all powers of \(\tau\) above \(r\) cancel at \(z=0\), and
the coefficient of \(\tau^r\) is the right-hand side.  At nonzero \(z\),
every excess power of \(\tau\) comes with a power of
\(\tau z=O(x_2)\); hence the same limit holds along the joint
specialization \eqref{eq:22-queue-specialization}.  Since
\[
 (\sigma z)^r=b_{\rho_1,\rho_2}^{\,r}\tau^{-r},
 \qquad
 \mathcal N_{z,\sigma,\tau}
 \longrightarrow(b_{\rho_1,\rho_2};q)_\infty,
\]
we obtain \eqref{eq:22-infinite-stack-kernel}.  Finally,
\begin{equation}
 \sum_{t=0}^{\min(g,h)}
 q^{(g-t)(h-t)}
 \frac{(q;q)_g(q;q)_h}
 {(q;q)_t(q;q)_{g-t}(q;q)_{h-t}}=1
 \label{eq:22-matching-kernel-normalization}
\end{equation}
is another terminating \(q\)-Chu--Vandermonde identity.  Thus
\(\mathsf M_q\) is a probability kernel, and the factorization in
\eqref{eq:22-infinite-stack-kernel} gives the stated independent sampling
rule and output map.  The \(q=0\) statement is the continuous extension
of the limiting kernel; one does not substitute \(q=0\)
into the prelimit expression containing \(z/q\).
\end{proof}

For the random inputs considered next, write
\[
 \omega:=g_1'-h_0'=g_1-\psi.
\]
We average the crossing law over a random finite bottom input and an
independent \(q\)-Poisson horizontal input.

\begin{theorem}
\label{thm:22-general-bottom-input}
Let \(g_1\) have distribution \(\mathbb P(g_1=g)=\gamma_g\) on
\(\mathbb Z_{\geq0}\), and let
\(h_0\sim q\text{-}\mathrm{Poi}(\beta)\), \(0\leq\beta<1\), be independent
of \(g_1\).  Apply the crossing of
Proposition~\ref{prop:22-infinite-stack-specialization} to
\(G=(g_1,\infty)\) and \(H=(h_0,0)\).
For \(v\in\mathbb Z_{\geq0}\), the output satisfies
\[
 \mathbb P(\omega=v)
 =(\beta;q)_v\sum_{g\geq v}\gamma_g\beta^{g-v}
   \begin{bmatrix}g\\v\end{bmatrix}_{\!q},
\]
and, whenever \(\mathbb P(\omega=v)>0\),
\[
 h_2'\mid\{\omega=v\}\sim q\text{-}\mathrm{Poi}(\beta q^v).
\]
Moreover, \(h_0'\sim q\text{-}\mathrm{Poi}(b_{\rho_1,\rho_2})\),
independently of \((h_2',\omega)\), and \(g_1'=h_0'+\omega\).
\end{theorem}

\begin{proof}
For \(v\in\mathbb Z_{\geq0}\), abbreviate the sum by
\begin{equation}
\label{eq:22-general-bottom-transform}
 \mathcal T_v^{g_1}(\beta)
 :=\sum_{g\geq v}\gamma_g\beta^{g-v}
   \begin{bmatrix}g\\v\end{bmatrix}_{\!q}.
\end{equation}
Fix \(r,v,k\in\mathbb Z_{\geq0}\), and write \(t=g-v\), \(h=t+k\).
Averaging the kernel \eqref{eq:22-infinite-stack-kernel} against
\(\gamma_g p_\beta^{(q)}(h)\) and summing over \(g\geq v\) gives
\begin{equation}
\mathbb P(g_1'=r+v,h_0'=r,h_2'=k)
=p_{b_{\rho_1,\rho_2}}^{(q)}(r)
 \frac{(\beta;q)_\infty\beta^kq^{kv}}{(q;q)_k}
 \mathcal T_v^{g_1}(\beta).
\label{eq:22-general-triple-law}
\end{equation}
Summing in \(r\) and \(k\), using the \(q\)-binomial theorem, gives
the stated law of \(\omega\).  Dividing the joint law of
\((\omega,h_2')\) by this marginal gives the conditional
\(q\)-Poisson law.  The separate factor
\(p_{b_{\rho_1,\rho_2}}^{(q)}(r)\) gives the law of \(h_0'\) and its
independence of \((h_2',\omega)\).
\end{proof}

For the conditioned queues used below, we need the case where the bottom
queue length is the sum of independent \(q\)-Poisson and Bernoulli variables.

\begin{corollary}
\label{cor:22-q-Poisson-Bernoulli-bottom-input}
Fix \(0\leq \gamma<1\), \(0\leq\beta<1\), and \(0\leq p\leq1\).  Let
\[
 \upsilon\sim q\text{-}\mathrm{Poi}(\gamma),
 \qquad \xi\sim\operatorname{Ber}(p),
 \qquad h_0\sim q\text{-}\mathrm{Poi}(\beta)
\]
be mutually independent, and take \(g_1=\upsilon+\xi\) in
Theorem~\ref{thm:22-general-bottom-input}.  The law of \(h_2'\) is a
mixture of the \(q\text{-}\mathrm{NB}(\gamma,\beta)\) and
\(q\text{-}\mathrm{NB}(\gamma,q\beta)\) distributions: for
\(k\in\mathbb Z_{\geq0}\),
\begin{equation}
\label{eq:22-H2-qNB-Bernoulli-law}
\mathbb P(h_2'=k)
=\bigl[1-p(1-\beta)\bigr]p_{\gamma,\beta}^{(q)}(k)
 +p(1-\beta)p_{\gamma,q\beta}^{(q)}(k).
\end{equation}
In particular, when \(p=0\), we have
\(h_2'\sim q\text{-}\mathrm{NB}(\gamma,\beta)\).
For every \(m\in\mathbb Z_{\geq0}\),
\begin{equation}
\mathbb E[q^{mh_2'}]
=\bigl[1-p(1-\beta)\bigr]\frac{(\beta;q)_m}{(\gamma\beta;q)_m}
 +p(1-\gamma\beta)\frac{(\beta;q)_{m+1}}{(\gamma\beta;q)_{m+1}}.
\label{eq:22-H2-qNB-Bernoulli-moments}
\end{equation}
\end{corollary}

\begin{proof}
Use the transform \eqref{eq:22-general-bottom-transform} for the inputs
\(\upsilon\) and \(\upsilon+1\).  The recursion
\[
 \begin{bmatrix}g+1\\v\end{bmatrix}_{\!q}
 =q^v\begin{bmatrix}g\\v\end{bmatrix}_{\!q}
  +\begin{bmatrix}g\\v-1\end{bmatrix}_{\!q}
\]
gives
\[
 \mathcal T_v^{\upsilon+1}(\beta)
 =\beta q^v\mathcal T_v^\upsilon(\beta)
  +\mathbf1_{\{v\geq1\}}\mathcal T_{v-1}^\upsilon(\beta).
\]
The \(q\)-binomial theorem yields
\[
 \mathcal T_v^\upsilon(\beta)
 =\frac{(\gamma;q)_\infty}{(\gamma\beta;q)_\infty}
   \frac{\gamma^v}{(q;q)_v}.
\]
Combining these two formulas for \(g_1=\upsilon+\xi\) in
\eqref{eq:22-general-triple-law} and summing in \(r\) gives, for
\(v,k\in\mathbb Z_{\geq0}\),
\begin{equation}
\label{eq:22-joint-output-q-Poisson-Bernoulli}
\mathbb P(\omega=v,h_2'=k)
=\frac{(\gamma;q)_\infty(\beta;q)_\infty}{(\gamma\beta;q)_\infty}
  \frac{\beta^kq^{vk}}{(q;q)_k}
  \left[
  (1-p+p\beta q^v)\frac{\gamma^v}{(q;q)_v}
  +p\mathbf1_{\{v\geq1\}}
       \frac{\gamma^{v-1}}{(q;q)_{v-1}}
  \right].
\end{equation}

Summing in \(v\) gives
\[
 \mathbb P(h_2'=k)
 =p_{\gamma,\beta}^{(q)}(k)
  \left[1-p(1-\beta)+p(1-\gamma\beta)q^k\right],
\]
which is the mixture \eqref{eq:22-H2-qNB-Bernoulli-law}.
Applying Lemma~\ref{lem:22-q-negative-binomial-moments} to this mixture
gives \eqref{eq:22-H2-qNB-Bernoulli-moments}.
\end{proof}

\subsection{Two-species high--high-density stationary measure}
\label{subsec:high-high}
Throughout this subsection we assume \(0<\alpha<1\),
\(\frac12<\rho_1<\rho_2<1\), and \(\alpha+\rho_1>1\).
We compute the two probabilities in \eqref{eq:two-colored-targets}
using the high--high-density construction with
\((u,v)=(\rho_1,\rho_2)\).  The two-engine zero-input construction is
stationary by Theorem~\ref{thm:engine-array-stationarity}, and
Proposition~\ref{prop:22-mixed-crossing-representation} gives the
representation obtained by Yang--Baxter moves that is used below.

The two unreflected rows in Figure~\ref{fig:mixed-engine-layers}
realize Martin's construction with densities \((\rho_1,\rho_2)\).
By Proposition~\ref{prop:martin-yang-baxter-steady-state}, their output
satisfies \(\pi\sim\pi^{\rho_1,\rho_2}\), with two-site marginal given by
Theorem~\ref{thm:full-space-two-site-law}.  The lower reflected row maps
\(\pi\) to \(\nu\), and the upper row maps \(\nu\) to \(\mu\), using
\eqref{eq:22-local-transition-table} with parameters \(\rho_1\) and
\(\rho_2\), respectively.
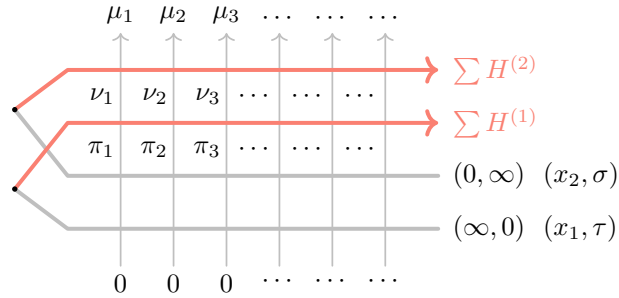
\begin{figure}[htbp]
\centering
\begin{tikzpicture}[
baseline={([yshift=-.5ex]current bounding box.center)},
scale=0.7
]
\foreach \x in {2,...,7}
    \draw[paper wire,->] (\x,-.7) -- (\x,3.7);

\draw[paper normal]
    (8,1) -- (1,1) -- (0,2.25);

\draw[paper hole,->]
    (0,2.25) -- (1,3) -- (8,3);

\draw[paper normal]
    (8,0) -- (1,0) -- (0,.75);

\draw[paper hole,->]
    (0,.75) -- (1,2) -- (8,2);

\node[paper boundary] at (0,2.25) {};
\node[paper boundary] at (0,.75)  {};

\node[right=2pt] at (8,0) {$(\infty,0)$};
\node[right=36pt] at (8,0) {$(x_1,\tau)$};
\node[right=2pt] at (8,1) {$(0,\infty)$};
\node[right=36pt] at (8,1) {$(x_2,\sigma)$};
\node[right=2pt,text=holecolor] at (8,2)
  {$\sum H^{(1)}$};
\node[right=2pt,text=holecolor] at (8,3)
  {$\sum H^{(2)}$};

\node[below] at (2,-.7) {0};
\node[below] at (3,-.7) {0};
\node[below] at (4,-.7) {0};
\node[below] at (5,-.7) {$\cdots$};
\node[below] at (6,-.7) {$\cdots$};
\node[below] at (7,-.7) {$\cdots$};

\node[anchor=east,xshift=-2pt,inner sep=0pt]
    at (2,1.5) {$\pi_1$};
\node[anchor=east,xshift=-2pt,inner sep=0pt]
    at (3,1.5) {$\pi_2$};
\node[anchor=east,xshift=-2pt,inner sep=0pt]
    at (4,1.5) {$\pi_3$};
\node[anchor=east,xshift=-2pt,inner sep=0pt]
    at (5,1.5) {$\cdots$};
\node[anchor=east,xshift=-2pt,inner sep=0pt]
    at (6,1.5) {$\cdots$};
\node[anchor=east,xshift=-2pt,inner sep=0pt]
    at (7,1.5) {$\cdots$};

\node[anchor=east,xshift=-2pt,inner sep=0pt]
    at (2,2.5) {$\nu_1$};
\node[anchor=east,xshift=-2pt,inner sep=0pt]
    at (3,2.5) {$\nu_2$};
\node[anchor=east,xshift=-2pt,inner sep=0pt]
    at (4,2.5) {$\nu_3$};
\node[anchor=east,xshift=-2pt,inner sep=0pt]
    at (5,2.5) {$\cdots$};
\node[anchor=east,xshift=-2pt,inner sep=0pt]
    at (6,2.5) {$\cdots$};
\node[anchor=east,xshift=-2pt,inner sep=0pt]
    at (7,2.5) {$\cdots$};

\node[above] at (2,3.7) {$\mu_1$};
\node[above] at (3,3.7) {$\mu_2$};
\node[above] at (4,3.7) {$\mu_3$};
\node[above] at (5,3.7) {$\cdots$};
\node[above] at (6,3.7) {$\cdots$};
\node[above] at (7,3.7) {$\cdots$};

\end{tikzpicture}
\caption{The high--high construction: \(\pi\) is the output of the
two unreflected rows, \(\nu\) the output of the lower reflected row,
and \(\mu\) the final configuration.}
\label{fig:mixed-engine-layers}
\end{figure}

We use \(\theta,\phi\) from \eqref{eq:two-site-abbreviations}
and the identity
\begin{equation}
\label{eq:22-cross-parameter-cocycle}
    b_{\bar\rho_1,\rho_2}\,b_{\alpha,\rho_1}
    =b_{\alpha,\rho_2}.
\end{equation}

For \(a,b\in\{0,1,2\}\), write
\[
 \pi_{ab}:=\mathbb P(\pi_1=a,\pi_2=b),
 \qquad
 \nu_{ab}:=\mathbb P(\nu_1=a,\nu_2=b).
\]

For \(ab\in\{11,12,21,10\}\), let \(g_1^{ab}\) have the law
of \(Q^{ab}\) in Proposition~\ref{prop:conditional-two-site-queue},
with \(r_{12}=b_{\bar\rho_1,\rho_2}\).  Conditional on the \(\pi\)-word
\(ab\), the mixed vertex has inputs \(G=(g_1^{ab},\infty)\) and
\(H=(h_0,0)\), where \(h_0\sim q\text{-}\mathrm{Poi}(b_{\alpha,\rho_1})\)
is independent of \(g_1^{ab}\) by
Lemma~\ref{lem:22-two-boundary-specialization}.  Its hole-indexed output
\(H'=(h_0^{\mathrm{low}},h_2^{ab})\) initializes the lower reflected row;
the upper reflected row starts from \((h_0^{\mathrm{up}},0)\).
Proposition~\ref{prop:22-infinite-stack-specialization} and
Lemma~\ref{lem:22-two-boundary-specialization} show that these two
color--\(0\) coordinates are independent, with laws
\begin{equation}
\label{eq:22-h0-laws}
    h_0^{\mathrm{low}}\sim q\text{-}\mathrm{Poi}(b_{\rho_1,\rho_2}),
    \qquad
    h_0^{\mathrm{up}}\sim q\text{-}\mathrm{Poi}(b_{\alpha,\rho_2}).
\end{equation}
Moreover, \(h_0^{\mathrm{low}}\) is independent of \(\pi\) and,
conditional on the \(\pi\)-word \(ab\), of \(h_2^{ab}\).
The upper boundary law does not depend on its input, so
\(h_0^{\mathrm{up}}\) is independent of \(\nu\).
Proposition~\ref{prop:conditional-two-site-queue} and
Corollary~\ref{cor:22-q-Poisson-Bernoulli-bottom-input} give
\begin{equation}
\label{eq:22-h2-11-12-21-laws}
h_2^{11}\sim q\text{-}\mathrm{NB}(b_{\bar\rho_1,\rho_2},b_{\alpha,\rho_1}),
\qquad
h_2^{12},h_2^{21}\sim q\text{-}\mathrm{NB}(qb_{\bar\rho_1,\rho_2},b_{\alpha,\rho_1}).
\end{equation}
Lemma~\ref{lem:22-q-negative-binomial-moments} therefore gives
\begin{align}
\mathbb E\!\left[
  \mathbf 1_{\{h_2^{11}\geq2\}}
  (1-q^{h_2^{11}})(1-q^{h_2^{11}-1})
 \right]
&=
\frac{b_{\alpha,\rho_1}^{2}
 (1-b_{\bar\rho_1,\rho_2})
 (1-qb_{\bar\rho_1,\rho_2})}
{(1-b_{\alpha,\rho_2})(1-qb_{\alpha,\rho_2})},
\label{eq:22-h2-11-double-weight}\\
&\mathbb E[1-q^{h_2^{12}}]
 =\mathbb E[1-q^{h_2^{21}}]
 =\frac{b_{\alpha,\rho_1}
      (1-qb_{\bar\rho_1,\rho_2})}
      {1-qb_{\alpha,\rho_2}},
\label{eq:22-h2-12-21-weight}\\
&\mathbb E[1-q^{h_2^{11}}]
 =\frac{b_{\alpha,\rho_1}
      (1-b_{\bar\rho_1,\rho_2})}
      {1-b_{\alpha,\rho_2}}.
\label{eq:22-h2-11-single-weight}
\end{align}
Finally, applying the same corollary to
\eqref{eq:22-queue-given-10}, with \(\beta_{10}\) as in
\eqref{eq:22-beta-10}, gives
\begin{equation}
\mathbb E[1-q^{h_2^{10}}]
 =b_{\alpha,\rho_1}(1-b_{\bar\rho_1,\rho_2})
\Biggl[
 \frac{1-\beta_{10}(1-b_{\alpha,\rho_1})}{1-b_{\alpha,\rho_2}}
 +\frac{q\beta_{10}(1-b_{\alpha,\rho_1})}{1-qb_{\alpha,\rho_2}}
\Biggr].
\label{eq:22-h2-10-weight}
\end{equation}

We now compute \(p_{22}\).

\begin{proposition}
\label{prop:high-high-p22}
Under the assumptions of this subsection, the first identity in
\eqref{eq:two-colored-targets} holds.
\end{proposition}

\begin{proof}
Since the upper reflected row starts with no color--\(2\) arrows,
\eqref{eq:22-local-transition-table} forces \(\nu_1\nu_2=22\) whenever
\(\mu_1\mu_2=22\).  Both transitions leave \((h_0^{\mathrm{up}},0)\)
unchanged, so independence of \(h_0^{\mathrm{up}}\) and \(\nu\) gives
\begin{equation}
 p_{22}
 =\rho_2^2\mathbb E[q^{2h_0^{\mathrm{up}}}]\nu_{22}.
\label{eq:22-upper-p22-transition}
\end{equation}

For the lower reflected row, only the \(\pi\)-words \(22,12,21,11\)
contribute.  By \eqref{eq:22-local-transition-table}, each output \(2\)
leaves \(h_0^{\mathrm{low}}\) unchanged, while each \(1\to2\) transition
decreases the color--\(2\) coordinate by one.  The contribution from
\(11\) therefore uses \(h_2^{11}-1\) at the second site and vanishes
when \(h_2^{11}<2\).  Conditional independence at the mixed vertex gives
\begin{align}
\nu_{22}
&=\rho_1^2\mathbb E[q^{2h_0^{\mathrm{low}}}]
 \Biggl[
  \pi_{22}
  +\pi_{12}\mathbb E[1-q^{h_2^{12}}]
  +\pi_{21}\mathbb E[1-q^{h_2^{21}}]
 \notag\\*
&\hspace{37mm}
  +\pi_{11}\mathbb E\!\left[
    \mathbf1_{\{h_2^{11}\geq2\}}
    (1-q^{h_2^{11}})(1-q^{h_2^{11}-1})
  \right]
 \Biggr].
\label{eq:22-nu22-decomposition}
\end{align}
Evaluate the \(h_0\) moments in
\eqref{eq:22-upper-p22-transition}--\eqref{eq:22-nu22-decomposition}
using \eqref{eq:22-h0-laws} and \eqref{eq:22-q-Poisson-useful-moments}.
Substituting the full-space probabilities from
Theorem~\ref{thm:full-space-two-site-law} and the moments
\eqref{eq:22-h2-11-double-weight}--\eqref{eq:22-h2-12-21-weight},
and simplifying with \eqref{eq:22-cross-parameter-identities} and
\(J_1-J_2=(\rho_2-\rho_1)(\rho_1+\rho_2-1)\), gives
\[
p_{22}=\frac{J_1-J_2}{\alpha^2}
 (\rho_2\phi-q\rho_1)(\rho_1\theta-q\bar\rho_2)
=\frac{J_1-J_2}{\alpha^2}(J_1\theta^2-qJ_2),
\]
which is the first identity in \eqref{eq:two-colored-targets}.
\end{proof}

It remains to compute \(p_{20}\).

\begin{proposition}
\label{prop:high-high-p20}
Under the assumptions of this subsection, the second identity in
\eqref{eq:two-colored-targets} holds.
\end{proposition}

\begin{proof}
As in the preceding proof, the first output \(2\) requires \(\nu_1=2\)
and leaves \((h_0^{\mathrm{up}},0)\) unchanged.  Separating the cases
\(\nu_2=0\) and \(\nu_2\in\{1,2\}\), and using
\eqref{eq:22-local-transition-table}, independence, and the
\(q\)-Poisson moments \eqref{eq:22-q-Poisson-useful-moments} for
\eqref{eq:22-h0-laws}, gives
\begin{align}
p_{20}
&=\rho_2^2(1-b_{\alpha,\rho_2})
 \Bigl[
 (1-qb_{\alpha,\rho_2})\nu_{20}
 +qb_{\alpha,\rho_2}
  (\nu_{20}+\nu_{21}+\nu_{22})
 \Bigr].
\label{eq:22-upper-p20-transition}
\end{align}

For the lower reflected row, output \(20\) requires \(\pi_1\in\{1,2\}\),
while \(\pi_2\) may be any color.  The first transition leaves
\(h_0^{\mathrm{low}}\) unchanged; a \(1\to2\) transition decreases the
color--\(2\) coordinate by one, but the second-site weights for output
\(0\) do not depend on that coordinate.  Applying
\eqref{eq:22-local-transition-table} and conditional independence,
and evaluating the \(h_0^{\mathrm{low}}\) moments by
\eqref{eq:22-q-Poisson-useful-moments}, yields
\begin{align}
\nu_{20}
&=\rho_1^2(1-b_{\rho_1,\rho_2})
 \Bigl[
 \pi_{20}+\pi_{10}\mathbb E[1-q^{h_2^{10}}]
\notag\\*
&\qquad
 +qb_{\rho_1,\rho_2}
 \bigl(
   \pi_{21}+\pi_{22}
  +\pi_{11}\mathbb E[1-q^{h_2^{11}}]
  +\pi_{12}\mathbb E[1-q^{h_2^{12}}]
 \bigr)
 \Bigr].
\label{eq:22-nu20-decomposition}
\end{align}
Substituting the full-space entries from
Theorem~\ref{thm:full-space-two-site-law} and the moment formulas
\eqref{eq:22-h2-10-weight}, \eqref{eq:22-h2-11-single-weight}, and
\eqref{eq:22-h2-12-21-weight} into
\eqref{eq:22-nu20-decomposition}, and then using
\eqref{eq:22-cross-parameter-identities} and
\(\theta+\phi-q=1+(1-q)(\rho_2-\rho_1)\), gives
\begin{equation}
\label{eq:22-nu20-simplified}
\nu_{20}
=\frac{(J_1-J_2)
 \bigl[J_1\theta(1-\theta)
       +q\bar\rho_2(\alpha+\rho_2-1)\bigr]}
 {(\alpha+\rho_2-1)
  [\alpha\rho_2-q(1-\alpha)\bar\rho_2]}.
\end{equation}

For the row sum in \eqref{eq:22-upper-p20-transition}, the same
upper-row transition at the first site gives
\[
\mathbb P(\mu(1)=2)
=\rho_2\mathbb E[q^{h_0^{\mathrm{up}}}]
 \mathbb P(\nu(1)=2)
=\rho_2(1-b_{\alpha,\rho_2})\mathbb P(\nu(1)=2).
\]
The one-species projections and
Proposition~\ref{prop:one-species-left-edge} give
\(\mathbb P(\mu(1)=2)=(J_1-J_2)/\alpha\), hence
\begin{equation}
\label{eq:22-nu-second-row-sum}
    \nu_{20}+\nu_{21}+\nu_{22}
    =\frac{J_1-J_2}{\alpha+\rho_2-1}.
\end{equation}
Substituting \eqref{eq:22-nu20-simplified} and
\eqref{eq:22-nu-second-row-sum} into \eqref{eq:22-upper-p20-transition},
and using \eqref{eq:22-cross-parameter-identities} together with
\(q\bar\rho_2(\alpha+\rho_2-1)+q(1-\alpha)\bar\rho_2=qJ_2\),
proves the second identity in \eqref{eq:two-colored-targets}.
\end{proof}

Together with the low--high-density calculation in
Subsection~\ref{subsec:low-high-asep} and the color-projection reduction,
Propositions~\ref{prop:high-high-p22} and~\ref{prop:high-high-p20}
prove Theorem~\ref{thm:two-site-marginal} for \(0<\alpha<1\).

\subsection{Boundary correlations for \texorpdfstring{$\alpha\geq1$}{alpha >= 1}}
\label{subsec:two-species-large-alpha}

We now prove Theorem~\ref{thm:two-site-marginal} in regime~\textup{(H)}
for \(\alpha\geq1\).  The stationary law has already been constructed in
Theorem~\ref{thm:main-stationary-construction}: for \(\alpha>1\), it uses
the exact specialization \eqref{eq:engine-large-rate-specialization}, and
both boundary outputs are the zero hole-indexed state.  It remains to
adapt the mixed-crossing representation and compute the two probabilities
in \eqref{eq:two-colored-targets}.

We first assume \(\alpha>1\) and \(0<q<1\); the endpoints
\(\alpha=1\) and \(q=0\) are treated at the end.  Set
\begin{equation}
\label{eq:large-alpha-specialization}
 a=\frac{\alpha}{\alpha-1},\qquad x_1=x_2=a^{-1},\qquad
 \tau=-\frac{a\bar\rho_1}{\rho_1},\qquad
 \sigma=-\frac{a\bar\rho_2}{\rho_2},\qquad z=a^{-2}.
\end{equation}
Thus
\begin{equation}
\label{eq:large-alpha-parameter-identities}
\begin{gathered}
 \sigma\tau z=b_{\rho_1,\rho_2},\qquad
 \sigma/\tau=b_{\bar\rho_1,\rho_2},\\
 \tau z=b_{\alpha,\rho_1},\quad \sigma z=b_{\alpha,\rho_2},\quad
 \tau^{-1}=b_{\alpha,\bar\rho_1},\quad
 \sigma^{-1}=b_{\alpha,\bar\rho_2}.
\end{gathered}
\end{equation}

Under this specialization, Lemma~\ref{prop:empty-cap-boundary} gives
the zero boundary state, while the usual and return vertices are obtained
from Figures~\ref{fig:intro-normal-weights} and
\ref{fig:intro-return-weights} by setting
$x=a^{-1}$.  We use these vertex rules below, together with the update
$H'=H+\widehat E_i-\widehat E_j$ for the return carrier.

We next compute the mixed crossing with empty horizontal input
at \(z=a^{-2}\).

\begin{proposition}
\label{prop:large-alpha-mixed-crossing}
Fix $g,r,k\in\mathbb Z_{\geq0}$.  For inputs $G=(g,\infty)$ and
$H=(0,0)$ and outputs $G'=(g+r+k,\infty)$ and $H'=(r,k)$, the queue
specialization of the mixed vertex is
\begin{align}
&\mathbb W^{[2]}_{z,\sigma,\tau}
 \bigl((g,\infty),(0,0);(g+r+k,\infty),(r,k)\bigr)\notag\\
&\quad=\mathcal N_{z,\sigma,\tau}
 \frac{b_{\rho_1,\rho_2}^{\,r}(-b_{\alpha,\rho_1})^k
 q^{gk+\binom{k}{2}}
 (b_{\alpha,\bar\rho_1};q)_{r+k}(b_{\alpha,\rho_1};q)_g}
 {(q;q)_r(q;q)_k(a^{-2};q)_{g+r+k}},
\label{eq:large-alpha-crossing}
\end{align}
where $z=a^{-2}$ and
\begin{equation}
\label{eq:large-alpha-crossing-normalization}
 \mathcal N_{z,\sigma,\tau}
 =\frac{(a^{-2};q)_\infty(b_{\rho_1,\rho_2};q)_\infty}
 {(b_{\alpha,\rho_1};q)_\infty(b_{\alpha,\rho_2};q)_\infty}.
\end{equation}
For fixed input $g$, the weights form the joint probability distribution
of the output coordinates $(h_0',h_2')$, evaluated at $(r,k)$.  The first
coordinate has the input-independent law
\begin{equation}
\label{eq:large-alpha-crossing-hole-law}
 h_0'\sim q\text{-}\mathrm{NB}
 (b_{\alpha,\bar\rho_1},b_{\rho_1,\rho_2}).
\end{equation}
The output vectors are
$H'=(h_0',h_2')$ and $G'=(g_1',\infty)$, with
$g_1'=g+h_0'+h_2'$.
\end{proposition}
\begin{proof}
Substitute $G=(g,R)$, $H=(0,0)$, $G'=(g+r+k,R-k)$,
$H'=(r,k)$ in \eqref{eq:22-fully-fused-weight}.  Only $P_2=0$
survives, with $p=P_1$ in the notation of that formula.  After $R\to\infty$, the identity
$(\tau q^{1-p};q)_p=(-\tau)^pq^{-\binom p2}(1/\tau;q)_p$
reduces the sum to the following finite identity, with $n=g+r+k$:
\[
\sum_{p=0}^{n}
 \frac{(\tau q^{1-r-k};q)_{n-p}(1/\tau;q)_p}
 {(q;q)_{n-p}(q;q)_p}
 \frac{(z/q;q)_p}{(z;q)_p}(\tau q^{1-k})^p
=(z/q)^k(\tau q^{1-r-k};q)_{r+k}
 \frac{(\tau z;q)_g}{(z;q)_{g+r+k}}.
\]
It follows by expanding $(z/q;q)_p/(z;q)_p$ geometrically for
$|z|<q$ and applying the $q$-binomial theorem, then by rational
continuation.  Restoring the prefactors gives
\begin{equation}
\label{eq:large-alpha-crossing-spectral}
 \mathcal N_{z,\sigma,\tau}
 \frac{(\sigma\tau z)^r(-\tau z)^kq^{gk+\binom k2}
 (1/\tau;q)_{r+k}(\tau z;q)_g}
 {(q;q)_r(q;q)_k(z;q)_{g+r+k}},\qquad
 \mathcal N_{z,\sigma,\tau}
 =\frac{(z;q)_\infty(\sigma\tau z;q)_\infty}
 {(\tau z;q)_\infty(\sigma z;q)_\infty}.
\end{equation}
For its sum over $k$, use the confluent $q$-Gauss identity
\[
 \sum_{k\geq0}\frac{(u;q)_kq^{\binom k2}(-v/u)^k}
 {(v;q)_k(q;q)_k}=\frac{(v/u;q)_\infty}{(v;q)_\infty}
\]
with $u=q^r/\tau$, $v=zq^{g+r}$.  The result is
\[
 \frac{(\sigma\tau z;q)_\infty}{(\sigma z;q)_\infty}
 \frac{(1/\tau;q)_r(\sigma\tau z)^r}{(q;q)_r}
 =p_{b_{\alpha,\bar\rho_1},b_{\rho_1,\rho_2}}^{(q)}(r),
\]
proving the marginal law and the sum-to-one property.
\end{proof}

\begin{proposition}
\label{prop:large-alpha-mixed-crossing-representation}
Under \eqref{eq:large-alpha-specialization}, the two constructions in
Figure~\ref{fig:mixed-engine-limit} have the same local limiting output law.
Equivalently, Figure~\ref{fig:mixed-engine-layers}, with both boundary
outputs equal to the zero hole-indexed state, represents the stationary
two-engine law of Theorem~\ref{thm:main-stationary-construction}.
\end{proposition}
\begin{proof}
Use the same finite-volume Yang--Baxter argument as in the proof of
Proposition~\ref{prop:22-mixed-crossing-representation}, now with the exact
specialization \eqref{eq:large-alpha-specialization}.  The engine and
boundary kernels are stochastic by
Proposition~\ref{prop:density-array-convergence}.  It remains only to check
the crossing at the right endpoint and the effect of its outgoing carrier.

For input $(0,\infty),(h,0)$, sample
\[
 \mathbb P(\psi_0=j)=\begin{bmatrix}h\\j\end{bmatrix}_{\!q}
 \frac{(-\sigma z)^{h-j}q^{\binom{h-j}{2}}}{(\sigma z;q)_h},
 \qquad 0\leq j\leq h.
\]
Independently of $\psi_0$, sample the output coordinates
$(\widetilde h_0,\widetilde h_2)$ from
\eqref{eq:large-alpha-crossing-spectral} with $g=0$ and
$(z,\sigma,\tau)$ replaced by $(zq^h,\sigma,\tau q^{-h})$.
The full output vectors are
\[
 H'=(h-\psi_0+\widetilde h_0,\psi_0+\widetilde h_2),
 \qquad G'=(\widetilde h_0+\widetilde h_2,\infty).
\]
This identity follows from the terminating Heine transformation
of \eqref{eq:22-fully-fused-weight}.  Both factors are probability
distributions, so the right-end crossing is stochastic.  The
positive-recurrence argument in the proof of
Proposition~\ref{prop:qualitative-localization} gives tightness of the
incoming family $h$, and the displayed factorization then gives tightness
of its finite outgoing coordinates.  Given \(\varepsilon>0\), restrict
these coordinates to a fixed finite set with probability at least
\(1-\varepsilon\).  On this event, the erasing-block coupling following
Lemma~\ref{lem:scalar-clearing} removes the resulting right entrance
before any fixed output window.  Letting the cutoff tend to infinity and
then \(\varepsilon\downarrow0\) completes the same last step as in the
proof of Proposition~\ref{prop:22-mixed-crossing-representation}.
\end{proof}

Unlike in Subsection~\ref{subsec:high-high}, the unreflected rows under
the exact specialization \eqref{eq:large-alpha-specialization} are not in
the Martin specialization of
Proposition~\ref{prop:martin-yang-baxter-steady-state}: here
$x=a^{-1}>0$, rather than $x\to0$.  Consequently, their internal queue
lengths have $\alpha$-dependent stationary laws, and the identification
of their two-site output with the corresponding marginal of
$\pi^{\rho_1,\rho_2}$ does not follow directly from that proposition.
The only new input is the stationary queue law; the two-site
conditioning is the same calculation as in
Proposition~\ref{prop:conditional-two-site-queue}.

\begin{lemma}
\label{lem:large-alpha-conditioned-queue-laws}
Start the upper unreflected row at site $3$ with
$g_3\sim q\text{-}\mathrm{NB}
(b_{\alpha,\bar\rho_2},b_{\bar\rho_1,\rho_2})$, independently of the
incoming colors and vertex randomness at sites $2$ and $1$.  Process
site $2$ and then site $1$, and denote the output colors by $d(2),d(1)$
and the resulting queue length by $g_1$.  Then $g_1$ has the stationary
law
\begin{equation}
\label{eq:large-alpha-normal-queue-law}
 g_1\sim q\text{-}\mathrm{NB}
 (b_{\alpha,\bar\rho_2},b_{\bar\rho_1,\rho_2}).
\end{equation}
Moreover, $(d(1),d(2))$ has the same law as the two-site marginal of
$\pi^{\rho_1,\rho_2}$ in
Theorem~\ref{thm:full-space-two-site-law}.  If $g_1^{ij}$ denotes the
conditional law
\[
 g_1^{ij}\overset d=
 \bigl(g_1\mid(d(1),d(2))=(i,j)\bigr),
\]
then
\begin{equation}
\label{eq:large-alpha-conditioned-queues}
 g_1^{11}\sim q\text{-}\mathrm{NB}
 (b_{\alpha,\bar\rho_2},b_{\bar\rho_1,\rho_2}),\qquad
 g_1^{12},g_1^{21}\sim q\text{-}\mathrm{NB}
 (q^{-1}b_{\alpha,\bar\rho_2},qb_{\bar\rho_1,\rho_2}).
\end{equation}
\end{lemma}
\begin{proof}
Detailed balance for the usual vertex in
Figure~\ref{fig:intro-normal-weights} gives
\eqref{eq:large-alpha-normal-queue-law}.  For the conditional laws,
repeat the proof of Proposition~\ref{prop:conditional-two-site-queue}
with
\[
 u=b_{\alpha,\bar\rho_2},\qquad
 v=b_{\bar\rho_1,\rho_2},\qquad x=a^{-1}.
\]
The defining mass formula \eqref{eq:q-negative-binomial-law} gives
\[
 (1-q^{m+1})p_{u,v}^{(q)}(m+1)
 =v(1-uq^m)p_{u,v}^{(q)}(m),
 \qquad
 v=\frac{\rho_1\bar\rho_2}{\bar\rho_1\rho_2},quad
 uv=-\frac{\rho_1x}{\bar\rho_1}.
\]
Thus the additional $x$-terms cancel in the same site-by-site case
split.  The output probabilities are those in
\eqref{eq:22-full-space-two-site-table}; conditioning on $11$ leaves
$(u,v)$ unchanged, while conditioning on $12$ or $21$ replaces it by
$(q^{-1}u,qv)$.  This proves
\eqref{eq:large-alpha-conditioned-queues}.
\end{proof}

\begin{lemma}
\label{lem:large-alpha-crossing-moments}
Conditional on the output word $ab$, feed
$G=(g_1^{ab},\infty)$ and $H=(0,0)$ through the crossing and write its
hole-indexed output as $H'=(h_0^{\mathrm{low}},h_2^{ab})$.
For $m\geq0$, the required moments are
\begin{align}
 \mathbb E[q^{m h_0^{\mathrm{low}}}]
 &=\frac{(b_{\rho_1,\rho_2};q)_m}{(b_{\alpha,\rho_2};q)_m},
 \qquad
 \mathbb E[q^{m(h_0^{\mathrm{low}}+h_2^{11})}]
 =\frac{(b_{\rho_1,\rho_2};q)_m}{(b_{\alpha,\rho_1};q)_m},
 \label{eq:large-alpha-hole-moments}\\
 \mathbb E[q^{2h_0^{\mathrm{low}}+h_2^{11}}]
 &=\frac{(b_{\rho_1,\rho_2};q)_2
 [1-qb_{\alpha,\rho_1}-(1-q)b_{\alpha,\rho_2}]}
 {(1-b_{\alpha,\rho_1})(1-b_{\alpha,\rho_2})
 (1-qb_{\alpha,\rho_1})},
 \label{eq:large-alpha-11-mixed-moment}\\
 \mathbb E[q^{2h_0^{\mathrm{low}}+h_2^{12}}]
 =\mathbb E[q^{2h_0^{\mathrm{low}}+h_2^{21}}]
 &=\frac{(b_{\rho_1,\rho_2};q)_2}
 {(1-b_{\alpha,\rho_1})(1-b_{\alpha,\rho_2})}.
 \label{eq:large-alpha-12-mixed-moment}
\end{align}
\end{lemma}
\begin{proof}
The first identity in \eqref{eq:large-alpha-hole-moments} follows from
the input-independent marginal in
Proposition~\ref{prop:large-alpha-mixed-crossing} and
Lemma~\ref{lem:22-q-negative-binomial-moments}.

For input $(g_1,\infty)$, use
$H'=(h_0',h_2')$ for the generic crossing output, and fix a value
$\ell\in\mathbb Z_{\geq0}$ of $h_0'+h_2'$.
The finite $q$-binomial theorem sums
\eqref{eq:large-alpha-crossing-spectral} over $k$, and averaging
over $g_1$ then uses
\[
 \sum_{g\geq0}\frac{(\tau z;q)_g(q^\ell/\sigma;q)_g}
 {(q;q)_g(zq^\ell;q)_g}(\sigma/\tau)^g
 =\frac{(q^\ell/\tau;q)_\infty(\sigma z;q)_\infty}
 {(zq^\ell;q)_\infty(\sigma/\tau;q)_\infty}.
\]
This gives $h_0'+h_2'\sim
q\text{-}\mathrm{NB}(b_{\alpha,\bar\rho_2},b_{\rho_1,\rho_2})$,
proving the second identity in \eqref{eq:large-alpha-hole-moments}.  Inserting $q^{h_0'}$
in the same sums gives
\[
 \mathbb E[q^{h_0'}\mid h_0'+h_2'=\ell]
 =q^\ell\left[1+\frac{1-\sigma/\tau}{1-\sigma z}
 \frac{q^\ell-1}{\sigma q(1-q^{\ell-1}/\sigma)}\right].
\]
Multiplication by $q^\ell$ and summation with
$p_\ell=p_{1/\sigma,\sigma\tau z}^{(q)}(\ell)$ uses
$p_\ell(q^\ell-1)/(1-q^{\ell-1}/\sigma)
=-\sigma\tau z\,p_{\ell-1}$, proving
\eqref{eq:large-alpha-11-mixed-moment}.
Finally, write $\mathbb E_{g_1}$ and $\mathbb E_{g_1+1}$ for expectations
with bottom input $(g_1,\infty)$ and $(g_1+1,\infty)$, respectively, and
horizontal input $(0,0)$.  Shifting the bottom input gives
\[
 \sigma zq\,\mathbb E_{g_1+1}[q^{2h_0'+h_2'}]
 =\mathbb E_{g_1}[q^{2h_0'+h_2'}]
 -(1-qb_{\rho_1,\rho_2})\mathbb E_{g_1}[q^{h_0'+h_2'}].
\]
This follows by shifting adjacent indices in
\eqref{eq:large-alpha-crossing-spectral}.  Since $g_1^{12}$ and
$g_1^{21}$ have the second law in
\eqref{eq:large-alpha-conditioned-queues}, the defining mass formula
\eqref{eq:q-negative-binomial-law} shows that the corresponding
difference of the $g_1$ and $g_1+1$ expectations gives
\eqref{eq:large-alpha-12-mixed-moment}.
\end{proof}

These crossing moments allow us to compute \(p_{22}\) and \(p_{20}\)
for the present specialization.

\begin{proposition}
\label{prop:large-alpha-two-colored-probabilities}
For \(\alpha>1\) and \(0<q<1\), the stationary law
$\mu^{\alpha,\rho_1,\rho_2}$ of
Theorem~\ref{thm:main-stationary-construction} satisfies the two identities in
\eqref{eq:two-colored-targets}.  By the color-projection reduction,
its two-site marginal is the one stated in
Theorem~\ref{thm:two-site-marginal}.
\end{proposition}
\begin{proof}
\medskip\noindent\emph{The lower output $22$.}
Condition on the two-letter output $ab$ of the unreflected rows and apply
the return-vertex rules in Figure~\ref{fig:intro-return-weights} twice.
For example, input $00$, with carrier state $(r,k)$, produces output $22$
with weight $\rho_1^2a^{-2}q^{2r+1}(1-q^k)(1-q^{k-1})$: the first
conversion changes the carrier state to $(r+1,k-1)$ before the second
vertex is applied.  Combining the path weights with the joint
queue/output probabilities from the two-site calculation in
Lemma~\ref{lem:large-alpha-conditioned-queue-laws}, and then using the
moments in Lemma~\ref{lem:large-alpha-crossing-moments}, gives
\begin{align}
\frac{\nu_{22}}{\rho_1^2}
&=\frac{\pi_{22}}{\alpha^2}
 \mathbb E[q^{2h_0^{\mathrm{low}}}]\notag\\
&\quad +(\rho_2-\rho_1)(\rho_1+a^{-1}\bar\rho_1)
 \left(\frac{1+\theta}{\alpha}+\frac{1-q}{a}\right)
 \left(
 \mathbb E[q^{2h_0^{\mathrm{low}}}]
 -\mathbb E[q^{2h_0^{\mathrm{low}}+h_2^{12}}]
 \right)\notag\\
&\quad +\frac{(\rho_1+a^{-1}\bar\rho_1)
                    (\rho_1+qa^{-1}\bar\rho_1)}q
 \mathbb E\!\left[
 q^{2h_0^{\mathrm{low}}}(q^{h_2^{11}}-1)(q^{h_2^{11}}-q)
 \right].
\label{eq:large-alpha-nu22-collected}
\end{align}

\medskip\noindent\emph{The lower output $20$.}
The same conditioning and sequential state update, now summing the paths
with output $20$, gives the following convenient expression for the row
sum containing $\nu_{20}$:
\begin{align}
\frac{\nu_{20}+\nu_{22}}{\rho_1^2}
&=\frac{\rho_2-\rho_1}{\alpha}
 \left(
 \mathbb E[q^{h_0^{\mathrm{low}}}]
 -(\rho_1+a^{-1}\bar\rho_1)
 \mathbb E[q^{2h_0^{\mathrm{low}}+h_2^{12}}]
 \right)\notag\\
&\quad+\frac{(\rho_1+a^{-1}\bar\rho_1)
                    (\rho_1+qa^{-1}\bar\rho_1)}q
 \left(
 \mathbb E[q^{2h_0^{\mathrm{low}}+2h_2^{11}}]
 -\mathbb E[q^{2h_0^{\mathrm{low}}+h_2^{11}}]
 \right)\notag\\
&\quad+(\rho_1+a^{-1}\bar\rho_1)
 \left(
 \mathbb E[q^{h_0^{\mathrm{low}}}]
 -\mathbb E[q^{h_0^{\mathrm{low}}+h_2^{11}}]
 \right).
\label{eq:large-alpha-nu20-plus-nu22-collected}
\end{align}

Substitute \eqref{eq:large-alpha-hole-moments}--%
\eqref{eq:large-alpha-12-mixed-moment}.  For the first expression,
use $\pi_{22}=(\rho_2-\rho_1)(\rho_2-\rho_1\theta)$ and
\[
 \rho_1(b_{\alpha,\rho_2}-b_{\alpha,\rho_1})
 =\frac{\rho_2-\rho_1}{a\rho_2},\qquad
 \rho_1(qb_{\alpha,\rho_2}-b_{\alpha,\rho_1})
 =\frac{\rho_2-\rho_1\theta}{a\rho_2}.
\]
The factor multiplying
$\rho_1^2(b_{\rho_1,\rho_2};q)_2\pi_{22}/(b_{\alpha,\rho_2};q)_2$
is exactly
\[
 \alpha^{-2}+\frac1{a\rho_2}
 \left(\frac{1+\theta}{\alpha}+\frac{1-q}{a}\right)
 +\frac q{a^2\rho_2^2}=(b_{\alpha,\rho_2};q)_2.
\]
The second expression simplifies using
$\alpha^{-1}+(a\rho_2)^{-1}=1-b_{\alpha,\rho_2}$.
Consequently
\[
 \nu_{22}=\rho_1^2(b_{\rho_1,\rho_2};q)_2\pi_{22},\qquad
 \nu_{20}+\nu_{22}=\rho_1^2(1-b_{\rho_1,\rho_2})
 (\rho_2-\rho_1)(\bar\rho_1+q\rho_1b_{\rho_1,\rho_2}),
\]
or, by \eqref{eq:22-cross-parameter-identities},
\begin{equation}
\label{eq:large-alpha-nu-current-form}
 \nu_{22}=\frac{J_1-J_2}{\rho_2^2}(J_1\theta^2-qJ_2),\qquad
 \nu_{20}=\frac{J_1-J_2}{\rho_2^2}
 [J_1\theta(1-\theta)+qJ_2].
\end{equation}

\medskip\noindent\emph{The upper reflected row.}
It starts with the zero hole-indexed state.  By
Figure~\ref{fig:intro-return-weights}, each output $0$ or $2$ requires
the same input color, has weight
$\rho_2(1-a^{-1})=\rho_2/\alpha$, and leaves the carrier state equal to
zero.  Hence
\[
 p_{22}=(\rho_2/\alpha)^2\nu_{22},\qquad
 p_{20}=(\rho_2/\alpha)^2\nu_{20}.
\]
Together with \eqref{eq:large-alpha-nu-current-form}, these are precisely
the two identities in \eqref{eq:two-colored-targets}.
\end{proof}

\paragraph{Endpoints.}
At $\alpha=1$, the common specialization in
Section~\ref{sec:queue-specialization-stationarity} has zero boundary
state.  Set $a^{-1}=0$ in the row kernels; no finite value of
$a,\tau,\sigma$ is needed.  All
$b_{1,\rho_i}$ and $b_{1,\bar\rho_i}$ vanish, and the crossing law is
$\mathbf1_{\{k=0\}}p_{b_{\rho_1,\rho_2}}^{(q)}(r)$.
Thus $h_2^{ab}=0$, and
$h_0^{\mathrm{low}}\sim q\text{-}\mathrm{Poi}(b_{\rho_1,\rho_2})$.
Substitution in the preceding proof gives
\eqref{eq:two-colored-targets} at $\alpha=1$. The result extends to \(q=0\) 
as the expressions are continuous at \(q=0\), which complete the proof of 
Theorem~\ref{thm:two-site-marginal}. 

\section{Application: speeds of colored particles in half-space}
\label{sec:colored-particle-speeds}

Throughout this section, we set \(q=0\) and fix \(0<\alpha\leq1\). We first recall the necessary preliminary results and determine the limiting distribution of the signed speed of color \(1\). Assuming a two-species local-equilibrium conjecture, we then use the stationary measure constructed earlier to compute the joint limiting law of the first two particles, conditional on both particles remaining on the positive half-line.

\subsection{The colored process, one-color projections, and hydrodynamic inputs}

Let $\mathbb Z^\ast=\mathbb Z\setminus\{0\}$.  The multispecies half-space TASEP is the continuous-time Markov process $(\eta_t)_{t\geq 0}$ on signed permutations of $\mathbb Z^\ast$, satisfying $\eta_t(-x)=-\eta_t(x)$ for $x\in\mathbb Z^\ast$, with initial condition  $\eta_0(x)=x$. Here $\eta_t(x)$ denotes the color occupying site $x$ at time $t$ and $X_i(t):=\eta_t^{-1}(i)$ denotes the signed position of color $i\geq1$.  For every $x\geq 1$, a rate-one clock exchanges the colors at $x,x+1$ and $-x,-x-1$ simultaneously when $\eta_t(x)<\eta_t(x+1)$.  A rate-$\alpha$ clock exchanges the colors at $-1,1$ whenever $\eta_t(-1)<\eta_t(1)$. 

We have the following color-position symmetry property of the signed permutation.
\begin{proposition}[\cite{bufetov2020interacting}, Proposition~2.1]
\label{prop:half-space-color-position}
For every $t\geq 0$,
\[
    \eta_t\overset{d}=\eta_t^{-1}.
\]
\end{proposition}

For a color threshold $i\in\mathbb Z_{>0}$, a site $x\in\mathbb Z_{<0}$, and $t\geq0$, define the \emph{height function}
\[
    h^{\geq i}(t,x)
    :=\#\{y\in\mathbb Z^\ast : y\leq x,\ \eta_t(y)\geq i\}.
\]
Thus, $h^{\geq i}(t,x)$ counts the particles with color at least $i$ located at or to the left of $x$ at time $t$. This count is finite at every finite time.

\begin{theorem}[\cite{MR5008991}, Theorem~1.6]
\label{thm:half-space-shift-invariance}
Fix $t\geq 0$.  Let
\[
 i_1,\ldots,i_{K_1},\ i'_1,\ldots,i'_{K_2}\in\mathbb Z_{>0},
 \qquad
 x_1,\ldots,x_{K_1},\ x'_1,\ldots,x'_{K_2}\in\mathbb Z_{<0},
\]
and suppose that $1\leq i_k\leq i'_l$ and $-1\geq x_k>x'_l$ for every $k,l$.  Then
\[
 \left(\bigl(h^{\geq i_k}(t,x_k)\bigr)_{k=1}^{K_1},
       \bigl(h^{\geq i'_l}(t,x'_l)\bigr)_{l=1}^{K_2}\right)
 \overset{d}=
 \left(\bigl(h^{\geq i_k}(t,x_k)\bigr)_{k=1}^{K_1},
       \bigl(h^{\geq i'_l+1}(t,x'_l+1)\bigr)_{l=1}^{K_2}\right).
\]
\end{theorem}

For an integer $m\geq 0$, define the color projection on $\mathbb Z_{>0}$ by
\begin{equation}
\label{eq:positive-color-projection}
    \zeta_t^{[m]}(x):=\mathbf 1_{\{\eta_t(x)\leq m\}},\qquad x\geq 1.
\end{equation}

\begin{lemma}
\label{lem:one-color-projection}
For every $t\geq 0$ and $m\geq 0$,
\begin{equation}
\label{eq:positive-tail-projection}
    \mathbb P\bigl(X_1(t)\leq m\bigr)
    =\mathbb P\bigl(\zeta_t^{[m]}(1)=1\bigr).
\end{equation}
\end{lemma}

\begin{proof}
Proposition~\ref{prop:half-space-color-position} and \eqref{eq:positive-color-projection} give
\[
 \mathbb P(X_1(t)\leq m)=\mathbb P(\eta_t(1)\leq m)
 =\mathbb P(\zeta_t^{[m]}(1)=1),
\]
which proves \eqref{eq:positive-tail-projection}.
\end{proof}

We first record a consequence of Bahadoran’s hydrodynamic theorem for half-line TASEP with a step initial condition.

\begin{theorem}[Hydrodynamic limit]
\label{thm:half-line-hydrodynamic-inputs}
Let $m_N$ be nonnegative integers such that $m_N/N\to u\in[0,1]$.
For $x>0$, define the density profile as follows.
If $u\geq1-2\alpha$, set
\begin{equation}
\label{eq:packed-hydrodynamic-profile-section}
 \rho^{\alpha,u}(x)=
 \begin{cases}
  \dfrac{1+u-x}{2},&0<x\leq u+1,\\[2mm]
  0,&x>u+1.
 \end{cases}
\end{equation}
If $u<1-2\alpha$, set
\begin{equation}
\label{eq:packed-hydrodynamic-profile-low-alpha-section}
 \rho^{\alpha,u}(x)=
 \begin{cases}
  \alpha,
    &0<x<\bigl(\sqrt{1-2\alpha}-\sqrt u\bigr)^2,\\[2mm]
  \dfrac{1+u-x}{2},
    &\bigl(\sqrt{1-2\alpha}-\sqrt u\bigr)^2\leq x\leq u+1,\\[2mm]
  0,&x>u+1.
 \end{cases}
\end{equation}
Then, for every $\varphi\in C_c((0,\infty))$,
\begin{equation}
\label{eq:packed-empirical-section}
 \frac1N\sum_{j\geq1}\varphi\left(\frac jN\right)\zeta_N^{[m_N]}(j)
 \xrightarrow{\mathbb P}
 \int_0^\infty\varphi(x)\rho^{\alpha,u}(x)\,dx.
\end{equation}
\end{theorem}

\begin{proof}
We apply \cite[Theorem~2.2]{Bahadoran2012} to the
totally asymmetric simple exclusion process on
$\Omega=(0,\infty)$, which is included in the framework
of \cite[Section~2]{Bahadoran2012}.
The Bernoulli product measures of density $r$ for the
full-line process give the flux $J(r)=r(1-r)$.
To verify the boundary and initial assumptions, take
the constant reservoir density $\lambda_N\equiv\alpha$
on $\mathbb Z_{\leq0}$ and the initial density
$\rho^N(j)=\mathbf1_{\{1\leq j\leq m_N\}}$.
This boundary choice gives exactly our entrance dynamics and satisfies \cite[(12)]{Bahadoran2012}, while
$m_N/N\to u$ verifies \cite[(13)--(14)]{Bahadoran2012}.
Thus the hydrodynamic limit is the unique entropy
solution of
\[
\partial_t\rho+\partial_xJ(\rho)=0,\qquad x>0,
\]
with initial condition $\mathbf1_{(0,u)}$ and boundary
condition $\alpha$. The profile of the entropy
solution at time $t=1$ is precisely
$\rho^{\alpha,u}$, which proves
\eqref{eq:packed-empirical-section}.
\end{proof}

\begin{corollary}
\label{cor:half-line-injection-current}
Let $N_t^{\mathrm{in},[m]}$ denote the number of successful injections
into $\zeta^{[m]}$ by time $t$.
\begin{enumerate}
\item[(i)]
Let $m_T$ be nonnegative integers such that $m_T/T\to u\in[0,1]$.
Then
\begin{equation}
\label{eq:injection-shape-section}
 \frac{N_T^{\mathrm{in},[m_T]}}T
 \xrightarrow{\mathbb P}
 \begin{cases}
  \alpha(1-\alpha-u),&u<1-2\alpha,\\[1mm]
  \dfrac{(1-u)^2}{4},&\text{otherwise}.
 \end{cases}
\end{equation}
\item[(ii)]
For every $x>0$,
\begin{equation}
\label{eq:empty-moving-site-marginal-section}
 \lim_{T\to\infty}\mathbb E\zeta_T^{[0]}(\lceil xT\rceil)
 =\rho^{\alpha,0}(x).
\end{equation}
\end{enumerate}
\end{corollary}

\begin{proof}
It suffices to prove the assertions for $T=N\in\mathbb N$.

For (i), we first extend
Theorem~\ref{thm:half-line-hydrodynamic-inputs} to the total mass.
Since occupations are bounded by one,
\[
 \frac1N\sum_{1\leq j\leq\varepsilon N}
 \zeta_N^{[m_N]}(j)\leq\varepsilon.
\]
Moreover, the rightmost occupied site at time $N$ is stochastically
bounded by $m_N+1+\operatorname{Pois}(N)$. Thus, for every $R>u+1$,
\[
 \mathbb P\bigl(\zeta_N^{[m_N]}(j)=1
       \text{ for some }j>RN\bigr)\longrightarrow0.
\]
Applying the hydrodynamic limit to continuous, compactly supported
cutoffs equal to one on $[\varepsilon,R]$, and then letting
$\varepsilon\downarrow0$, gives
\[
 \frac1N\sum_{j\geq1}\zeta_N^{[m_N]}(j)
 \xrightarrow{\mathbb P}
 \int_0^\infty\rho^{\alpha,u}(x)\,dx.
\]
Using the fact that $\rho^{\alpha,u}$ has compact support and
\[
 N_N^{\mathrm{in},[m_N]}
 =\sum_{j\geq1}\zeta_N^{[m_N]}(j)-m_N,
\]
and subtracting $m_N/N$ and integrating the profiles in
\eqref{eq:packed-hydrodynamic-profile-section} and
\eqref{eq:packed-hydrodynamic-profile-low-alpha-section}
proves (i).

For (ii), \eqref{eq:negative-speed-exact}, derived below in the proof
of Theorem~\ref{thm:first-color-signed-speed}, shows that
$n\mapsto\mathbb E\zeta_t^{[0]}(n)$ is nonincreasing.
Let $n_N$ be positive integers such that $n_N/N\to x>0$.
For $0<\varepsilon<x$, this monotonicity bounds
$\mathbb E\zeta_N^{[0]}(n_N)$ between the mean occupation
probabilities on the blocks of $\lfloor\varepsilon N\rfloor$
sites immediately to its right and left.
The hydrodynamic limit and boundedness therefore give
\[
 \frac1\varepsilon
 \int_x^{x+\varepsilon}\rho^{\alpha,0}(y)\,dy
 \leq\liminf_{N\to\infty}\mathbb E\zeta_N^{[0]}(n_N)
 \leq\limsup_{N\to\infty}\mathbb E\zeta_N^{[0]}(n_N)
 \leq\frac1\varepsilon
 \int_{x-\varepsilon}^{x}\rho^{\alpha,0}(y)\,dy.
\]
Since $\rho^{\alpha,0}$ is continuous on $(0,\infty)$, letting
$\varepsilon\downarrow0$ and taking $n_N=\lceil xN\rceil$
proves (ii).
\end{proof}

\subsection{Boundary local equilibrium for step initial contion}
\label{subsec:packed-boundary-local-equilibrium}
Throughout this subsection, we consider only the particle-hole TASEP on $\mathbb Z_{>0}$. Let $S_t^\alpha$ denote its Markov semigroup, acting on bounded measurable functions by
\[
(S_t^\alpha f)(\xi)=\mathbb E[f(\zeta_t^\xi)],
\]
where $\zeta_t^\xi$ is the process started from $\xi\in\{0,1\}^{\mathbb Z_{>0}}$. For a probability measure $\nu$ and a bounded measurable
function $f$, write $\nu(f)=\int f\,d\nu$. Define $\nu S_t^\alpha$ to be the probability measure satisfying
\[
(\nu S_t^\alpha)(f)=\int (S_t^\alpha f)(\xi)\,\nu(d\xi).
\]
Thus $\nu S_t^\alpha$ is the law at time $t$ with initial
distribution $\nu$. For deterministic $\xi$, we write
$\xi S_t^\alpha=\delta_\xi S_t^\alpha$.

For any probability measure $\nu$, let $\nu|_{[1,k]}$ denote
its marginal on the first $k$ sites. We recall the notation $\mu^{\alpha,\rho}$ for the high density half-space stationary measures
from the preceding section. \(\pi^\rho\) is the bernoulli product measure in half-space. Liggett proved in \cite[Theorem~1.8]{Liggett1975} the following local convergence of measures:
\begin{equation}
\label{eq:liggett-convergence-section}
    \lim_{t\to\infty} \pi^\rho S_t^\alpha\Longrightarrow\mu^{\alpha,\rho}.
\end{equation}

\begin{proposition}
\label{prop:packed-boundary-local-equilibrium}
Let $n_T$ be nonnegative integers with $n_T/T\to v$ and set
$\rho=(1+v)/2$. Assume $0<\alpha\leq1$ and $\max\{0,1-2\alpha\}<v<1$. Then, for every fixed $k\geq 1$,
\[
 \bigl(\zeta_0^{[n_T]}S_T^\alpha\bigr)|_{[1,k]}
 \Longrightarrow
 \mu^{\alpha,\rho}|_{[1,k]}.
\]
\end{proposition}

We prove the proposition after two lemmas.  The first identifies the only
step-height shape needed in the argument, and the second localizes a backward
path for a Bernoulli initial configuration.

For a one-color configuration $\xi$, anchor its height by
\[
    h_\xi(0)=0,
    \qquad h_\xi(x)-h_\xi(x-1)=1-2\xi(x).
\]
Write $N_t^{\mathrm{in}}$ for the number of successful injections in the
process $\zeta_t^\xi$.  For any initial vertical shift $C$, define its
evolving height by
\[
 h(x,t):=C+2N_t^{\mathrm{in}}+x-2\sum_{j=1}^x\zeta_t^\xi(j),
 \qquad x\in\mathbb Z_{\geq0}.
\]
Thus $h(x,0)=h_\xi(x)+C$ and $h(0,t)=C+2N_t^{\mathrm{in}}$;
additive constants are preserved by the evolution.  At a restart time $s$,
we use the actual profile $h(\cdot,s)$, without resetting its boundary
height.  All height processes below are under basic coupling, driven by
the same boundary and bulk clocks.

For $0\leq s\leq t$ and $y\in\mathbb Z_{\geq0}$, let
$h^{\mathrm{step}}_{y,s}(\,\cdot\,,r)$, $r\geq s$, be the half-line TASEP
height function with boundary rate $\alpha$, started at time $s$ from the
step profile centered at $y$:
\[
    h^{\mathrm{step}}_{y,s}(z,s)=|z-y|,
    \qquad z\in\mathbb Z_{\geq0}.
\]
This is the packed configuration on $[1,y]$, with initial boundary
height $y$ rather than zero.  The dependence on $\alpha$ is suppressed
from the notation.  Under the basic coupling, the variational formula gives, for every half-line TASEP height function $h$,
every $0\leq s\leq t$, and every $x\geq0$,
\begin{equation}
\label{eq:half-line-envelope-section}
    h(x,t)=\min_{y\geq0}
    \{h(y,s)+h^{\mathrm{step}}_{y,s}(x,t)\}.
\end{equation}
This is the standard step-height form of the variational coupling.  Under the
present height normalization, it is equivalent to the half-space
Poisson-avoiding-metric formulation in \cite[(4)--(5)]{DauvergneZhang2026}:
the variational kernel between $(y,s)$ and $(x,t)$ there is
$h^{\mathrm{step}}_{y,s}(x,t)$.

We also recall the corresponding backward-path interpretation.  Starting
from $(x,t)$, trace the clock rings backward.  When the path is at $z$ and
encounters a ring of the clock associated with the height at $z$, it stays at
$z$ if the height was raised at that ring; otherwise it moves to a neighboring
site $w\in\mathbb Z_{\geq0}$ with
\[
    h(w,r)=h(z,r)-1.
\]
At the boundary this neighbor is $w=1$.  Whenever two neighbors are
admissible, choose the larger one.  Between clock rings the path stays put.
This rule defines the canonical rightmost backward path.  For the height
process started from $h_\xi$, denote it by $B^{\xi;x,t}$.  Then
\begin{equation}
\label{eq:backward-path-identity-section}
 h(x,t)=h(B^{\xi;x,t}(r),r)
       +h^{\mathrm{step}}_{B^{\xi;x,t}(r),r}(x,t),
 \qquad 0\leq r\leq t.
\end{equation}
In particular, for each $s\in[0,t]$, the point $B^{\xi;x,t}(s)$ is a
minimizer in \eqref{eq:half-line-envelope-section} with starting time $s$;
its time-zero endpoint is the corresponding minimizer when $s=0$.
This is the classic backward-path
form of the variational coupling; under the present height normalization it is
equivalent to the formulation of \cite{DauvergneZhang2026} recalled above.

We shall use the elementary bounds
\begin{align}
 h^{\mathrm{step}}_{y,s}(x,t)&\geq |y-x|,
 \label{eq:step-height-lower-section}\\
 |h^{\mathrm{step}}_{y,s}(x,t)-h^{\mathrm{step}}_{z,s}(x,t)|&\leq |y-z|,
 \label{eq:step-height-lip-start-section}\\
 |h^{\mathrm{step}}_{y,s}(x,t)-h^{\mathrm{step}}_{y,s}(z,t)|&\leq |x-z|.
 \label{eq:step-height-lip-terminal-section}
\end{align}
The first follows because a height evolves only by raising local minima.  For
the second, the two initial wedges differ uniformly by at most $|y-z|$, and
the coupled height evolution is order preserving and commutes with addition
of constants.  The last bound follows because every height profile has slopes
in $\{-1,1\}$.

For $0<\alpha\leq1$ and $u\in[0,1]$, define
\begin{equation}
\label{eq:step-shape-casewise-section}
\gamma_\alpha(u)
=\frac{1+u^2-(1-2\alpha-u)_+^2}{2}.
\end{equation}

\begin{lemma}
\label{lem:step-height-shape-section}
Let $m_T$ be nonnegative integers such that $m_T/T\to u\in[0,1]$.  Then,
for every fixed integer $k\geq0$,
\begin{equation}
\label{eq:pointwise-step-shape-section}
 \max_{0\leq x\leq k}
 \left|\frac1T h^{\mathrm{step}}_{m_T,0}(x,T)-\gamma_\alpha(u)\right|
 \xrightarrow{\mathbb P}0.
\end{equation}
\end{lemma}

\begin{proof}
The step profile centered at $n$ is the packed configuration on $[1,n]$,
with boundary height $h^{\mathrm{step}}_{n,0}(0,0)=n$.  Consequently,
\begin{equation}
\label{eq:step-height-current-section}
    h^{\mathrm{step}}_{n,0}(0,T)=n+2N_T^{\mathrm{in},[n]}.
\end{equation}
Corollary~\ref{cor:half-line-injection-current}\textup{(i)} and the definition of $\gamma_\alpha$ therefore give
\[
 \frac1T h^{\mathrm{step}}_{m_T,0}(0,T)
 \xrightarrow{\mathbb P}\gamma_\alpha(u).
\]
The bound \eqref{eq:step-height-lip-terminal-section} makes the difference between the values at $0$ and at any $x\in\{0,\ldots,k\}$ at most $k$, which proves \eqref{eq:pointwise-step-shape-section}.
\end{proof}

We denote $B^{\xi^\rho;x,T}$ by $B^{\rho;x,T}$ in short.

\begin{lemma}
\label{lem:backward-characteristic-section}
Assume $0<\alpha\leq1$ and $\max\{\frac12,1-\alpha\}<\rho<1$.
Then, for every fixed $x\in\mathbb Z_{\geq0}$,
\begin{equation}
\label{eq:backward-characteristic-section}
 \frac{B^{\rho;x,T}(0)}T
 \xrightarrow{\mathbb P}2\rho-1.
\end{equation}
\end{lemma}

\begin{proof}
Fix $x\in\mathbb Z_{\geq0}$ and write
\[
F_T(y):=h_{\xi^\rho}(y)+h^{\mathrm{step}}_{y,0}(x,T),
\qquad
\Phi(u):=\gamma_\alpha(u)-(2\rho-1)u.
\]
By the strong law of large numbers,
\[
\frac{h_{\xi^\rho}(y)}y\longrightarrow 1-2\rho
\qquad\text{almost surely as }y\to\infty.
\]
Together with Lemma~\ref{lem:step-height-shape-section},
this gives, for each $u\in[0,1]$,
\[
\frac{F_T(\lfloor uT\rfloor)}T
\xrightarrow{\mathbb P}\Phi(u).
\]
We claim that this convergence is uniform on $[0,1]$.
Indeed, $F_T$ and $\Phi$ are both Lipschitz functions.
For any finite $\delta$-net $\{u_0,\ldots,u_M\}$ of $[0,1]$,
\[
\sup_{0\leq y\leq T}
\left|\frac{F_T(y)}T-\Phi\left(\frac yT\right)\right|
\leq
\max_{0\leq j\leq M}
\left|\frac{F_T(\lfloor u_jT\rfloor)}T-\Phi(u_j)\right|
+4\delta+\frac2T,
\]
where the supremum is over integer $y$.
The maximum tends to zero in probability for fixed $\delta$.
Letting $\delta\downarrow0$ proves
\begin{equation}
\label{eq:random-objective-uniform-section}
\sup_{0\leq y\leq T}
\left|\frac{F_T(y)}T-\Phi\left(\frac yT\right)\right|
\xrightarrow{\mathbb P}0.
\end{equation}

To exclude minimizers beyond $T$, we apply the preceding almost sure limit and \eqref{eq:step-height-lower-section} to derive
\[
\liminf_{T\to\infty}\inf_{y\geq T}
\frac{F_T(y)}T\geq 2(1-\rho)
\qquad\text{almost surely}.
\]
On the other hand,
\[
\frac{F_T(\lfloor(2\rho-1)T\rfloor)}T
\xrightarrow{\mathbb P}2\rho(1-\rho)
<2(1-\rho).
\]
Hence every minimizer lies in $[0,T]$ with probability
tending to one.

Finally, $\Phi$ has a unique minimizer at $2\rho-1$,
as follows from the definition of $\gamma_\alpha$.
By \eqref{eq:backward-path-identity-section},
$B^{\rho;x,T}(0)$ minimizes $F_T$.
The uniform convergence therefore imply \eqref{eq:backward-characteristic-section}.
\end{proof}

\begin{proof}[Proof of Proposition~\ref{prop:packed-boundary-local-equilibrium}]
Choose $\varepsilon>0$ so that $\rho_-:=\rho-\varepsilon$ and
$\rho_+:=\rho+\varepsilon$ both lie in $(\frac12,1)$ if
$\frac12\leq\alpha\leq1$, and in
$(1-\alpha,1)$ if $0<\alpha<\frac12$.  Then
$2\rho_--1<v<2\rho_+-1$.

\emph{Truncation.}
For a configuration $\xi$, define
$(\mathcal T_n\xi)(j)=\xi(j)\mathbf 1_{\{j\leq n\}}$, so that
\[
 h_{\mathcal T_n\xi}(y)-h_\xi(y)=
 \begin{cases}0,&y\leq n,\\2\sum_{j=n+1}^y\xi(j),&y>n,\end{cases}
\]
and the added term is nonnegative.  Suppose that
$B^{\xi;x,T}(0)\leq n$.  By \eqref{eq:backward-path-identity-section}, this
endpoint attains the minimum in \eqref{eq:half-line-envelope-section} with
$s=0$; since
the truncation term vanishes there and is nonnegative everywhere, truncation
does not change the height at $(x,T)$.  Consequently, if this holds for every
$x=0,\ldots,k$, then
\[
    \zeta_T^{\mathcal T_n\xi}|_{[1,k]}
    =\zeta_T^\xi|_{[1,k]},
\]
because the two height functions agree at $x=0,\ldots,k$.

Let
\[
 G_T^-:=\bigcap_{x=0}^k
 \{B^{\rho_-;x,T}(0)\leq n_T\}.
\]
Put $a_-:=2\rho_--1$ and $\eta_-:=(v-a_-)/3>0$.  Since $n_T/T\to v$,
for all sufficiently large $T$ a union bound and
Lemma~\ref{lem:backward-characteristic-section} give
\[
 \mathbb P((G_T^-)^c)
 \leq\sum_{x=0}^k
 \mathbb P\bigl(B^{\rho_-;x,T}(0)>n_T\bigr)
 \leq\sum_{x=0}^k
 \mathbb P\left(
  \frac{B^{\rho_-;x,T}(0)}T>a_-+\eta_-
 \right)
 \longrightarrow0.
\]
Moreover, pathwise,
$\mathcal T_{n_T}\xi^{\rho_-}\leq\zeta_0^{[n_T]}$.

\emph{Filling.}
Define
\[
 (\mathcal F_n\xi)(j)=
 \begin{cases}1,&j\leq n,\\ \xi(j),&j>n,\end{cases}
 \qquad
 \Lambda_n(\xi):=2\sum_{j=1}^n(1-\xi(j)).
\]
Then
\[
 h_{\mathcal F_n\xi}(y)+\Lambda_n(\xi)-h_\xi(y)=
 \begin{cases}2\sum_{j=y+1}^n(1-\xi(j)),&y<n,\\0,&y\geq n.\end{cases}
\]
Again the added term is nonnegative.  Hence, if
$B^{\xi;x,T}(0)\geq n$ for every $x=0,\ldots,k$, then
\[
    \zeta_T^{\mathcal F_n\xi}|_{[1,k]}
    =\zeta_T^\xi|_{[1,k]}.
\]
Indeed, apply the same variational argument to
$h_{\mathcal F_n\xi}+\Lambda_n(\xi)$.  Adding this constant commutes
with the height evolution and disappears upon taking adjacent height
differences.

Let
\[
 G_T^+:=\bigcap_{x=0}^k
 \{B^{\rho_+;x,T}(0)\geq n_T\}.
\]
Put $a_+:=2\rho_+-1$ and $\eta_+:=(a_+-v)/3>0$.  Since $n_T/T\to v$,
for all sufficiently large $T$ a union bound and
Lemma~\ref{lem:backward-characteristic-section} give
\[
 \mathbb P((G_T^+)^c)
 \leq\sum_{x=0}^k
 \mathbb P\bigl(B^{\rho_+;x,T}(0)<n_T\bigr)
 \leq\sum_{x=0}^k
 \mathbb P\left(
  \frac{B^{\rho_+;x,T}(0)}T<a_+-\eta_+
 \right)
 \longrightarrow0.
\]
Moreover, pathwise,
$\zeta_0^{[n_T]}\leq\mathcal F_{n_T}\xi^{\rho_+}$.

\emph{The squeeze.}
Couple all evolutions using the clocks already fixed above.  On $G_T^-$ the
truncated process agrees with $\zeta_T^{\xi^{\rho_-}}$ on $[1,k]$, and on
$G_T^+$ the filled process agrees with $\zeta_T^{\xi^{\rho_+}}$ there.
Attractiveness therefore gives, for every increasing $0\leq f\leq1$
depending only on $[1,k]$,
\begin{equation}
\label{eq:finite-time-squeeze-section}
 (\pi^{\rho_-}S_T^\alpha)(f)-o(1)
 \leq\mathbb E f(\zeta_T^{[n_T]})
 \leq(\pi^{\rho_+}S_T^\alpha)(f)+o(1).
\end{equation}
Here we used that $\xi^{\rho_\pm}$ are independent of the clocks and have
laws $\pi^{\rho_\pm}$; the two error terms are bounded by
$\mathbb P((G_T^\pm)^c)$.

Letting $T\to\infty$ and using \eqref{eq:liggett-convergence-section} yields
\[
 \mu^{\alpha,\rho_-}(f)
 \leq\liminf_{T\to\infty}\mathbb E f(\zeta_T^{[n_T]})
 \leq\limsup_{T\to\infty}\mathbb E f(\zeta_T^{[n_T]})
 \leq\mu^{\alpha,\rho_+}(f).
\]
By the matrix-product formula in Proposition~\ref{prop:half-space-mpa},
the finite-dimensional marginals of $\mu^{\alpha,\rho}$ depend continuously
on $\rho$, so letting $\varepsilon\downarrow0$ in this sandwich gives
the claimed convergence.
\end{proof}

\subsection{The signed speed of color 1}

\begin{theorem}
\label{thm:first-color-signed-speed}
As $t\to\infty$,
\[
    \frac{X_1(t)}t\Longrightarrow V_\alpha,
\]
whose absolutely continuous part has density
\begin{equation}
\label{eq:first-color-speed-density}
 f_\alpha(v)=
 \begin{cases}
 \displaystyle
  \frac12\mathbf 1_{(-1,0)}(v)
  +\frac{v}{2\alpha}\mathbf 1_{(0,1)}(v),
     &\frac12\leq\alpha\leq1,\\[3mm]
 \displaystyle
  \frac12\mathbf 1_{(-1,\,2\alpha-1)}(v)
  +\frac{v}{2\alpha}\mathbf 1_{(1-2\alpha,\,1)}(v),
     &0<\alpha<\frac12.
 \end{cases}
\end{equation}
If $\alpha>\frac12$, the limiting law also has an atom at zero of mass
$(2\alpha-1)/(4\alpha)$; this mass vanishes at $\alpha=\frac12$.  For
$0<\alpha<\frac12$, the limiting law has no atoms.
\end{theorem}

\begin{proof}
We first consider negative speeds. For $n\geq1$, signed symmetry and
Proposition~\ref{prop:half-space-color-position} give
\[
 \mathbb P(X_1(t)\leq-n)
 =\mathbb P(\eta_t(-1)\geq n)
 =\mathbb E\bigl[
    h^{\geq n}(t,-1)-h^{\geq n}(t,-2)
   \bigr].
\]
By Theorem~\ref{thm:half-space-shift-invariance}, applied to
$(n,-1)$ and $(n,-2)$, this expectation equals
\[
 \mathbb E\bigl[
    h^{\geq n}(t,-1)-h^{\geq n+1}(t,-1)
   \bigr].
\]
The difference inside the expectation is the indicator that color $n$
lies on the negative half-line. Another application of
Proposition~\ref{prop:half-space-color-position} therefore yields
\begin{equation}
\label{eq:negative-speed-exact}
 \mathbb P(X_1(t)\leq-n)
 =\mathbb P(X_n(t)<0)
 =\mathbb P(\eta_t(n)<0)
 =\mathbb E\zeta_t^{[0]}(n).
\end{equation}
Taking $n=\lceil st\rceil$ and applying
Corollary~\ref{cor:half-line-injection-current}\textup{(ii)}, we obtain
\begin{equation}
\label{eq:negative-speed-limit}
 \lim_{t\to\infty}
 \mathbb P\left(\frac{X_1(t)}t\leq-s\right)
 =\rho^{\alpha,0}(s),
 \qquad s>0.
\end{equation}

For positive speeds, let $v>0$ and put $m_t=\lfloor vt\rfloor$.
By \eqref{eq:positive-tail-projection},
\begin{equation}
\label{eq:positive-speed-exact}
 \mathbb P\left(\frac{X_1(t)}t\leq v\right)
 =\mathbb P\bigl(\zeta_t^{[m_t]}(1)=1\bigr).
\end{equation}
If $\max\{0,1-2\alpha\}<v<1$,
Proposition~\ref{prop:packed-boundary-local-equilibrium} applies with
$\rho=(1+v)/2$, and Proposition~\ref{prop:one-species-left-edge} gives
\[
 \mu^{\alpha,\rho}(\xi(1)=0)
 =\frac{1-v^2}{4\alpha}.
\]
Consequently,
\begin{equation}
\label{eq:positive-speed-limit}
 \lim_{t\to\infty}
 \mathbb P\left(\frac{X_1(t)}t\leq v\right)
 =1-\frac{1-v^2}{4\alpha},
 \qquad \max\{0,1-2\alpha\}<v<1.
\end{equation}

For $0<\alpha<\frac12$, the limits in
\eqref{eq:negative-speed-limit} and
\eqref{eq:positive-speed-limit} both approach $\alpha$
as $v\uparrow 2\alpha-1$ and $v\downarrow 1-2\alpha$,
respectively. By monotonicity,
\[
 \lim_{t\to\infty}
 \mathbb P\left(\frac{X_1(t)}t\leq v\right)
 =\alpha,\qquad 2\alpha-1\leq v\leq1-2\alpha.
\]

Differentiating the limiting distribution function gives \eqref{eq:first-color-speed-density}. Its jump at zero has size $(2\alpha-1)/(4\alpha)$ when $\alpha\geq\frac12$ and is zero otherwise. This proves the theorem.
\end{proof}

\subsection{Joint speeds of two boundary particles}
\label{subsec:joint-boundary-speeds}

For integers $0\leq m\leq n$, let
$\chi^{[m,n]}=(\chi_t^{[m,n]})_{t\geq0}$ be the two-species half-line
TASEP with priorities $1\succ2\succ0$, boundary transitions
$0\to1$ and $2\to1$ at rate $\alpha$, and initial configuration
\[
 \chi_0^{[m,n]}(x)=
 \begin{cases}
  1,&1\leq x\leq m,\\
  2,&m<x\leq n,\\
  0,&x>n.
 \end{cases}
\]

The process $\chi^{[m,n]}$ is the basic coupling of its two nested one-color
projections:
\[
 \mathbf 1_{\{\chi_t^{[m,n]}(x)=1\}}=\zeta_t^{[m]}(x),
 \qquad
 \mathbf 1_{\{\chi_t^{[m,n]}(x)\in\{1,2\}\}}=\zeta_t^{[n]}(x),
 \qquad x\geq1, t\geq0,
\]
Proposition~\ref{prop:packed-boundary-local-equilibrium} identifies the
boundary limit of each projection separately.  It is therefore natural to expect the
coupled process to converge to this measure near the boundary as well.
We state it as a conjecture.

\begin{conjecture}
\label{conj:two-species-boundary-local-equilibrium}
Let $m_T,n_T$ be nonnegative integers such that $0\leq m_T\leq n_T$,
$m_T/T\to u$, and $n_T/T\to v$, where
\[
 \max\{0,1-2\alpha\}<u<v<1.
\]
Set $\rho_1=(1+u)/2$ and $\rho_2=(1+v)/2$.
Then, for every fixed $k\geq1$,
\[
 \bigl(\chi_0^{[m_T,n_T]}S_T^\alpha\bigr)|_{[1,k]}
 \Longrightarrow
 \mu^{\alpha,\rho_1,\rho_2}|_{[1,k]}.
\]
Here $\mu^{\alpha,\rho_1,\rho_2}$ is the high--high stationary two-species
measure constructed in Subsection~\ref{subsec:high-high}.
\end{conjecture}

We now consider the joint limiting distribution of $X_1(t)/t$ and $X_2(t)/t$, and set $\mathcal S_t=\{X_1(t)>0,\ X_2(t)>0\}$.

\begin{theorem}
\label{thm:joint-boundary-speeds}
Set $\hat{\alpha}=\max\{0,1-2\alpha\}$ and assume
Conjecture~\ref{conj:two-species-boundary-local-equilibrium}.
Then, conditional on $\mathcal S_t$, the pair
$\bigl(X_1(t)/t,X_2(t)/t\bigr)$ converges weakly to a probability
measure supported on $[\hat{\alpha},1]^2$, with off-diagonal density
\[
 \frac{1}{(1-\hat{\alpha}^2)^2}
 \begin{cases}
  2uv(v^2-u^2),&\hat{\alpha}<u<v<1,\\[1mm]
  v(1+2u+3u^2-2v^2),&\hat{\alpha}<v<u<1,
 \end{cases}
\]
and diagonal line density at $(u,u)$ given by
\[
 \frac{u(1+u)^2(1-u^2)}{2(1-\hat{\alpha}^2)^2},
 \qquad \hat{\alpha}<u<1.
\]
\end{theorem}

\begin{proof}
Set $\rho_w=(1+w)/2$.
By color--position symmetry, \cite[Theorem~1.8]{Liggett1975},
and Proposition~\ref{prop:one-species-left-edge},
\[
 \mathbb P(\mathcal S_t)
 =\mathbb P\bigl(\zeta_t^{[0]}(1)=\zeta_t^{[0]}(2)=0\bigr)
 \longrightarrow
 c_\alpha:=\frac{(1-\hat{\alpha}^2)^2}{16\alpha^2}.
\]

For $\hat{\alpha}<u<v<1$, write $p_{ij}(u,v)$ for the probability
of the word $ij$ under $\mu^{\alpha,\rho_u,\rho_v}$.
Color--position symmetry and
Conjecture~\ref{conj:two-species-boundary-local-equilibrium}
identify the limiting joint distribution functions with
$p_{11}(u,v)+p_{12}(u,v)$ when $u<v$, and
$p_{11}(v,u)+p_{21}(v,u)$ when $v<u$.
Thus, the unnormalized density in the first quadrant can be computed by
\[
 f_0(u,v)=
 \begin{cases}
  \partial_u\partial_v p_{12}(u,v),
    &\hat{\alpha}<u<v<1,\\[1mm]
  \partial_u\partial_v p_{21}(v,u),
    &\hat{\alpha}<v<u<1.
 \end{cases}
\]
Likewise, the word $22$ corresponds to both positions lying in $(u,u+\varepsilon]$, so the diagonal line density is
\[
 d_0(u)=
 \lim_{\varepsilon\downarrow0}
 \frac{p_{22}(u,u+\varepsilon)}{\varepsilon},
 \qquad \hat{\alpha}<u<1.
\]

Substituting the two-site probabilities from
Theorem~\ref{thm:two-site-marginal} gives the densities in the theorem after division by the normalization constant $c_\alpha$.
\end{proof}

\appendix
\section{Conjectural full two-speed distribution}
\label{app:full-two-speed-law}

We record a conjectural expression for the unconditioned joint law of
the first two signed asymptotic speeds: We have not addressed its negative-speed sectors in this paper. See Figure \ref{fig:half-space-two-speed-simulation} for the simulation.

\begin{figure}[htbp]
	\centering
	\includegraphics[width=\textwidth]{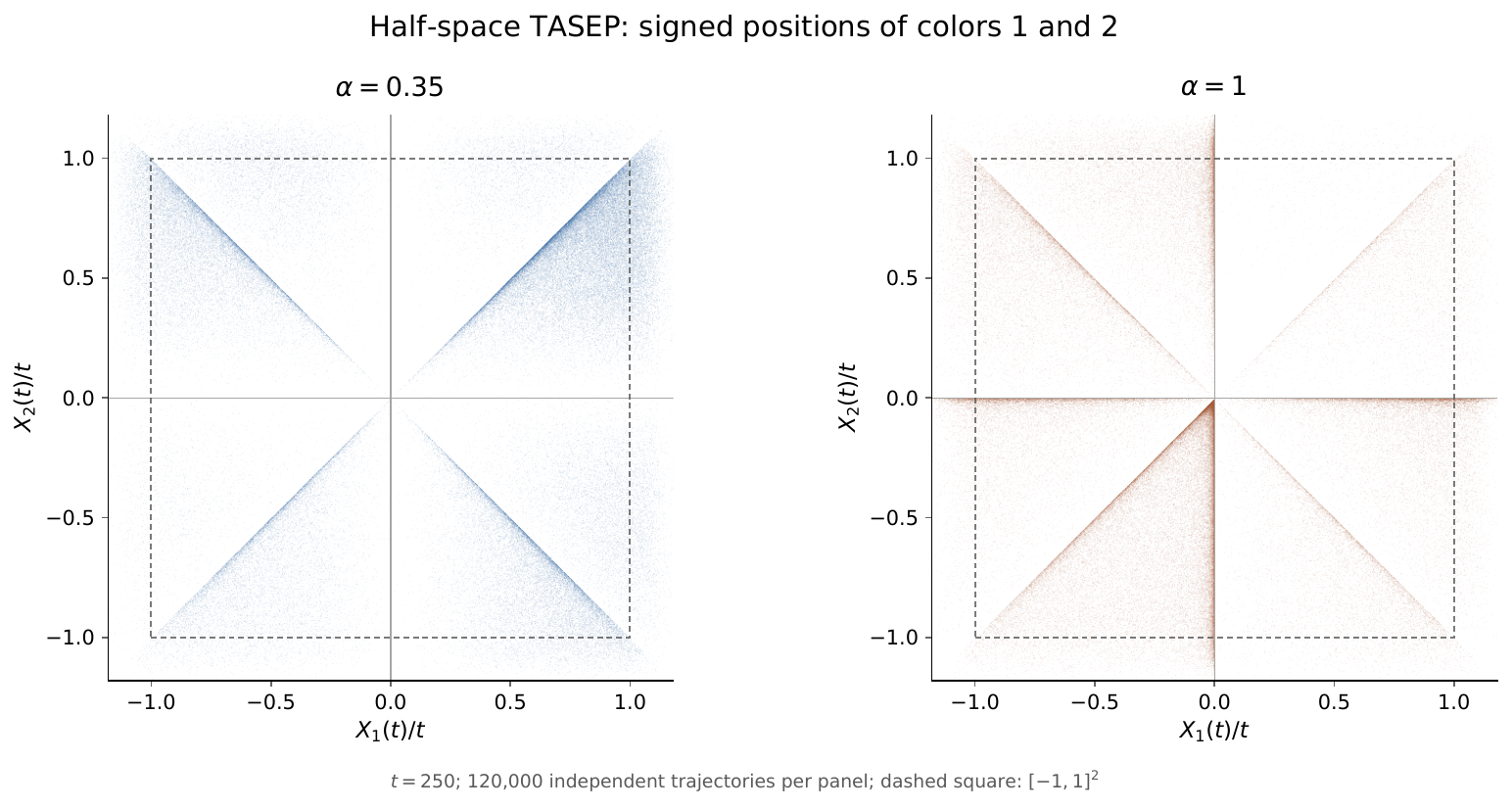}
	\caption{Simulation of
		$\bigl(X_1(t)/t,X_2(t)/t\bigr)$ in the half-space TASEP at $t=250$,
		with $120{,}000$ independent samples.} 
	\label{fig:half-space-two-speed-simulation}
\end{figure}

\begin{conjecture}
\label{conj:full-two-speed-law}
Assume Conjectures~\ref{conj:half-space-asymptotic-speeds}
and~\ref{conj:two-species-boundary-local-equilibrium}.
Set $\hat\alpha=\max\{0,1-2\alpha\}$.
The absolutely continuous part of the joint law of $(U_1,U_2)$
has density $f_\alpha^{(2)}$ described, for
$\hat\alpha<u,v<1$ with $u\ne v$, by
\begin{align*}
 f_\alpha^{(2)}(u,v)
 &=\begin{cases}
 \dfrac{uv(v^2-u^2)}{8\alpha^2},&u<v,\\[2mm]
 \dfrac{v(1+2u+3u^2-2v^2)}{16\alpha^2},&v<u,
 \end{cases}\\[2mm]
 f_\alpha^{(2)}(-u,-v)
 &=\begin{cases}
 \dfrac14,&u<v,\\[2mm]
 \dfrac{u-v}{4},&v<u,
 \end{cases}\\[2mm]
 f_\alpha^{(2)}(u,-v)
 &=\begin{cases}
 \dfrac{u}{4\alpha},&u<v,\\[2mm]
 \dfrac{(u-v)(2+3u+3v)}{16\alpha},&v<u,
 \end{cases}\\[2mm]
 f_\alpha^{(2)}(-u,v)
 &=\begin{cases}
 \dfrac{v(\alpha+u)}{4\alpha},&v<u,\\[2mm]
 \dfrac{(v-u)\bigl(2\alpha+3\alpha(u+v)
  +2u^2+5uv+2v^2-1\bigr)}{16\alpha},&u<v.
 \end{cases}
\end{align*}
The joint law also has singular components along the four diagonals,
with the following densities with respect to the parameter $ds$,
for $\hat\alpha<s<1$:
\[
\begin{array}{c|c}
\text{Location}&\text{Density}\\ \hline
(s,s)&\displaystyle\frac{s(1+s)^2(1-s^2)}{32\alpha^2}\\[2mm]
(-s,-s)&\displaystyle\frac{1-s^2}{8}\\[2mm]
(s,-s)&\displaystyle\frac{s(1-s^2)}{8\alpha}\\[2mm]
(-s,s)&\displaystyle\frac{s(\alpha+s)(1-s^2)}{8\alpha}
\end{array}
\]
If $\alpha>\tfrac12$, there are additionally singular components
along the coordinate axes, with densities with respect to $ds$,
for $0<s<1$:
\[
\begin{array}{c|c}
\text{Location}&\text{Density}\\ \hline
(s,0)&\displaystyle\frac{(2\alpha-1)s(3s+2)}{32\alpha^2}\\[2mm]
(-s,0)&\displaystyle\frac{(2\alpha-1)s}{8\alpha}\\[2mm]
(0,s)&\displaystyle\frac{(2\alpha-1)s(2s^2+3\alpha s+2\alpha-1)}{32\alpha^2}\\[2mm]
(0,-s)&\displaystyle\frac{2\alpha-1}{8\alpha}
\end{array}
\]
and an atom at the origin of mass
\[
\mathbb P(U_1=U_2=0)=\frac{2\alpha-1}{16\alpha}.
\]
For $0<\alpha\le\tfrac12$, there are no axis components or
atom at the origin. 
\end{conjecture}

\section*{}
\noindent
(\textsc{Alexey Bufetov})\\
Institute of Mathematics, Leipzig University, Germany\\
Email address: \texttt{alexey.bufetov@gmail.com}\\

\noindent
(\textsc{Yuan Tian})\\
Max Planck Institute for Mathematics in the Sciences, Germany\\
 Institute of Mathematics, Leipzig University, Germany\\
Email address: \texttt{yuan.tian@mis.mpg.de}
\end{document}